\documentclass[10.5pt]{amsart}
\usepackage{color}
\usepackage{graphicx}
\usepackage[mathcal]{euscript}
\usepackage{float}
\usepackage{amsmath,amsthm,amssymb,amsfonts,amscd,epsfig,latexsym,graphicx,textcomp}
\usepackage{tikz-cd}
\usepackage{psfrag}
\usepackage[all]{xy}
\usepackage{stackengine}
\usepackage{romannum}
\usepackage{comment}
\usepackage{hhline}
\usepackage{diagbox}
\usepackage{float}
\usepackage{caption}
\usepackage{eufrak}
\usepackage[scaled=.92]{helvet}

\restylefloat{figure}

\newtheorem{Theorem}{Theorem}
\newtheorem*{theorem*}{Main Theorem}

\usepackage{thmtools}
\usepackage{thm-restate}

\usepackage{hyperref}

\usepackage{cleveref}

\newtheorem{theorem}{Theorem}[section]
\newtheorem{lemma}[theorem]{Lemma}
\newtheorem{corollary}[theorem]{Corollary}
\newtheorem{proposition}[theorem]{Proposition}

\newtheorem{conjecture}[theorem]{Conjecture}

\newtheorem{question}[theorem]{Question}
\newtheorem{analysis}[theorem]{Analysis}

\theoremstyle{definition}
\newtheorem{definition}[theorem]{Definition} 
\newtheorem{example}[theorem]{Example} 

\theoremstyle{remark}
\newtheorem{remark}[theorem]{Remark}

\numberwithin{equation}{section}

\newcommand{\ot}{\otimes}
\newcommand{\ra}{\rightarrow}

\newcommand{\BR}{\mathbb{R}}

\newcommand{\divides}{\mid}
\newcommand{\notdivides}{\nmid}

\DeclareMathOperator{\Ima}{Im}

\newcommand{\IIDiag}{\raisebox{-0.33\height}{\includegraphics[scale=0.18]{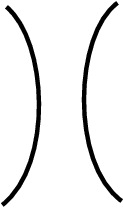}}}

\newcommand{\XDiag}{\raisebox{-0.33\height}{\includegraphics[scale=0.18]{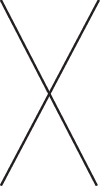}}}
\newcommand{\PMEdgeDiag}{\raisebox{-0.33\height}{\includegraphics[scale=0.18]{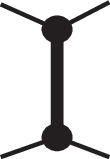}}}
\newcommand{\dumbbell}{\bigcirc\hspace{-.15cm}-\hspace{-.15cm}\bigcirc}

\newcommand{\qdim}{q\!\dim}

\newcommand{\FourRegPMEdge}{\raisebox{-0.36\height}{\includegraphics[scale=0.27]{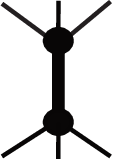}}}
\newcommand{\IIIsmoothing}{\raisebox{-0.33\height}{\includegraphics[scale=0.22]{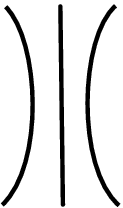}}}
\newcommand{\IXsmoothing}{\raisebox{-0.33\height}{\includegraphics[scale=0.22]{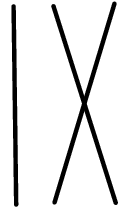}}}
\newcommand{\XIsmoothing}{\raisebox{-0.33\height}{\includegraphics[scale=0.22]{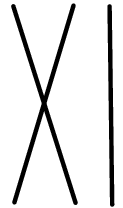}}}
\newcommand{\XXIsmoothing}{\raisebox{-0.33\height}{\includegraphics[scale=0.22]{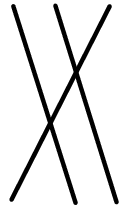}}}
\newcommand{\IXXsmoothing}{\raisebox{-0.33\height}{\includegraphics[scale=0.22]{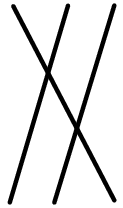}}}
\newcommand{\XXXsmoothing}{\raisebox{-0.33\height}{\includegraphics[scale=0.22]{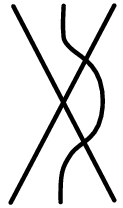}}}

\newcommand{\vertexbracketvertex}{\raisebox{-0.33\height}{\includegraphics[scale=1.2]{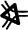}}}
\newcommand{\vertexbracketzero}{\raisebox{-0.33\height}{\includegraphics[scale=1.2]{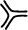}}}
\newcommand{\vertexbracketone}{\raisebox{-0.33\height}{\includegraphics[scale=1.2]{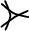}}}

\newcommand{\HorseshoeDiag}{\raisebox{-0.33\height}{\includegraphics[scale=0.18]{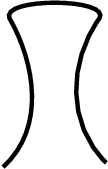}}}
\newcommand{\RibbonDiag}{\raisebox{-0.33\height}{\includegraphics[scale=0.18]{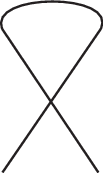}}}

\newcommand{\BZ}{\mathbb{Z}}

\newcommand{\tfp}[2]{{\langle}#1{:}#2{\rangle}_{2}}
\newcommand{\tfb}[1]{{\langle}#1{\rangle}_{2}}
\newcommand{\findex}[2]{{[}#1{:}#2{]}}
\newcommand{\fsize}[2]{{|}#1{:}#2{|}}

\newcommand{\threefp}[2]{{\langle}#1{:}#2{\rangle}_{3}}

\newcommand{\Der}[2]{\frac{\del#1}{\del#2}}

\def\del{\partial}

\def\del{\partial}

\begin{document}
\pagenumbering{arabic}

\title{ A New Cohomology Theory for Planar Trivalent Graphs with Perfect Matchings}

\author[S. Baldridge]{Scott Baldridge}

\address{Department of Mathematics, Louisiana State University \newline
\hspace*{.375in} Baton Rouge, LA 70817, USA} \email{\rm{sbaldrid@math.lsu.edu}}

\subjclass{}
\date{}

\begin{abstract}
We introduce a new cohomology theory for planar trivalent graphs with perfect matchings.  The graded Euler characteristic of the cohomology is a one variable polynomial called the $2$-factor polynomial that, if nonzero when evaluated at one, implies that the perfect matching is even and therefore the graph is $4$-face colorable.  We also define several new polynomials invariants of graphs with and without perfect matchings that are invariants of abstract tensors systems and spin networks defined by Roger Penrose in the 1970s. We show how some of these polynomials can  be ``categorified'' into their own homology theories.  
\end{abstract}

\maketitle

\section{Introduction}

The introduction of Khovanov homology \cite{K1} to knot theory provided a beautiful example of how Topological Quantum Field Theories (TQFTs) have been able to connect  fields together.  Khovanov's categorification of the  Jones polynomial quickly lead to new and surprising results in low dimensional topology, including Rasmussen's proof of the Milnor conjecture \cite{R} and Kronheimer and Mrowka's use of Khovanov homology to detect the unknot \cite{KM4}. In this paper, we begin to develop the same type of connections for graphs.

There are  recent papers that also develop TQFTs of trivalent graphs (cf. \cite{KM3, KM2, KM1, KR, RW}).  They are similar to the theory presented in this paper in the sense that they come out of a gauge- and category-theoretic perspective---a viewpoint not studied in graph theory in detail before. However, the cohomology of this paper is based upon a Kauffman-like bracket that is  similar to the Jones polynomial, and thus is more adaptable to arguments and ideas that are already found in knot theory. For example, this paper can be used to motivate a Lee-type homology theory and a Rasmussen ``$s$-invariant'' for graphs (cf. \cite{BKR}).

\subsection{Statements of the main theorems} To  investigate TQFTs of planar graphs similar to the Jones polynomial and Khovanov homology, this paper addresses two questions: (1) What polynomial invariants based upon a Kauffman-like bracket exist for planar graphs, and (2) can any of these polynomials be ``categorified'' into a cohomology theory?   These questions are answered by the two main theorems of this paper, Theorem~\ref{theorem:general_poly_invariants} and Theorem~\ref{thm:main-theorem} respectively.

Unlike the Jones polynomial in knot theory, there are plethora of polynomial invariants of graphs with perfect matchings---one for each 3-tuple of Laurent polynomials.  Some of these polynomials contain important information about graphs.  Briefly, here is the setup to describe the first  theorem: Given an abstract planar trivalent graph $G$ with a perfect matching $M$, the polynomial is defined using a special plane graph $\Gamma$ of the pair $(G,M)$ called a  perfect matching graph  (Definition \ref{def:perfect-matching-graph}). Two perfect matching graphs of the pair $(G,M)$ are related by a sequence of flip moves (Figure \ref{fig:flip-moves} in Section~\ref{section:planar_trivalent_graphs}) that are analogous to Reidemeister moves for knots. To get the polynomial, resolve the perfect matchings edges in two different ways inductively, 
replacing immersed circles  when they appear with a Laurent polynomial expression. Then: 

\begin{Theorem}
Let $G$ be a planar trivalent graph with perfect matching $M$.  If $\Gamma_M$ is a perfect matching graph for the pair $(G,M)$, generate an element  $\langle \Gamma_M \rangle \in \BZ[A,B,C]$ characterized by:
\begin{eqnarray}
\langle \PMEdgeDiag \rangle &=&  A \langle \IIDiag \rangle \ + \ B \langle \XDiag \rangle \label{eq:EMformula}\\ 
\langle \bigcirc  \rangle & = & C  \label{eq:immersed_circle}\\\
\langle \Gamma_1 \sqcup \Gamma_2 \rangle&=& \langle \Gamma_1 \rangle \cdot \langle \Gamma_2 \rangle \label{eq:disjoint_graphs_identity}
\end{eqnarray}
Then $\langle \Gamma_M \rangle$  depends only upon the pair $(G,M)$ and not the perfect matching graph used to define it.  
\label{theorem:general_poly_invariants}
\end{Theorem}

The element $\langle \Gamma_M \rangle$ is called the {\em bracket} of $\Gamma_M$. The bold edge in Equation~\ref{eq:EMformula} above represents an edge in the perfect matching set $M$, and $\bigcirc$ in Equation~\ref{eq:immersed_circle} represents any immersed circle (with no vertices). These vertexless circles are considered as trivalent graphs in this paper. 

In terms of answering the second question above, the most important  polynomial invariant of this paper is the {\em $2$-factor polynomial}.  This  polynomial, denoted $\tfp{G}{M} \in \BZ[q,q^{-1}]$, is defined using the bracket 
\begin{align*}
\langle \PMEdgeDiag \rangle &= \langle \IIDiag  \rangle  - q \langle \XDiag \rangle, \mbox{ \ and} \\ 
\langle \bigcirc \rangle &= q^{-1}+q. 
\end{align*}
One should immediately recognize the similarities the $2$-factor polynomial has to the $q$-variable form of the Kauffman bracket in knot theory. This is partially how it was discovered. The $2$-factor polynomial contains valuable combinatorial information about the graph not seen before in graph theory: it counts the number of $2$-factors, i.e., sets of $2$-cycles of the graph, that span the perfect matching edges (see Theorem~\ref{conj:f_size_conjecture}). This count, in turn, can be used to describe certain 3-edge colorings of the trivalent graph, which is important in the study of the Four Color Theorem. Because of its central role in this paper and importance to graph theory in general, we will describe some of the properties of this new polynomial in Subsections~\ref{subsec:the-2-factor-poly} and \ref{subsection:distinguishing-pm-of-a-graph}, and revisit it and its generalizations throughout the paper.

The $2$-factor polynomial can be categorified much like Khovanov homology categorifies the Jones polynomial.  This answers the second question above and is the content of the second theorem.  To describe this theorem, first let $\Gamma$ be a perfect matching graph of the pair $(G,M)$.  Like  the roles that  crossings play in the construction of Khovanov homology of a knot, the perfect matchings edges are resolved into two different types of smoothings to create a hypercube of smoothings.  In this hypercube,  states (vertices) are replaced with $\BZ_2$-coefficient vector spaces and edges are replaced with maps to form a bigraded differential chain complex $C^{*,*}(\Gamma)$  (see Section~\ref{section:2-factor-cohomology}).  The homology of this chain complex, $H^{*,*}(\Gamma)$, is invariant of the flip moves (see Section~\ref{section:invariant-under-0-and-1-flip-moves} and Section~\ref{section:cohomology-2-flip-move-invariance}), giving:  

\begin{Theorem} Let $G$ be a planar trivalent graph and $M$ a perfect matching of $G$.  The bigraded space, $H^{*,*}(\Gamma)$, of a perfect matching graph $\Gamma$ depends only on the pair $(G,M)$ and is denoted $H^{*,*}(G,M)$.  Furthermore, the graded Euler characteristic of $H^{*,*}(G,M)$ is the $2$-factor polynomial:
$$\tfp{G}{M}(q) = \sum_{i\in \BZ} (-1)^{i}\qdim(H^{i,*}(G,M)).$$
 \label{thm:main-theorem}
\end{Theorem}

We  prove more than Theorem~\ref{thm:main-theorem} in this paper: In Section~\ref{section:invariant-under-0-and-1-flip-moves}, we show that the differential chain complex $C^{*,*}(\Gamma)$  defined from  perfect matching graph $\Gamma$ is itself  invariant under the flip moves (see Theorem~\ref{theorem:main-theorem-of-cohomology}).  This invariance allows for the generalization of this cohomology to other related polynomials (cf. Section~\ref{section:extending-to-graphs}). 

The difficult step in the proof of Theorem~\ref{thm:main-theorem} is showing that the chain complex is  invariant under the $2$-flip move.  In fact, almost the entire proof is dedicated to this one case. Due to the similarities to Khovanov homology, one might think that a proof of the 2-flip move should already be found in knot theory (see Theorem~\ref{theorem:add-perfect-matching-edge} for an example of why this thought is reasonable). The reason why it does not exist in knot theory is thoroughly explored in \cite{BKR}, a paper that came after the first version of this one.  One of the motivations of that paper was to investigate the possibility of such a proof.  In that paper, we show that the corresponding theorem in knot theory would be about ``mutations'' of virtual links diagrams with virtual crossings, i.e., the $2$-flip move can thought of as a ``mutation''  of a virtual tangle that {\em cannot} be embedded in a $3$-ball. Such an embedding is, of course, a necessary condition to perform a mutation.  Therefore, a proof of the $2$-flip move using ideas from knot theory occurs in exactly the situation when the theorem for it does not make sense to even write down (cf. Section 5.2 of \cite{BKR}). Thus, Theorem~\ref{thm:main-theorem} and its proof are  new in the literature: The cohomology of this paper is fundamentally about graphs, not knots. 

Even though it appears to be a theorem of graphs, Theorem~\ref{thm:main-theorem} is subtly related to topology---a connected ribbon graph can be thought of as a $1$-dimensional cell complex for a complex surface (planar graphs are ribbon graphs for spheres).  Future research will explore how this TQFT-like theory applies to problems in the theory of moduli spaces of stable curves with $n$-marked points.
 
\subsection{Other polynomials and further consequences of the main theorems}
Theorem~\ref{theorem:general_poly_invariants}  and Theorem~\ref{thm:main-theorem}  both play important roles in investigating ribbon graphs with perfect matchings through the lens of TQFTs:  The second theorem, thought of more broadly, is an existence result, i.e., there exists at least one Khovanov-like cohomology theory based upon TQFT that contains useful graph theoretic information.  The first theorem, on the other hand, provides the fertile ground for other TQFT-like homology theories yet to come.  

Theorem~\ref{theorem:general_poly_invariants} has rich soil to cultivate.  While the $2$-factor polynomial is a new invariant of the pair $(G,M)$, it was not the first attempt to develop invariants for graphs based upon brackets in the history of graph theory.  It discovery is directly due to the author studying Roger Penrose's 1971 paper on abstract tensor systems \cite{Penrose}. In that paper, Penrose describes several {\em number-only} Kauffman-like brackets of planar trivalent graphs.  For example, the {\em Penrose Formula}, $[G]$, defined by $\langle \PMEdgeDiag \rangle = \langle \IIDiag  \rangle  -  \langle \XDiag \rangle$ and $\langle \bigcirc \rangle = 3$, is a number that calculates the 3-edge colorings of a trivalent planar graph $G$ (cf.~\cite{Kauffman} for a recent paper on the formula).  For fifty years, mathematicians have been trying to turn the Penrose Formula into a polynomial invariant without success (other than replacing $3$ with a variable). One of the insights of this paper was to see that a flip-move invariant polynomial could be defined if a perfect matching was included.  As explained in Section~\ref{section:extending-to-graphs}, the  choice of a perfect matching is not a restriction for defining invariants of the graph itself: one can ``blow-up'' the graph at each vertex to get a canonically-defined perfect matching graph and thus define an invariant of the graph itself. Thus the main theorems give invariants of graphs with perfect matchings and invariants of graphs by themselves, both of which are useful in graph theory in different ways.

Penrose's paper is a source of inspiration and motivation for many of the other polynomials introduced in this paper. These polynomials contain valuable combinatorial information about graphs.  For example, the {\em four color polynomial}, $P_4(G,M)\in \BZ[q,q^{-1}]$, defined by  the bracket $\langle \PMEdgeDiag \rangle = \langle \IIDiag  \rangle  - q\langle \XDiag \rangle$ and $\langle \bigcirc \rangle = q^{-1}+1+q$ in Theorem~\ref{theorem:general_poly_invariants}, extends the Penrose Formula to a polynomial  (cf. Section~\ref{section:polynomial_invariants}). When this polynomial is evaluated at $q=1$, the Penrose Formula is recovered along with the number of $3$-edge colorings of the graph.   It is well known that the number of $3$-edge colorings are  useful in proving the four color theorem: If there are $3$-edge colorings for all bridgeless planar trivalent graphs, then the four color theorem is true.   

Blowing-up of a  graph at each of its vertices produces an associated graph with a canonically-defined perfect matching (see Definition~\ref{def:blow-up-of-a-graph}). Theorem~\ref{theorem:general_poly_invariants} continues to apply to flip-moves on these graphs, and hence each of the polynomials defined above (including all polynomials defined in Theorem~\ref{theorem:general_poly_invariants}) has a corresponding polynomial that only depends on the original graph.  In the case of the four-color polynomial, the polynomial, $P_4(G)$, when evaluated at $q=1$, continues to be equal to the number of $3$-edge colorings of the original graph $G$.  This is because of the one-to-one correspondence between $3$-edge colorings of $G$ and the $3$-edge colorings of the blow-up. 

There are new polynomials that can be defined from the blow-up of a graph.  We define the {\em binor polynomial} by $\langle \PMEdgeDiag \rangle = \langle \IIDiag  \rangle  - q \langle \XDiag \rangle$ and $\langle \bigcirc \rangle = -q^{-1}-q$ in Theorem~\ref{theorem:general_poly_invariants} on the blow-up of a trivalent plane graph of $G$.  This invariant polynomial of $G$  corresponds to a negative dimensional abstract tensor system defined in \cite{Penrose}.  When evaluated at one, the resulting number is a fractional multiple of the Penrose Formula, i.e., it too calculates the number of 3-edge colorings of the graph (see Proposition~\ref{prop:binor-polynomial}).

We also introduce a bracket polynomial, the {\em vertex polynomial}, $\langle G \rangle_v$,  that can be thought of as resolving the vertices of a graph instead of the perfect matching edges of the graph (see Theorem~\ref{thm:vertex-poly}). This bracket corresponds to a special abstract tensor system defined in \cite{Penrose} and is computed using a subset of the states of the hypercube of states of the blow-up of $G$ (cf. Section~\ref{section:2-factor-cohomology} for the definition of the hypercube). When evaluated at $q=1$, it is a fractional multiple of the Penrose Formula and so determines when a planar graph is $4$-face colored or not.  However, in this special case, since the differential chain complex of the hypercube is invariant under flip-moves (via Theorem~\ref{theorem:main-theorem-of-cohomology}), a {\em vertex cohomology} can also be defined whose graded Euler characteristic is the vertex polynomial. Thus, Theorem~\ref{thm:main-theorem}, while initially defined for a specific TQFT-like theory, can be used to define other cohomology theories. 

One of the exotic aspects of the differential for the vertex cohomology is that it is defined using a composition of maps that correspond to three connected edges of the hypercube (see Example~\ref{ex:vertex-cohomology}).  This is not like any type of differential found in knot theory (as of yet) on the Khovanov hypercube of states. This exotic aspect to the differential also suggests new ways to look at states and hypercubes.
 
\subsection{Generalizations to $n$-regular planar graphs} Finally, we note that many of the polynomials in this paper can be generalized to $n$-regular planar graphs with perfect matchings with $n>3$.  In Section~\ref{subsection:generalizing-the-bracket}, the {\em $3$-factor polynomial}, $\threefp{G}{M}$, is defined for planar $4$-regular graphs with perfect matchings. It is an example of how one can generalize Theorem~\ref{theorem:general_poly_invariants} to all regular planar graphs.  The $3$-factor polynomial is a polynomial version of  Penrose's ``generalized Kronecker delta'' of ordinary tensor systems defined in \cite{Penrose}.  Similar to the $2$-factor polynomial,  when it is evaluated at 1, the $3$-factor polynomial counts the number of $4$-edge colorings of $G$ that have all perfect matching edges labeled the same color (cf. Equation~\ref{eq:3-factor-eval-at-1} and Corollary~\ref{cor:3-edge-color-of-G}).  This is the first interpretation of Penrose's generalized Kronecker delta for planar $n$-regular  graphs with $n>3$ in terms of graph theory  since his paper first appeared.

\subsection{Outline of the paper} This paper is organized as follows.  In Section~\ref{section:planar_trivalent_graphs}, a short introduction to planar trivalent graphs, flip-moves, and perfect matchings is given. Theorem~\ref{theorem:general_poly_invariants} is proved in Section~\ref{section:polynomial_invariants}. This section also includes some applications and conjectures about the $2$-factor polynomial and how to extend Theorem~\ref{theorem:general_poly_invariants}  to $n$-regular graphs in general. The cohomology of Theorem~\ref{thm:main-theorem} is defined in Section~\ref{section:2-factor-cohomology} and proven to be invariant of the flip-moves in Section~\ref{section:invariant-under-0-and-1-flip-moves}.  Examples of how the cohomology theory is a stronger invariant than the $2$-factor polynomial are given in Section~\ref{section:Examples}. Finally, in Section~\ref{section:extending-to-graphs}, we show how to extend the main theorems to get invariants of planar trivalent graphs  and use the idea of a blow-up of a graph to define other new polynomial and cohomology invariants based upon the main theorems.

\vspace{-.1cm}

\tableofcontents
\section*{Acknowledgements}
I greatly appreciate the feedback from the following people while developing this result:  Oliver Dasbach, Lou Kauffman, Ben McCarty, Bogdan Oporowski,  James Oxley, and Will Rushworth. I would also like to thank the Institute for Pure and Applied Mathematics (IPAM) for sponsoring a workshop that lead to me working on this project.  

%

\section{Planar trivalent graphs}
\label{section:planar_trivalent_graphs}

All graphs discussed in this paper are multigraphs.  A {\em multigraph} $G$ or $G(V,E)$ is the ordered pair $(V,E)$ where $V$ is a finite set of elements called vertices and $E$ is a finite set of edges, defined as follows:  $E$ is the set of ordered pairs $(V', i)$ where $V'$ is a node set containing zero, one or two elements of  $V$ and $i$ is an index from one to the number of instances of ordered pairs of E that have the same node set.  In this paper, multigraphs  include vertexless ``circles,'' i.e., where the node set for an edge is empty, and the empty graph where the vertex set and edge set are both empty.  Multigraphs can have loops.  A {\em loop} is an edge whose node set contains exactly one element.  The {\em degree} of a vertex of a graph is the number of edges the vertex belongs to, with loops counted twice.  A {\em trivalent graph} (or {\em cubic graph} or a {\em $3$-regular graph}) is a graph in which every vertex has degree three. 

This paper is concerned with {\em planar graphs}, i.e., graphs that are embeddable in the $2$-sphere. A given embedding of a planar graph, $i:G \ra S^2$, is called a \emph{plane graph}, and its image in $S^2$ is denoted by $\Gamma$. Vertices and edges of $\Gamma$ are the images of the vertices and edges of $G$, and when the graph is connected, all faces of $\Gamma$ are disks.

\subsection{Relationship between plane graphs of planar trivalent graphs}
\label{subsection:relationship-between-plane-graphs}

Whitney developed the theory of how two different plane graphs of a graph are  related (cf. \cite{W}). We follow Greene's formulation \cite{G} of Whitney, but we restrict to plane graphs of planar trivalent graphs  to simplify the moves that relate two plane graphs. 

All plane graphs of the same planar trivalent graph are related to each other by three types of reflections together with local isotopy.  To describe these moves, let $G$ be a planar trivalent graph and $\Gamma$ be a plane graph of $G$ in $S^2$.  Consider a Jordan curve $S^1 \subset S^2$ such that the curve intersects $\Gamma$ in  either 0, 1, or 2 points  in the interiors of some edge(s) of $\Gamma$. In this paper, we cut along edges (instead of vertices) to emphasize how the edges in a perfect matching of $G$ behave under the reflection moves.  


Choose either disk bounded by the Jordan curve $S^1$ in $S^2$, remove the disk, and re-glue it (with the part of the graph it contains) by an orientation-reversing homeomorphism of $S^1$ that fixes the points in the intersection set $S^1\cap \Gamma$ (if the intersection set is empty, any orientation-reversing homeomorphism will do).  This reflection preserves vertices and edges of $G$. Following Greene, we call these reflections {\em flip moves} and name them by the number of intersection points (see Figure~\ref{fig:flip-moves}).

\begin{figure}[h]
\includegraphics[scale = 1]{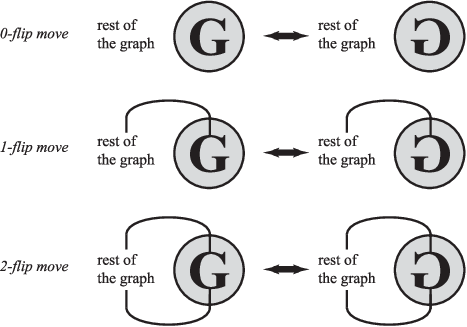}
\caption{The different types of flip moves.}
\label{fig:flip-moves}
\end{figure}

The following theorem records the relationship between plane graphs of a planar trivalent graph:

\begin{theorem}[Greene \cite{G}, Lemma 4.2, see also Mohar-Thomassen \cite{MT}, Theorem 2.6.8]
Any two plane graphs of a connected planar trivalent graph are related by a sequence of flip moves and local isotopies.
\label{graph-iso}\label{thm:graph-iso}
\end{theorem}

Note that the theorem above is stated in a different language to that of Greene and Mohar-Thomassen.  Those papers were concern with any type of planar graph. Restricting to planar trivalent graphs and using Jordan curves that intersect the graph along edges instead of vertices gives the formulation above.

\subsection{Perfect matchings}
\label{section:perfect_matchings}  A {\em perfect matching} in a graph $G(V, E)$ is a set of edges $M\subset E$ such that every vertex of the corresponding subgraph $G_M(V, M)$  has degree 1.  This definition implies that  (1) no two edges in the matching share a common endpoint, (2) a loop cannot be an edge in a matching, and (3) every vertex in $G$ is `matched' with another vertex through a connecting edge in $M$.  Clearly, $G_M$ is a spanning subgraph of $G$, in fact, it is a $1$-regular subgraph.  For this reason, perfect matchings are also called $1$-factors.

Not all planar trivalent graphs have perfect matchings, but bridgeless ones do.  Petersen showed that there are perfect matchings for all bridgeless trivalent graphs \cite{P}.  A {\em bridge} (or {\em isthmus} or {\em cut-edge}) is an edge of a graph whose deletion increases the number of connected components of the resulting graph.  An edge is a bridge if and only if it is not contained in any cycle.  A bridgeless trivalent graph is loop free.

There are positive results to the question of the existence of perfect matchings even in the presence of bridges.   A connected graph $G$ has a perfect matching if it contains at most two bridges (Petersen \cite{P}) or if the bridges lie on a single path of $G$ (Errera \cite{Err}).  Tutte gave necessary and sufficient conditions for the existence of a perfect matching \cite{T1}.  In fact, perfect matchings are abundant.  Lov\'{a}sz and Plummer  conjectured that, for  trivalent bridgeless graphs, the number of perfect matchings of graphs grows exponentially with the number of vertices \cite{LP}.  Esperet, Kardo\v{s}, King, Kr\'{a}l', and Norine proved this conjecture in \cite{EKKKN}.

Flip moves and isotopy of plane graphs can be used to tell when two planar trivalent graphs and a choice of perfect matching edges for each are  equivalent.  To build up to a theorem like Theorem~\ref{thm:graph-iso} for a graph with a perfect matching (Theorem~\ref{thm:perfect-matching-iso} below), we need  two lemmas and an analysis of the flip moves with respect to  perfect matchings.  The first lemma implies that, for all perfect matchings $M$ of $G$, all 1-flip moves must occur for a Jordan curve in a plane graph of $G$ that intersects a perfect matching edge in $M$.

\begin{lemma} If $G$ is a planar trivalent graph with a given perfect matching $M$, then all of the bridges of $G$ are contained within $M$.  \label{thm:bridge-is-in-perfectmatching}
\end{lemma}

\begin{proof} The graph $G' = (V, E\setminus M)$ is a collection of cycles.  If a bridge edge $e$ was not contained in $M$, then $e$ is part of a cycle in $G'$.  Thus, $e$ was part of a cycle in the original graph $G$, which contradicts the fact that bridges are not contained within cycles.
\end{proof}

A similar observation can be made for 2-flip moves of a plane graph of a planar trivalent graph $G$.  

\begin{lemma}
Let $G(V,E)$ be a  planar trivalent graph with a perfect matching $M\subset E$.  If a Jordan curve intersects transversely a plane graph of $G$ in the interiors of two distinct edges each once, then either both edges are in $M$ or neither edges are in $M$. \label{lem:both-or-neither-edge-in-M}
\end{lemma}

\begin{proof}
Let $\Gamma$ be a plane graph of $G$ and suppose that a Jordan curve $S\subset S^2$ intersects $\Gamma$ along two distinct edges $e, f \in E$, each in an interior point of each edge.  By way of contradiction, assume that $e \in M$ and $f\not\in M$.  We can modify $G$ to get a new graph $G'$ as follows: Remove edge $e$ and replace it with two new edges $e', e''$ and two loops:

\begin{center}
\includegraphics[scale = .7]{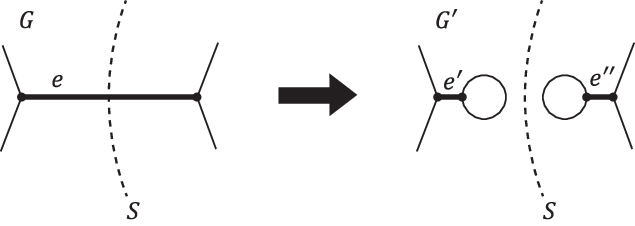} 
\label{fig:IsoPerfectMatchings}
\end{center}

We get a perfect matching $M'$ for $G'$ by removing $e$ from $M$ and including $e'$ and $e''$ in $M'$.  This new perfect matching $M'\subset E'$ is the same as before outside of the edge $e$.  In particular, $f\not\in M'$.  However, $f$ is now a bridge in $G'$.  This contradicts Lemma~\ref{thm:bridge-is-in-perfectmatching}.
\end{proof}
 
These two lemmas show that  flip moves interact with perfect matchings  in controlled ways:

\begin{analysis}
\label{analysis:flip_moves}
Let $G$ be a planar trivalent graph with a perfect matching $M$.  Let $\Gamma$ be a plane graph of $G$ and $S$ be a Jordan curve in $S^2$ that intersects $\Gamma$ in zero, one, or two points along the interiors of edges in $\Gamma$.  Let $D$ be one of the two disks in $S^2$ with boundary $S$ that has been chosen to perform a flip move.  The following are all the possible ways flip moves interact with $\Gamma$ and the perfect matching $M$:
\end{analysis}

\begin{enumerate}
\item A $0$-flip move can occur when the flipping disk $D$ contains zero, one or more entire components of $\Gamma$.  The result of the flip simply reflects these components of $\Gamma$.

\item A $1$-flip move can occur when $S$ intersects the interior of an edge $e$ of $\Gamma$ once.  In this case,  $e\in M$.

\item A $2$-flip move can occur when $S$ intersects $\Gamma$  in the following cases:

\begin{enumerate}
\item One edge case.  The curve $S$ intersects one edge $e$ of $\Gamma$ in two interior points.  The edge $e$  may or may not be in the perfect matching.  The result of the $2$-flip reflects zero, one or more components of $\Gamma$ to the adjacent face of $e$ in $\Gamma$.

\item Two  disconnected edges case. The curve $S$ intersects the interiors of two distinct edges $e,f \in \Gamma$ such that there is no path of edges in $D$ from the vertex of edge $e$ in $D$ to the vertex of edge $f$ in $D$.  Both edges in this case are bridges and $e,f\in M$.  Hence, this type of flip is equivalent to doing two separate $1$-flips.  

\item Two  connected edges case.  The curve $S$ intersects the interior of two distinct edges $e,f$ of $\Gamma$ such that there is a path of edges in $D$ from the vertex of edge $e$ in $D$ to the vertex of edge $f$ in $D$.  Either (1) both $e$ and $f$ are in $M$ or (2) neither edges are in $M$. 
\end{enumerate}

\end{enumerate}

The  above analysis establishes the different cases that will need to be addressed to prove invariance when the perfect matching set is paired with a plane graph.  In this paper, perfect matching edges will be denoted in graphs by `thickened' edges.   Thus, we can keep track of the pair $(\Gamma, M)$ in drawings of plane graphs while performing an isotopy or a flip move.

\begin{definition}
Let $(G,M)$ be a  planar trivalent graph and perfect matching pair.  A {\em perfect matching graph} (or {\em perfect matching drawing}) is a plane graph $\Gamma$ of $G$ together with the perfect matching $M$, i.e., the pair $(\Gamma, M)$.  A perfect matching graph is denoted by $\Gamma_M$ or simply $\Gamma$ when the context is clear.   \label{def:perfect-matching-graph}
\end{definition}

Theorem~\ref{thm:graph-iso} can now be extended to perfect matching graphs using Analysis~\ref{analysis:flip_moves}:

\begin{theorem}
Let $(G,M)$ be a planar trivalent graph $G$ with a perfect matching $M$.  Any two perfect matching graphs of the pair $(G,M)$ are related by a sequence of flip moves described in (1), (2), and (3c) of Analysis~\ref{analysis:flip_moves} and local isotopies.
\label{graph-iso}\label{thm:perfect-matching-iso}
\end{theorem}

\section{Polynomial invariants of perfect matchings of planar trivalent graphs}
\label{section:polynomial_invariants}

In this section we prove that the $2$-factor polynomial and four color polynomial of a planar trivalent graph $G$ with perfect matching $M$ are  invariants of the pair $(G,M)$.  This is done by proving that there is a family of such invariants of which the $2$-factor polynomial and the four-color polynomial are examples, i.e., we prove Theorem~\ref{theorem:general_poly_invariants} in the introduction:

\begin{theorem}
Let $G$ be a planar trivalent graph with perfect matching $M$.  If $\Gamma_M$ is a perfect matching graph for the pair $(G,M)$, generate an element  $\langle \Gamma_M \rangle \in \BZ[A,B,C]$ characterized by:
\begin{eqnarray}
\bigg\langle \PMEdgeDiag \bigg\rangle &=&  A \bigg\langle \IIDiag \bigg\rangle \ + \ B \bigg\langle \XDiag \bigg\rangle \label{eq:generalformula}\\ 
\bigg\langle \bigcirc  \bigg\rangle & = & C  \label{eq:immersed_circle_general}\\\
\bigg\langle \Gamma_1 \sqcup \Gamma_2 \bigg\rangle&=& \bigg\langle \Gamma_1 \bigg\rangle \cdot \bigg\langle \Gamma_2 \bigg\rangle\label{eq:disjoint_graphs_identity_general}
\end{eqnarray}
Then $\langle \Gamma_M \rangle$  depends only upon the pair $(G,M)$ and not the perfect matching graph used to define it.  
\label{theorem:general_poly_invariants_general}
\end{theorem}

With this theorem as backdrop, the following definitions can be made:

\begin{definition} Let $(G,M)$ be a planar trivalent graph $G$ with perfect matching $M$.  The {\em generalized invariant}, $\mathcal{G}(G,M)$,  is the element in $\BZ[A,B,C]$  such that  $\mathcal{G}(G,M) = \langle \Gamma_M \rangle$ for any choice of perfect matching graph $\Gamma_M$ of $(G,M)$.  
\end{definition}

This definition can be generalized if desired. If $\mathcal{A}$ is an associative algebra over a unital ring and there is an algebra map $\BZ[A,B,C] \rightarrow \mathcal{A}$, then $\mathcal{G}(G,M)$ can be thought of as an element of $\mathcal{A}$. Hence it is possible to think of this invariant in terms of matrices instead of polynomials, for example.

The important polynomial of this paper is:

\begin{definition} Let $(G,M)$ be a planar trivalent graph $G$ with perfect matching $M$.  The {\em $2$-factor polynomial}, $\tfp{G}{M}$, is the Laurent polynomial invariant given by  $\langle \PMEdgeDiag \rangle = \langle \IIDiag  \rangle  - q \langle \XDiag \rangle$ and $\langle \bigcirc \rangle = q^{-1} +q$. 
\end{definition}

For a planar trivalent graph $G$, the {\em Penrose Formula}, $[G]$, is the number found by applying the generalized bracket using $\langle \PMEdgeDiag \rangle = \langle \IIDiag  \rangle  - \langle \XDiag \rangle$ and $\langle \bigcirc \rangle = 3$ to any perfect matching graph of $G$ \cite{Kauffman2}. This number counts the number of Tait colorings of the graph $G$ \cite{Penrose}.  As mentioned in the introduction,  for many years mathematicians have looked for a way to turn the Penrose Formula  into a bracket polynomial.  In fact, Kauffman was studying the Penrose Formula when he discovered the Kauffman bracket in knot theory.  The key to getting an invariant polynomial is to choose a perfect matching.  For example,
 
\begin{definition} 
Let $(G,M)$ be a planar trivalent graph $G$ with perfect matching $M$.  The {\em four color polynomial}, $P_4(G,M)$, is an invariant of the pair $(G,M)$ given by $\langle \PMEdgeDiag \rangle = \langle \IIDiag  \rangle - q \langle \XDiag \rangle$ and $\langle \bigcirc \rangle = q^{-1} +1+q$.\label{defn:four-color-poly}
\end{definition}

By definition, $P_4(G,M)(1) = [G]$ for all perfect matchings $M$ of $G$, and therefore $4\cdot P_4(G,M)(1)$ counts the total number of 4-face colorings of $G$ (cf. \cite{Tait}). Hence its name. While beyond the scope of this paper, an interesting question is whether there is an invariant cohomology theory whose $q$-graded Euler characteristic is $P_4(G,M)(q)$.

\subsection{Proof of Theorem~\ref{theorem:general_poly_invariants_general}}

To prove  Theorem~\ref{theorem:general_poly_invariants_general}, the generalized bracket must be shown to be  invariant under flip moves on perfect matching graphs.  To set up the proof, make the following  choices and definitions. If $G$ is a planar trivalent graph, then by Euler's theorem, the number of vertices is even, say $2n$ for some $n\in \BZ$.  The number of perfect matchings of $M$ is then $n$.  Order and label these edges  by $M=\{e_1, e_2, \dots, e_n\}.$ 
If $\Gamma_M$ is a perfect matching graph for $(G,M)$, resolve each perfect matching edge $e_i$ in one of two possible ways according to the pictures shown in Equation~\ref{eq:generalformula}.  That is, replace a neighborhood of each perfect matching edge $e_i$ in $\Gamma_M \subset S^2$ with $\IIDiag$ or $\XDiag$. The resulting set of (possibly immersed) circles in $S^2$ is called a {\em state} of $\Gamma_M$.

There are $2^n$ number of states of $\Gamma_M$, each of which can be indexed by an $n$-tuple of $0$'s and $1$'s.  To describe this correspondence, call the local picture $\IIDiag$ of the state a {\em $0$-smoothing} and the local picture $\XDiag$ a {\em $1$-smoothing}.  For  $\alpha= (\alpha_1, \dots,\alpha_n)$ in $\{0,1\}^n$, let $\Gamma_\alpha$ denote the state where each perfect matching edge $e_i$ has been resolved by an $\alpha_i$-smoothing. We will often refer to the state $\Gamma_\alpha$ simply by $\alpha$ via this correspondence.

The bracket $\langle \Gamma_M \rangle$ can be expressed as follows:  Define $|\alpha|=\alpha_1+\cdots + \alpha_n$ to be the number of $1$'s in $\alpha$. Also, define a function $k$ from the set of states to the nonnegative integers that counts the number of circles in a state.  For notational convenience, let $k_\alpha$ to be the number of (immersed) circles in $\Gamma_\alpha$, i.e., $k(\Gamma_\alpha)=k_\alpha$.  Then, for a perfect matching graph $\Gamma_M$ of a nonempty planar trivalent graph with perfect matching,

\begin{equation}
\langle \Gamma_M \rangle = \sum_{\alpha\in\{0,1\}^n} A^{n-|\alpha|} B^{|\alpha|} C^{k_\alpha}, \label{eqn:general_state-sum}
\end{equation}
and $\langle \emptyset \rangle = 1$.   Note: If $G$ is a vertexless graph, i.e., a set of $k$ circles, then the perfect matching of $G$ is the empty set, and for a perfect matching graph $\Gamma_\emptyset$ of $G$, $\langle \Gamma_\emptyset \rangle = C^{k}$. This expression of the  generalized bracket is called a {\em state sum}.  

\begin{remark}
In a nonempty graph each state  has at least one immersed circle. Hence, like in the definition of the normalized Jones polynomial,  the generalized bracket can be normalized if needed by ``dividing by'' $C$, i.e., take $k_\alpha -1$ for the exponent of $C$ in the Equation~\ref{eqn:general_state-sum}. We prefer not to do this in this paper because of the important relationship between the $2$-factor polynomial and the number of Tait colorings (see Theorem~\ref{thm:existence_of_even_perfect_matching}).
\end{remark}

\begin{remark}
Since $A$, $B$, and $C$ commute with each other, Equation~\ref{eqn:general_state-sum} shows that the state sum (and hence the generalized bracket) is independent of the choice of ordering of edges in $M$.
\end{remark}

\begin{proof}[Proof of Theorem~\ref{theorem:general_poly_invariants_general}] We need to show that the generalized bracket is invariant under the different flip moves and local isotopy.  By local isotopy we mean an ambient isotopy of the perfect matching picture that moves $\Gamma$ around in $S^2$ without creating any crossings.  Hence, the generalized bracket is invariant under local isotopy. The generalized bracket is also  invariant under a $0$-flip move by Equation~\ref{eq:disjoint_graphs_identity_general}  and the fact that reflecting a connected perfect matching picture will not change the number of circles in each $\Gamma_\alpha$.  The invariance of the $1$-flip move is proven in Corollary~\ref{cor:1-flip-poly-invariant} below and the invariance of the $2$-flip move is proven in Corollary~\ref{cor:2-flip-poly-invariant}.
\end{proof}

To prove invariance under a 1-flip move (Corollary~\ref{cor:1-flip-poly-invariant} below), the following proposition is needed.  It says more than what is needed for the corollary: It also gives a condition for when a bridge exists.

\begin{proposition}
Let the pair $(G,M)$ be a planar trivalent graph $G$ and a perfect matching $M$ of $G$. Let $\Gamma_M$ be a perfect matching graph of $(G,M)$ with states $\Gamma_\alpha$ for each $\alpha \in\{0,1\}^n$.  The edge $e_i \in M$ is a bridge if and only if the following statement is true: for all $\alpha, \alpha' \in \{0,1\}^n$ such that $\alpha_k=\alpha'_k$ for all $k$ except $k=i$, and $\alpha_i=|\alpha'_i-1|$, then $k_{\alpha}=k_{\alpha'}$.\label{lemma:k_alpha_is_k_alpha_prime}
\end{proposition}

This proposition says that of $e_i$ is a bridge, then the number of circles of the state where $e_i$ has a $0$-smoothing is equal to the number of circles of the state where $e_i$ has a $1$-smoothing when all of the other smoothings are the same. A graphical way to express Proposition~\ref{lemma:k_alpha_is_k_alpha_prime} for the edge $e_i$ is as follows:

\begin{eqnarray}
e_i \ \mbox{is a bridge} \Rightarrow \
\bigg\langle \raisebox{-0.40\height}{\includegraphics[scale=0.15]{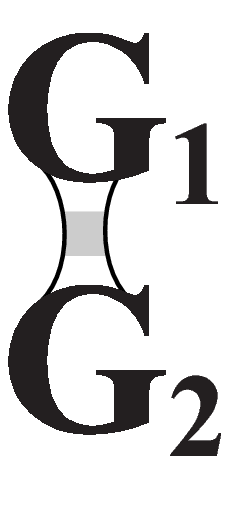}} \bigg\rangle = \bigg\langle \raisebox{-0.40\height}{\includegraphics[scale=0.15]{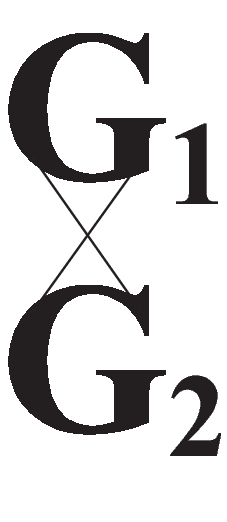}} \bigg\rangle \label{eq:PM_Edge_Identity} 
\end{eqnarray}

\begin{proof}[Proof of Proposition~\ref{lemma:k_alpha_is_k_alpha_prime}]
Assume $e_i$ is a bridge (which implies by Lemma~\ref{thm:bridge-is-in-perfectmatching} that $e_i\in M$).  Since $\alpha_k = \alpha'_k$ for all $k\not=i$, the circles of $\Gamma_{\alpha}$ that are not part of the $\alpha_i$-smoothing are exactly the same as the circles of $\Gamma_{\alpha'}$ that are not part of the $\alpha'_i$-smoothing.  Therefore we only have to inspect the circles that are part of the $\alpha_i$- and $\alpha'_i$-smoothings.  Clearly, there can be at most two circles that are part of the $\alpha_i$-smoothing (one for each arc in the $\alpha_i$-smoothing).  We claim that there is only one circle:  Since $e_i$ is a bridge between two parts of the graph, any circle that enters along one of the two arcs of the $\alpha_i$-smoothing would then have to exit along the other arc.  By the same argument, there is only one circle that is part of the $\alpha'_i$-smoothing as well.  Thus, $k_{\alpha}=k_{\alpha'}$.

Now suppose the statement is true for an edge $e_i \in M$, i.e., $k_\alpha=k_{\alpha'}$ for $\alpha$ and $\alpha'$ that differ only in that for  $e_i$,  $\alpha_i=|\alpha'_i-1|$.  By way of contradiction, suppose that there are two circles that are part of the $\alpha_i$-smoothing, $c_1$ and $c_2$.  Since $k_\alpha=k_{\alpha'}$ there are two circles $c'_1$ and $c'_2$ that are part of the $\alpha'_i$-smoothing as well. Without loss of generality, assume that $\alpha_i$ is a $0$-smoothing.  Since $c_1$ and $c_2$ are immersed  curves that only intersect themselves and each other in double points, $c_1$ must intersect $c_2$ in an even number of points.  The state $\Gamma_{\alpha'}$ is the same as $\Gamma_\alpha$ except the $\IIDiag$ at $e_i$ has been replaced with $\XDiag$.  This replacement introduces another intersection point between the two circles, and therefore the circles of $c'_1$ and $c'_2$ of $\Gamma_{\alpha'}$ intersect in an odd number of points---a contradiction.    Therefore, there can only be one circle that is part of the $\alpha_i$-smoothing, and one circle that is part of the $\alpha'_i$-smoothing.  That is, for all $\Gamma_\alpha$, any circle that enters along one of the two arcs of the $\alpha_i$-smoothing would have to exit along the other arc (and similarly for $\Gamma_{\alpha'}$).  This fact implies that $e_i$ is not an edge in any cycle of $\Gamma$.  Therefore $e_i$ is a bridge.
\end{proof}

The forward implication of Proposition~\ref{lemma:k_alpha_is_k_alpha_prime} implies:

\begin{corollary}
The generalized bracket is invariant under $1$-flip moves. \label{cor:1-flip-poly-invariant} 
\end{corollary}

The bracket $\langle \Gamma \rangle$ is also invariant of $2$-flip moves. This amounts to proving Case 3(c) of Analysis~\ref{analysis:flip_moves}.

\begin{proposition}
Let the pair $(G,M)$ be a planar trivalent graph $G$ and a perfect matching $M$ of $G$. Let $\Gamma$ be a perfect matching graph of $(G,M)$ with states $\Gamma_\alpha$ for each $\alpha \in\{0,1\}^n$.  Let $e_i, e_j \in M$ such that $i\not=j$ (two distinct edges).  Let $D \subset S^2$ be a disk whose boundary intersects $\Gamma$ transversely in the interiors of $e_i$ and $e_j$ each once.  If $\alpha, \alpha' \in \{0,1\}^n$ are two states such that $\alpha_k=\alpha'_k$ for all $k$ except $i$ or $j$, and
\begin{enumerate}
\item $\alpha_i=\alpha_j=|\alpha'_i-1|=|\alpha'_j-1|$, or\\
\item $\alpha_i=|\alpha_j-1|=|\alpha'_i-1|=\alpha'_j$,
\end{enumerate} 
then $k_\alpha=k_{\alpha'}$.
\label{lemma:k_alpha_is_k_alpha_prime-cut-edge-set-case}
\end{proposition}

This proposition can be stated graphically  as:

\begin{eqnarray}
\raisebox{-0.40\height}{\includegraphics[scale=0.15]{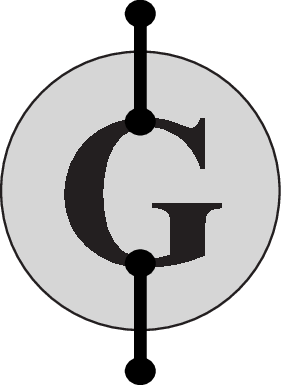}} \ \ \ \Rightarrow \ \ \
\bigg\langle \raisebox{-0.40\height}{\includegraphics[scale=0.15]{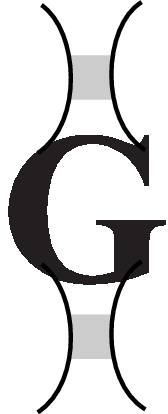}} \bigg\rangle = \bigg\langle \raisebox{-0.40\height}{\includegraphics[scale=0.15]{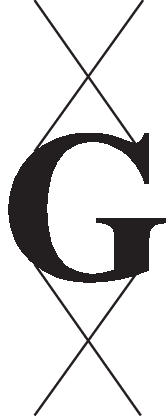}} \bigg\rangle \ \ \mbox{and} \ \ 
\bigg\langle \raisebox{-0.40\height}{\includegraphics[scale=0.15]{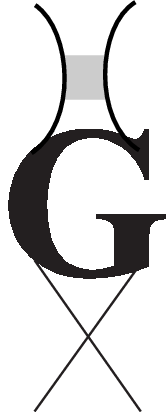}} \bigg\rangle = \bigg\langle \raisebox{-0.40\height}{\includegraphics[scale=0.15]{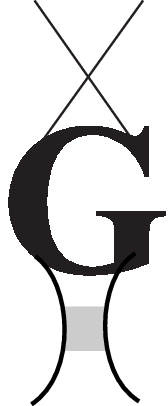}} \bigg\rangle 
\label{eq:PM_2-Edge_Identity} 
\end{eqnarray}

\begin{proof}
We address the first case when $\alpha_i=\alpha_j=|\alpha'_i-1|=|\alpha'_j-1|$, the second case is similar.  Since both $\Gamma_\alpha$ and $\Gamma_{\alpha'}$ are the same outside of the $\alpha_i$- and $\alpha_j$-smoothings, we need only show that the count of circles that enter the $\alpha_i$-smoothing and exit either the $\alpha_i$- or $\alpha_j$-smoothings  are the same as the count of circles that enter the $\alpha'_i$ smoothing and exit either the $\alpha'_i$- or $\alpha'_j$-smoothings.  Without loss of generality, we may assume that $\alpha_i=\alpha_j=0$ and $\alpha'_i=\alpha'_j=1$.  Consider the following schematic of this situation:

\begin{figure}[H]
\psfragscanon
\psfrag{a}{$a$}\psfrag{A}{$a'$}
\psfrag{b}{$b$}\psfrag{B}{$b'$}
\psfrag{c}{$c$}\psfrag{C}{$c'$}
\psfrag{d}{$d$}\psfrag{D}{$d'$}
\psfrag{g}{$\Gamma_\alpha$}
\psfrag{f}{$\Gamma_{\alpha'}$}
\includegraphics[scale = 0.25]{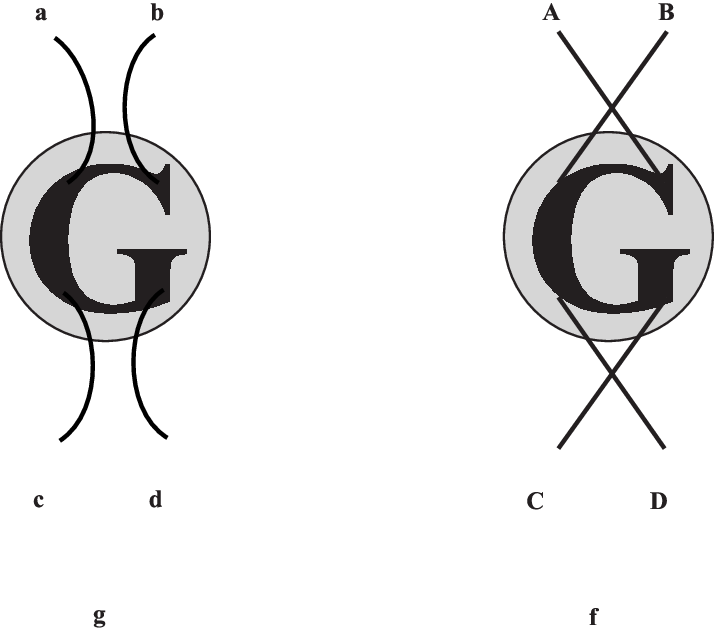}
\label{fig:2-flip-proof-step_2}
\end{figure}

\noindent In $\Gamma_\alpha$, if we travel along a circle and enter the flipping disk along arc $a$, then the first arc we transverse as we exit the flipping disk must be either $b$, $c$, or $d$. Note that it is  possible to continue traveling the circle after exiting and then reenter the flipping disk again.  For example, if we enter along arc $a$ and exit along arc $c$, then it is possible to reenter along arc $b$, but this time we must leave along arc $d$.

Step through each of the cases above separately.  If we enter along arc $a$ and the first arc we exit along is $b$, then the same will happen in $\Gamma_{\alpha'}$, that is, if we enter along arc $a'$ we will have to exit along arc $b'$.  The only difference will be that we will transverse the same part of the circle in the opposite direction while in the disk.  Note that if we enter along arc $c$ on a (possibly new) circle then we must exit along $d$, and the same thing will occur in $\Gamma_{\alpha'}$.

If we enter along arc $a$ and exit along arc $c$, then entering along arc $b$ on a (possibly different) circle would force us to exit along arc $d$ on that circle.  Looking at the diagram in $\Gamma_{\alpha'}$, if we enter along arc $a'$, we will transverse the part of the circle in the disk that took us from arc $b$ to arc $d$ in $\Gamma_\alpha$, but the results is the same as before: we will leave along arc $c'$.  Similarly, enter along arc $b'$ will mean we exit along arc $d'$.

The last case is entering along arc $a$ and exiting along arc $d$.  The same analysis as the previous case applies to this one.

Thus, if there is one circle in $\Gamma_\alpha$ that enters and exits the disk twice, then there will be one circle in $\Gamma_{\alpha'}$ that enters and exists the disk twice, and vice-versa.  If there are two circles in $\Gamma_\alpha$ that each enter and exit the disk once, then there will be two circles in $\Gamma_{\alpha'}$ that enter and exit the disk once, and vice-versa.  In either situation, $k_\alpha=k_{\alpha'}$.
\end{proof}

Equation~\ref{eq:PM_2-Edge_Identity} can be used to show:

\begin{corollary}
The generalized polynomial is invariant under $2$-flip moves.  \label{cor:2-flip-poly-invariant} 
\end{corollary}

\subsection{The 2-factor polynomial}
\label{subsec:the-2-factor-poly}

Because the $2$-factor polynomial is new and because it is an important aspect of Theorem~\ref{thm:main-theorem}, it is illuminating to see what type of combinatorial information it generates about a planar graph.  We start by computing the $2$-factor polynomial of a couple of small graphs to investigate its usefulness. For example, the $2$-factor polynomial of the theta graph $\theta$ is:

\begin{eqnarray*}
\Big\langle \raisebox{-0.43\height}{\includegraphics[scale=0.30]{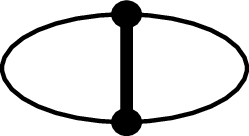}} \Big{\rangle}_{\!\!2} &=& \Big \langle \raisebox{-0.43\height}{\includegraphics[scale=0.20]{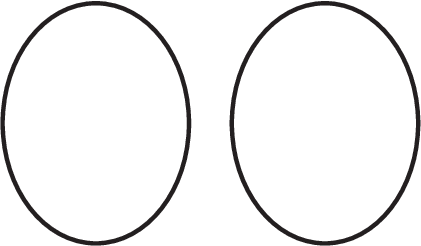}} \Big\rangle_{\!\!2} -q\Big\langle \raisebox{-0.43\height}{\includegraphics[scale=0.20]{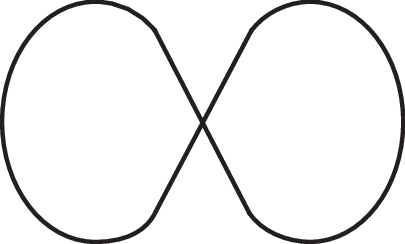}} \Big\rangle_{\!\!2}\\
&=& (q^{-1}+q)^2 - q(q^{-1}+q)\\
&=& q^{-2}+1.
\end{eqnarray*}

Evaluating the $2$-factor polynomial of $\theta$ at $1$ is $2$.  Like the unnormalized Jones polynomial, which evaluated at $1$ is $2^k$ where $k$ is the number of components of the link, something similar appears to be happening here: $G\setminus M$ has one cycle and $\tfp{\theta}{M}(1)=2^1$.  More computations suggest this may be true in general: $$\Big\langle\raisebox{-0.33\height}{\includegraphics[scale=.15]{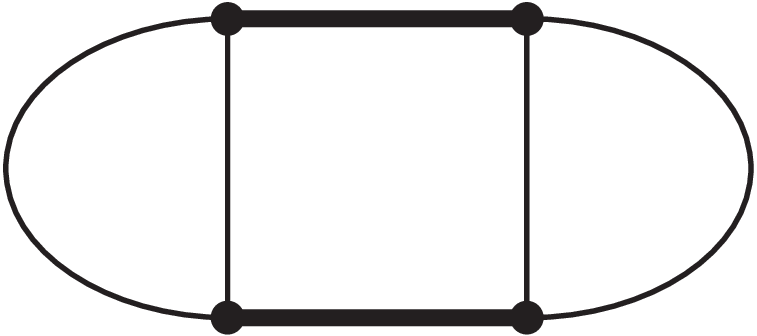}}\Big\rangle_{\!\!2}(q)=q^{-2}+1+q^2+q^4$$ and $\big\langle\raisebox{-0.3\height}{\includegraphics[scale=.1]{theta_2.eps}}\big\rangle_{\!2}(1)=2^2$. However, the $2$-factor polynomial of the dumbbell graph with its one perfect matching edge is $\tfb{{\dumbbell}}(z) =z^{-1}-1+z-z^2$. Evaluating this polynomial at $1$ is $0$. A close inspection of this phenomena in many graphs reveals that the $2$-factor polynomial evaluated at one counts  the number of $2$-factors that span the perfect matching set---hence the reason for the name of the polynomial.   This fact is not obvious nor simple to prove: 

\begin{theorem}[Baldridge--Lowrance--McCarty~\cite{BM}]  Let $G$ be a connected, planar trivalent graph with perfect matching $M$.  Then evaluating the $2$-factor polynomial at $1$ counts the number of $2$-factors that span the perfect matching set, i.e.,
$$\tfp{G}{M}(1) = \mbox{\#\{2-factors that span $M$\}}.$$ \label{conj:f_size_conjecture}
\end{theorem}

\begin{remark} This theorem was originally a conjecture in the first version of this paper. It was proven after the paper was completed but before it was published.  
\end{remark}

The $2$-factor polynomial does more; it detects the existence of a $2$-factor with only even cycles, i.e., cycles with an even number of vertices.  If we remove the perfect matching edges $M$ from the graph $G$, the graph $G(V, E\setminus M)$ is a set of disjoint cycles in the plane.  If $(G,M)$ has a $2$-factor that contains the perfect matching set $M$, then this $2$-factor must intersect each of the cycles in $G(V,E\setminus M)$ in an even number of vertices. Such a $2$-factor exists if and only if all of the cycles of $G(V,E\setminus M)$ have an even number of vertices. A perfect matching $M$ is called {\em even} if all cycles of $G(V,E\setminus M)$ are even.  Theorem~\ref{conj:f_size_conjecture} can be used to show:

\begin{theorem}[Baldridge--Lowrance--McCarty~\cite{BM}]  Let $G$ be a planar trivalent graph with perfect matching $M$.   If any of the cycles in $G(V,E\setminus M)$ have an odd number of edges, then the $2$-factor polynomial  satisfies $$\tfp{G}{M}(1)=0.$$  In particular, if $\tfp{G}{M}(1) > 0$ for the pair $(G,M)$, then $M$ is an even perfect matching of $G$.
\label{thm:existence_of_even_perfect_matching}
\end{theorem}

Even perfect matchings are  important in graph theory:  A planar trivalent graph with an even perfect matching is $3$-edge colorable, which implies that its faces are 4-colorable (cf. \cite{Kauffman}). Hence, a non-computer-aided proof of the four color theorem \cite{AH, RSST} (see also \cite{BN2, Kauffman3}) is equivalent to proving:

\begin{conjecture} Let $G$ be a connected planar trivalent graph.  If $G$ is bridgeless, then there exists a perfect matching $M$ of $G$ such that $\tfp{G}{M}(1)>0$. \label{conj:The_2_factor_poly_conj}
\end{conjecture}

These facts about 3-edge colorings and even perfect matchings of $G$ can be summarized by the following corollary of Theorem~\ref{conj:f_size_conjecture}:

\begin{corollary}
Let $G$ be a planar trivalent graph with perfect matching $M$. Choose three colors $\{red, blue, purple\}$ to color edges of $G$. Then $\tfp{G}{M}(1)$ counts the number of 3-edge colorings of $G$ where all of the perfect matching edges of $G$ are always labeled the same color, say purple. \label{cor:3-edge-color-of-G}
\end{corollary}

The $2$-factor polynomial can be used to (1) distinguish perfect matchings of a graph, (2) count $2$-factors that contain the perfect matching, (3) detect even perfect matchings, and (4) restate the four color theorem.  It is also the $q$-graded Euler characteristic of a Khovanov-like cohomology theory $H^{i,j}(G,M)$ described in the next section. This link blends the strength of homology theories typical of TQFTs together with the combinatorial strength of polynomial invariants. Before introducing the cohomology theory, we show how to use it to distinguish perfect matchings of a graph and how to generalize the $2$-factor polynomial to $n$-regular graphs. 

\subsection{Distinguishing  perfect matchings of a graph} \label{subsection:distinguishing-pm-of-a-graph} Over the past century, graph theory has been concerned with the existence, enumeration, and properties of perfect matchings of graphs (cf. \cite{LP}). For example, the enumeration of perfect matchings of trivalent graphs was recently shown to grow exponentially with the number of vertices in a graph \cite{EKKKN}. Some of these perfect matchings will be equivalent up to flip moves (eg. theta graph has three equivalent perfect matchings). By defining the $2$-factor polynomial, this paper introduces new ways of distinguishing perfect matchings of a graph up to the moves---such inequivalent perfect matchings should be seen as fundamentally different and worthy of study from a graph theory point of view. For example, some questions immediately come to mind: What does the $2$-factor polynomial or homology theory say about nowhere-zero $4$-flows where the perfect matching edges are labeled by $\pm 2$? What can be said about the invariants of this paper for planar bipartite trivalent (bicubic) graphs (cf. \cite{T2})? If two perfect matchings of a graph have isomorphic homologies, are they flip move equivalent? The purpose of this subsection is to start the exploration of questions like these.

One simple way to distinguish two perfect matchings is to count the number of cycles after removing the perfect matching edges from the graph.  For example, removing $M_1$ and $M_2$ from the $\theta_2$ graph below gives one cycle and two cycles respectively:
\begin{center}
\includegraphics[scale = .6]{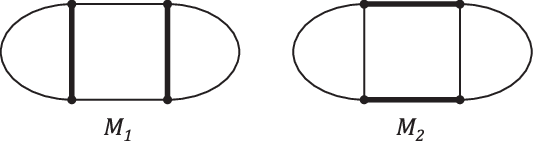} 
\label{fig:NoIsoPerfectMatchings}
\end{center}

\noindent The $2$-factor polynomial can also used to distinguish these perfect matchings: $\tfp{\theta_2}{M_1}(1)=2^1$ and $\tfp{\theta_2}{M_2}(1)=2^2$ (cf. Theorem~\ref{conj:f_size_conjecture} and Corollary~\ref{cor:3-edge-color-of-G}). This leads to the following question: 

\begin{question}
Does there exists a graph $G$ and two perfect matchings $M_1$ and $M_2$ of $G$ such that $\tfp{G}{M_1}(1)=\tfp{G}{M_2}(1)$ but $\tfp{G}{M_1} \not= \tfp{G}{M_2}$?
\end{question}

This question was answered by James Oxley with the following graph and perfect matchings:

\begin{center}
\psfragscanon
\psfrag{m}{$(G,M_1)$}\psfrag{N}{$(G,M_2)$}
\psfrag{b}{$b$}\psfrag{B}{$b'$}
\includegraphics[scale = .22]{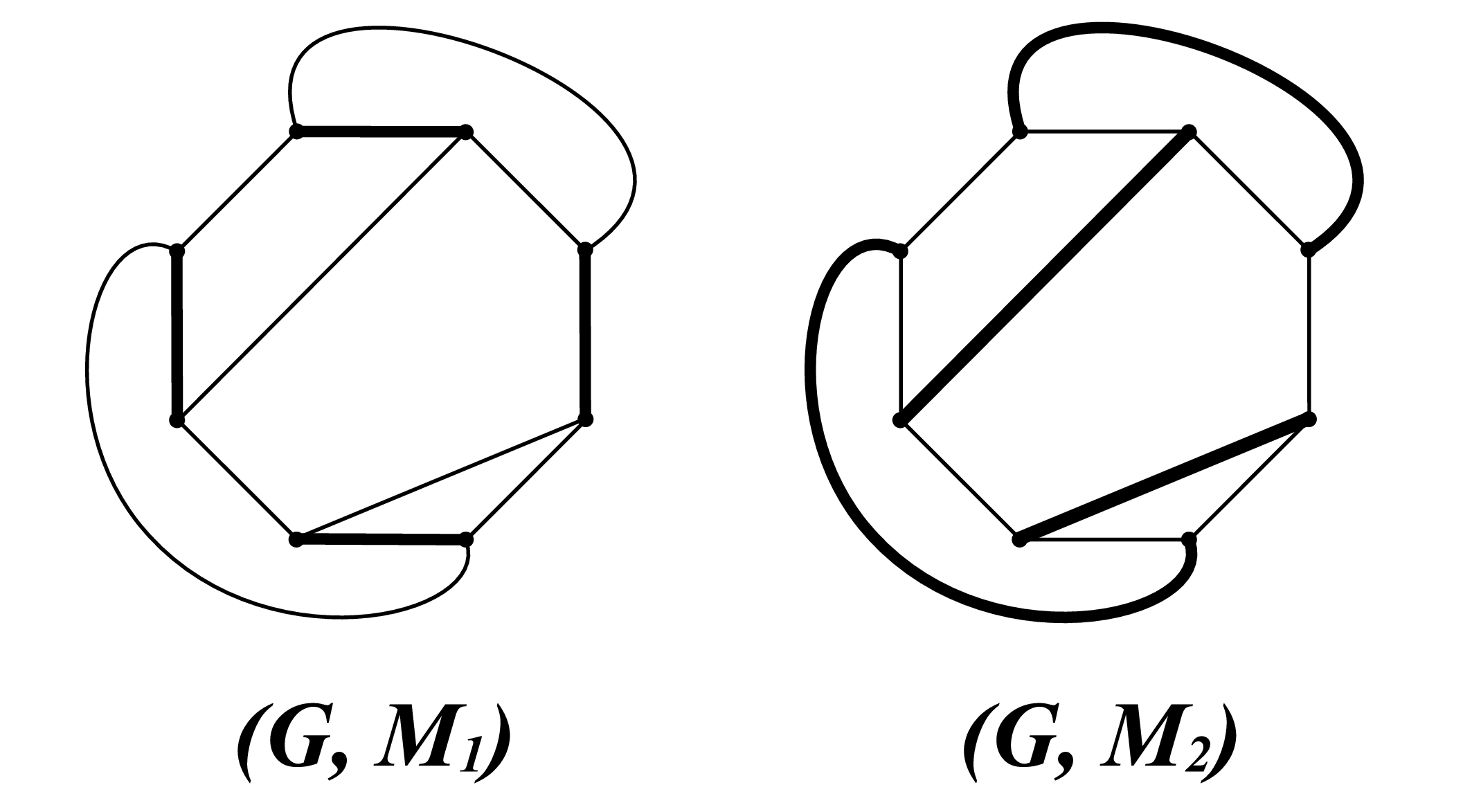} 
\label{fig:Oxley-Example}
\end{center}

A calculation shows that $\tfp{G}{M_1}(q)=q^{-3}-q^{-2}+q^3+q^6$ and $\tfp{G}{M_1}(q)=q^{-1}-1+2q+q^4-q^5$.  Thus the $2$-factor polynomial distinguishes both perfect matchings even though they both have the same value when evaluated at $1$. This example leads one to make a more refined conjecture:

\begin{conjecture} There exists a trivalent planar graph $G$ and two perfect matchings of the graph $M_1$ and $M_2$ such that $\tfp{G}{M_1}= \tfp{G}{M_2}$, but the pairs have different homologies, $H^{*,*}(G,M_1) \not\cong H^{*,*}(G,M_2)$. \label{conj:same-graph-diff-pm}
\end{conjecture}

This conjecture is different to how one thinks in knot theory.  In knot theory, one looks for two knots with the same Jones polynomial but different Khovanov homologies (cf. \cite{BN}).  This would be equivalent to finding {\em different} planar trivalent graphs with perfect matchings that have the same $2$-factor polynomial but different homology theories. Conjecture~\ref{conj:same-graph-diff-pm}, however, is looking for different perfect matchings {\em on the same graph} with this property. The knot theory version is:

\begin{conjecture} There exists two planar trivalent graphs with perfect matchings, $(G_1,M_2)$ and $(G_2,M_2)$,  such that 
$\tfp{G_1}{M_1}= \tfp{G_2}{M_2}$ and different homologies, $H^{*,*}(G_1,M_1) \not\cong H^{*,*}(G_2,M_2)$.
\label{conj:same-2-factor-different-cohomology}
\end{conjecture}

Taking another idea from knot theory, one can define the {\em ungraph} as a single vertex-less circle and ask:

\begin{question}
Are there any planar trivalent graphs with perfect matchings that have the same $2$-factor polynomial or homology theory as the ungraph? \label{quest:ungraph-detector}
\end{question}

The $2$-factor polynomial is more like the Kauffman bracket than the Jones polynomial in the sense that it is not ``normed'' like the Jones polynomial.  Recall that the Jones polynomial is normed by multiplying the Kauffman bracket  by a factor $(-1)^{n_-}q^{(n_+-2n_-)}$ where $n_+$ and $n_-$ are the total number of positive and negative crossings of a knot diagram.  (The $2$-factor can also be normed, but it is not necessary to do so. See \cite{BKR} for how.) This fact means that the $2$-factor polynomial distinguishes graphs that could be the same under a normed version of the invariant. For example, the {\em $m$-theta graph} $\theta_m$ with its natural perfect matching $M_\theta$ given in Figure~\ref{fig:m-theta-graph} has $2$-factor polynomial $ \tfp{\theta_m}{M_{\theta}}(q) =q^{-m}(q^{-1}+q)$.

\begin{figure}[H]
\psfrag{A}{$\ldots m$}\psfrag{B}{$\ldots m$}
\psfrag{C}{$(\theta_m, M_{\theta})$}\psfrag{D}{$(D_m, M_m)$}
\includegraphics[scale=0.25]{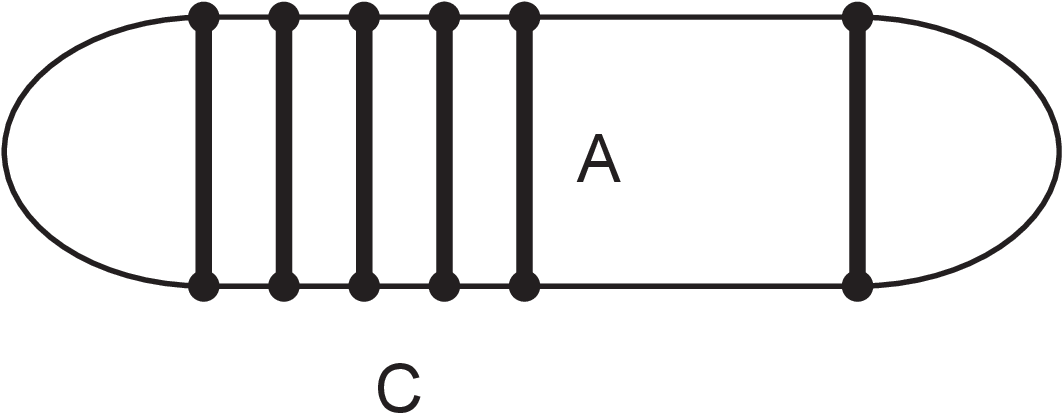}
\caption{The $\theta_m$ graph with perfect matchings $M_\theta$ shown with bold edges. Note that $m=0$ is the ungraph and $m=1$ is the usual theta graph.}
\label{fig:m-theta-graph}
\end{figure}

The normed version of the $2$-factor polynomial would multiply $q^{-m}(q^{-1}+q)$ by $q^{m}$, which would give the same invariant for all $m$-theta graphs.   (See Theorem~\ref{theorem:add-perfect-matching-edge} and the discussion above it to see why the operation of adding a perfect matching edge is like a Reidemeister one move in knot theory.)  Hence, Question~\ref{quest:ungraph-detector} is more like a question about whether the Kauffman bracket detects the framed unknot with the Seifert framing than whether the Jones polynomial detects the unknot (cf. \cite{KM4}).

The conjectures and questions in this section show that similar questions in knot theory take on different significance and meaning when ported over to graph theory.  They also show a set of problems about perfect matchings that have not been considered before in graph theory.  Since the invariants of this paper are providing key information about  cycles in a graph, and cycles are important to  graph theory (cf. the cycle double cover conjecture), these problems and their answers are likely to have consequences to graph theory topics suggested at the beginning of this section. 

The next section shows that the invariants of this paper generalize to  topics in graph theory that are not related to knot theory at all.

\subsection{Generalizing the bracket to $n$-regular graphs} \label{subsection:generalizing-the-bracket}
Theorem~\ref{theorem:general_poly_invariants}  and Theorem~\ref{thm:main-theorem} can be generalized to $n$-regular graphs.  In this subsection, the case for $n=4$ is briefly presented, which should be enough to indicate how to generalize  the theorems to $n$-regular planar graphs when $n>4$.   In particular, this subsection touches on two themes of this paper: (1) to explore other TQFT contexts in which the ideas behind the main theorems apply and (2) expose why the main theorems are broader than a straight translation of the Jones polynomial and Khovanov homology over to graph theory.  

Let $(G,M)$ be a planar $4$-regular graph $G$ with a perfect matching $M$. Let $\Gamma$ be a perfect matching graph of $(G,M)$, i.e., a plane graph of $G$ together with the perfect matching. Define a bracket   $\langle \Gamma \rangle \in \BZ[A,B,C]$ on the perfect matching graph $\Gamma$ characterized by:
\begin{eqnarray}
\bigg\langle \FourRegPMEdge \bigg\rangle &=&  A \left(\bigg\langle \IIIsmoothing \bigg\rangle + \bigg\langle \XXIsmoothing \bigg\rangle + \bigg\langle \IXXsmoothing \bigg\rangle \right) \ + \ B \left( \bigg\langle \XXXsmoothing \bigg\rangle + \bigg\langle \XIsmoothing \bigg\rangle + \bigg\langle \IXsmoothing \bigg\rangle \right)\label{eq:4-regular-generalformula}\\ 
\bigg\langle \bigcirc  \bigg\rangle & = & C  \label{eq:4-regular-immersed_circle_general}\\\
\bigg\langle \Gamma_1 \sqcup \Gamma_2 \bigg\rangle&=& \bigg\langle \Gamma_1 \bigg\rangle \cdot \bigg\langle \Gamma_2 \bigg\rangle\label{eq:4-regular-disjoint_graphs_identity_general}
\end{eqnarray}

This generalized bracket for $4$-regular planar graphs is invariant of flip moves on graphs.  For example, the proof of invariance for flipping a disk that intersects two distinct perfect matching edges of the plane graph each once in their interiors, i.e., a $2$-flip move, is similar to the $2$-flip move proof above for trivalent graphs. While there are no ``$1$-flip'' moves on $4$-regular planar graphs, there are other types of flip moves that need to be considered (cf. \cite{G} or \cite{MT}).  A careful check of these flip moves shows that the generalized bracket is an invariant of the planar graph and perfect matching.

\begin{theorem}
Let $(G,M)$ be a planar $4$-regular graph and $M$ a perfect matching of it.  The bracket $\langle \Gamma \rangle \in \BZ[A,B,C]$ of a perfect matching graph $\Gamma$ only depends on the pair $(G,M)$.\label{theorem:4-reg-invariant}
\end{theorem}

The generalization of this theorem to $n$-regular graphs should hold as well.  This is due to the fact that the even permutations smoothings associated with the $A$ term and the odd permutation smoothings associated with the $B$ term are grouped together, and there is a correspondence between smoothings of $A$ and the smoothings of $B$ under twisting. For example, a  twist of the $\IIIsmoothing$-smoothing of $A$ is the $\XXXsmoothing$-smoothing of $B$.

Just like the $2$-factor polynomial, a $3$-factor polynomial can be defined:

\begin{definition}
Let $(G,M)$ be a planar $4$-regular graph and $M$ be a perfect matching of it.  The {\em $3$-factor polynomial}, $\threefp{G}{M} \in \BZ[q,q^{-1}]$,  is the Laurent polynomial given by using $A=1$, $B=-q$, and $C=(q^{-1}+1+q)$ in Theorem~\ref{theorem:4-reg-invariant}. 
\end{definition}

Compare the definition of $3$-factor polynomial with Penrose's ``generalized Kronecker delta'' on page 227 of \cite{Penrose}.  In that paper, Penrose gave an interpretation for the generalized Kronecker delta in the special case of  two strands with loop value $C=3$  (in his notation: \raisebox{-0.36\height}{\includegraphics{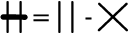}} with dimension $v=3$) in terms of $3$-edge colorings for trivalent planar graphs (the Penrose Formula).  However, he did not give meaning to the generalized Kronecker delta in other cases, including what happens for $n$-regular planar graphs with $n>3$.  In this paper, the $2$-factor polynomial  has already shed new light on the meaning when the loop value is $C=2$ for planar trivalent graphs (cf. Corollary~\ref{cor:3-edge-color-of-G}).  Next, we show how to generalize this idea to the $3$-factor polynomial and $4$-regular planar graphs.

To understand what colorings of planar $4$-regular graphs the $3$-factor polynomial is counting, the notion of when a $k$-factor ``factors through'' an $\ell$-factor  for $0\leq \ell \leq k$ is needed. First, recall the definition of a $k$-factor:  A {\em $k$-factor} of a graph $G$ is a spanning $k$-regular subgraph of $G$.   For example, given an $n$-regular graph, the vertex set $V(G)$ is always a $0$-factor, a perfect matching $M$ with its vertices is $1$-factor, and an $n$-regular graph is an $n$-factor.  (Here and below, when we think of a perfect matching $M$ as a $1$-factor, we include the vertices and edges, i.e., $(V(G), M)$, not just the edge set $M$.  We continue to call this subgraph $M$ when the context is clear.)  The following notation will be used for the set of $k$-factors that span an $\ell$-factor:

\begin{definition}
Given an $n$-regular graph $G$ and an $\ell$-factor $M$ of $G$, then a $k$-factor $K$  {\em factors through} $M$ if $M$ is a subgraph of $K$. Denote the set of all $k$-factors of $G$ that factor through $M$  by $\findex{G}{M}$. For a specific $k$, denote the set of all $k$-factors of $G$ that factor through $M$ by $\findex{G}{M}_k$.  Call the number of elements of $\findex{G}{M}_k$ the {\em index} and denote it by $\fsize{G}{M}_k$. 
\end{definition}

In this definition, the notation for the index, $\fsize{G}{M}_k$, was chosen to remind the reader of the index of a subgroup  in group theory, and in many ways, this is a helpful way to think about the number. In general, for a given $\ell$-factor $M$, $\fsize{G}{M}_k$ can be difficult to calculate. For example, $\fsize{G}{V(G)}_1$ counts the number of perfect matchings of $G$ and $\fsize{G}{V(G)}_2$ counts the number of $2$-factors of $G$.  The index is a useful notion.  For instance, the conclusion of Theorem~\ref{conj:f_size_conjecture} becomes $\tfp{G}{M}(1)=\fsize{G}{M}_2$.

The index can be used to describe what the $3$-factor polynomial counts: The $3$-factor polynomial, evaluated at one, is the total number of $2$-factors of a $3$-factor that factor through $M$, summed over all $3$-factors of the graph $G$ that factor through $M$, i.e., for a $4$-regular planar graph $G$ with perfect matching $M$,

\begin{eqnarray}
\threefp{G}{M}(1)  =\sum_{G' \in \findex{G}{M}_3} \fsize{G'}{M}_2. \label{eq:3-factor-eval-at-1}
\end{eqnarray}

The number, $\threefp{G}{M}(1)$, counts the number of 4-edge colorings (say with colors white, red, blue, purple) of $G$ that have all perfect matching edges labeled the same color, say purple (compare to Corollary~\ref{cor:3-edge-color-of-G}).  As a generalization of the Penrose Formula, it is an interesting enough number by itself to describe value of $\threefp{G}{M}(1)$, but Equation~\ref{eq:3-factor-eval-at-1} actually says more.  Since each term in the sum of the equation is non-negative, when $\threefp{G}{M}(1) = 0$, it implies that for all $3$-factors $G'$ that span $M$,  the graph $G'\setminus M$ must have an odd cycle (compare to Theorem~\ref{thm:existence_of_even_perfect_matching}), or that $G'$ has no $2$-factors that span $M$ (compare to Theorem~\ref{conj:f_size_conjecture}). Note: the last implication does not imply that when $\threefp{G}{M}(1) = 0$ that $G$ has no $2$-factors that span $M$---there can still be $2$-factors that span $M$ that are not part of any $3$-factor $G'$ of $G$. It is in this sense that the polynomial is  about $3$-factors and not $2$-factors, and therefore why it is called the $3$-factor polynomial.

In the next section, a bigraded cohomology theory is developed whose graded Euler characteristic is the $2$-factor polynomial (see Theorem~\ref{thm:main-theorem}).  The author has developed a cohomology theory for $4$-regular planar graphs based upon that construction using the graded vector space $V=\BZ_2[x]/(x^3)$.  In it, there are maps between each of the vector spaces associated to the smoothings in the $A$ term of  Equation~\ref{eq:4-regular-generalformula} to each of the vector spaces associated to the smoothings in the $B$ term.  Proving that it is invariant under flip moves of $4$-regular graphs is beyond the scope of this paper, and left as a conjecture:

\begin{conjecture}
Let $G$ be a planar 4-regular graph and let $M$ be a perfect matching of $G$.  Then there exists a cohomology theory of the pair $(G,M)$ that does not depend upon the perfect matching graph used to define it. Furthermore, the graded Euler characteristic of this cohomology is the $3$-factor polynomial of the pair $(G,M)$.\label{conj:4-reg-cohomology}
\end{conjecture}

 The cohomology in Section~\ref{section:2-factor-cohomology} for $3$-regular planar graphs and its generalizations to $n>3$ like in Conjecture~\ref{conj:4-reg-cohomology} represent a new family of TQFT-like theories. These homology theories are different from link homologies  since, for example, $\mbox{dim } V =3$ for the $V$ above, while  the dimension of the algebra for link homologies must be $2$ in order to be invariant under Reidemeister moves.

\section{The cohomology theory}
\label{section:2-factor-cohomology}

In this section a bigraded cohomology is defined whose graded Euler characteristic is the $2$-factor polynomial. Thus, the cohomology categorifies the $2$-factor polynomial.  The important results of this section are in defining the cochain complex $(C^{i,j},\del)$ for a perfect matching graph and showing $\del^2=0$.  The next  section shows that the cohomology defined by this complex is invariant after performing flip moves.  

For a perfect matching graph $\Gamma$ of a planar trivalent graph $G$ with perfect matching $M$, we assign a bi-graded cochain complex $(C^{*,*}(\Gamma), \partial)$ using tensors and sums of the graded vector space $V=\BZ_2[x]/(x^2)$.  

\subsection{Finite dimensional graded vector spaces} 
\label{subsection:finite-dim-graded-vs}
Recall that the {\em graded (or quantum) dimension}, $\qdim$, of a graded vector space $V=\oplus_m V^m$ is the polynomial in $q$ defined by $$\qdim(V)=\sum_m q^m\dim(V^m).$$

\noindent For a graded vector space $V$, we can shift the grading up or down by $\ell$ by $(V\{\ell\})^m=V^{m-\ell}$.
Clearly, $\qdim(V\{\ell\}) = q^\ell\cdot\qdim(V)$.  The  $\qdim$ is a polynomial in integer powers.

The graded vector space we use in this paper is $V=\BZ_2[x]/(x^2) = \langle 1, x\rangle$.  The grading is:
\begin{eqnarray} \label{eq:degrees-of-V}
\deg 1 &=& 1,\\
\deg x & = & -1. \nonumber
\end{eqnarray}

\noindent Thus, $\qdim V^{\otimes k} = (q^{-1}+q)^k$.
\subsection{Smoothings, states, and hypercubes}
\label{section:smoothing-states-hypercubes}

In this subsection, smoothings and states discussed in Section~\ref{section:polynomial_invariants} are used to describe the hypercube of states.  Using the same notation as the previous section, let $(G,M)$ be a planar trivalent graph $G$ with a perfect matching $M$ where the edges of $M$ are indexed from one to $n$: $M=\{e_1,\dots,e_n\}$. For a perfect matching graph $\Gamma$ of $(G,M)$, let $\Gamma_\alpha$ be a state indexed by $\alpha\in\{0,1\}^n$, where each $\alpha_i$ in $\alpha=(\alpha_1,\dots, \alpha_n)$ represents doing either a $0$-smoothing or $1$-smoothing for each $e_i \in M$.

It is useful to picture the set of states as a hypercube where each state is the vertex of a hypercube (with edges described momentarily).  For example,  Figure~\ref{fig:Theta_3_cube_of_states} is the hypercube of states of $(\theta_3,M_\theta)$.  Observe in Figure~\ref{fig:Theta_3_cube_of_states} that each state $\Gamma_\alpha$ is a collection of (possibly immersed) circles that intersect  transversely in double points. 

\begin{figure}[h]
\psfrag{-1-1-1-1}{\small $(\!-\!1\!,\!-\!1\!,\!-\!1\!,\!-1\!)$}\psfrag{A}{$a'$}
\psfrag{theta3}{$\theta_3=$}\psfrag{B}{$b'$}
\psfrag{000}{$(0,0,0)$}
\psfrag{100}{$(1,0,0)$}
\psfrag{010}{$(0,1,0)$}
\psfrag{001}{$(0,0,1)$}
\psfrag{110}{$(1,1,0)$}
\psfrag{101}{$(1,0,1)$}
\psfrag{011}{$(0,1,1)$}
\psfrag{111}{$(1,1,1)$}
\includegraphics[scale=.3]{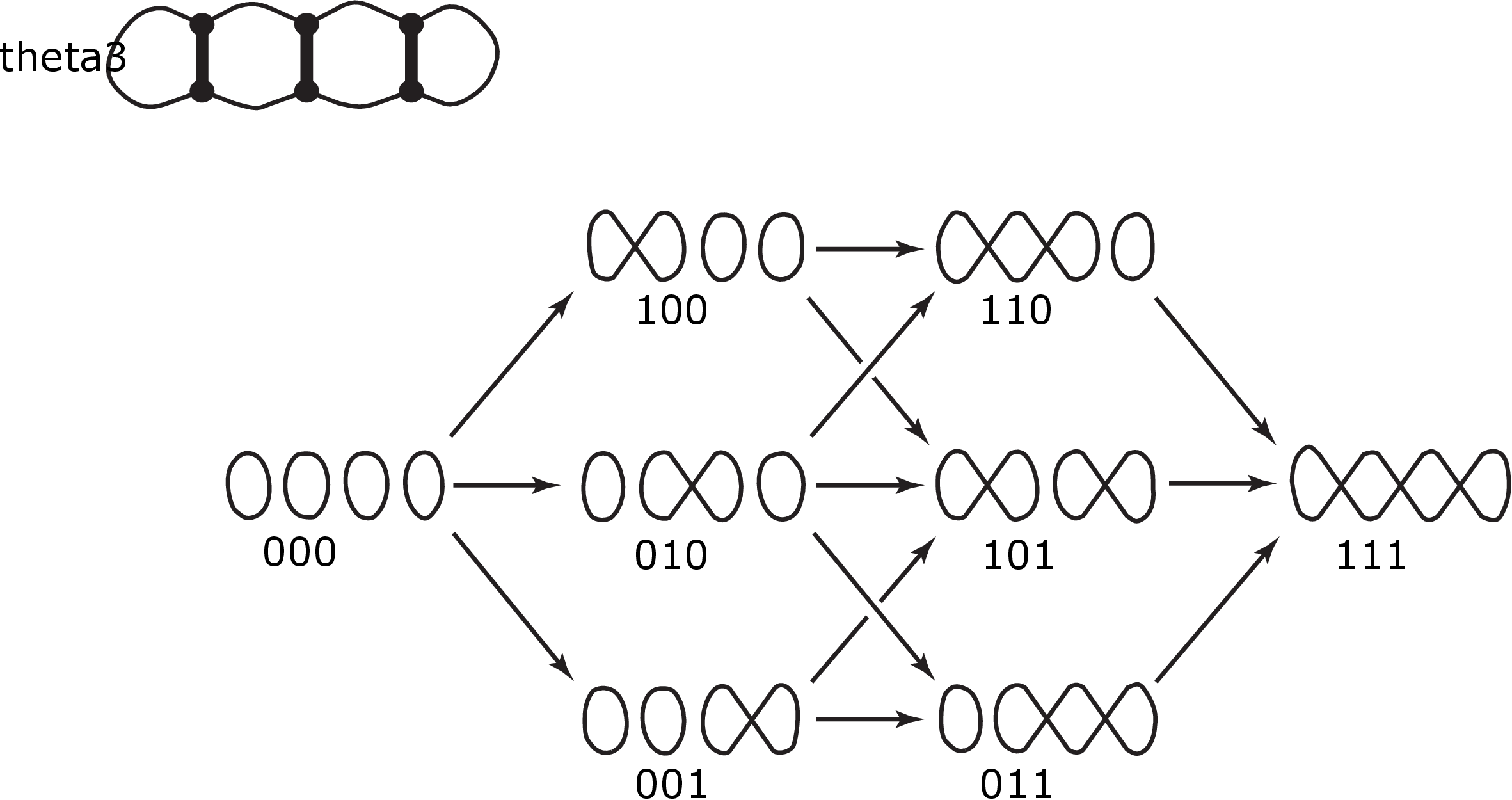}
\caption{Hypercube of states for the graph $\theta_3$ with perfect matching.}
\label{fig:Theta_3_cube_of_states}
\end{figure}


\subsection{The differential chain complex for $\Gamma$}  
\label{subsection:differential}

We are now ready to associate a graded vector space to our perfect matching graph $\Gamma$ for the pair $(G,M)$.  To each $\alpha\in\{0,1\}^n$, associate to $\Gamma_\alpha$ the  vector space $\large V_\alpha = \large V^{\otimes k_\alpha}\!\!\left\{|\alpha|\right\}.$
For example, in Figure~\ref{fig:Theta_3_cube_of_states}, the vector space associated to $\Gamma_{(1,0,0)}$ is $V_{(1,0,0)} = V^{\ot3}\{1\}$.

Define the complex $C^{*,*}(\Gamma)$ by

$$C^{i,*}(\Gamma)=\bigoplus_{\substack{\alpha\in\{0,1\}^n \\ i=|\alpha|}}V_\alpha.$$
The internal grading ($q$-grading) is defined by the grading of the elements in $V_\alpha$.  The homological grading $i$ is integer valued.   For an element $v\in V_\alpha\subset C^{*,*}(\Gamma)$, the homological grading $i$ and the $q$-grading $j$ satisfy:
\begin{eqnarray*}
i & =& |\alpha|,\\
j &=& \deg(v)+|\alpha|,
\end{eqnarray*}
where $\deg(v)$ is the degree of $v$ as an element of  $V^{\otimes k_\alpha}$ of $V_\alpha=V^{\otimes k_\alpha}\{|\alpha|\}$ before shifting the grading by $|\alpha|$. The complex is trivial outside of $i=0, \dots, n$.

The differential can now be defined. Each  $C^{i,*}(\Gamma)$ is the direct sum of vector spaces of the hypercube given by $i=|\alpha|$ (e.g. the columns of Figure~\ref{fig:Theta_3_cube_of_states}).  The edges in the hypercube of states correspond to maps between the graded vector spaces $V_\alpha$ in $C^{i,*}(\Gamma)$ to vector spaces in $C^{i+1,*}(\Gamma)$.  We need some notation to make this map precise.  Consider an edge $\eta$ in the hypercube between two states $\Gamma_\alpha$ and $\Gamma_{\alpha'}$.  Edges occur when $\alpha_i=\alpha'_i$ for all $i$ except for one edge $e_k\in M$ where $\alpha_k=|\alpha'_k-1|$.  For example, there is an edge  in Figure~\ref{fig:Theta_3_cube_of_states} between $(0,0,0)$ and $(1,0,0)$.  Turn each  of these edges into a directed segment $\eta:\Gamma_\alpha \rightarrow \Gamma_{\alpha'}$ by requiring the tail to be where $\alpha_k=0$ and the head where $\alpha'_k=1$, that is, the $0$-smoothing in $\Gamma_\alpha$ is changed into a $1$-smoothing in $\Gamma_{\alpha'}$.

On the level of vector spaces, the directed segment $\eta$ between $\Gamma_\alpha$ and $\Gamma_{\alpha'}$ corresponds to a linear map, $\del_\eta:V_\alpha \rightarrow V_{\alpha'}$.  To define $\del_\eta$,  note that each circle $c$ in the state $\Gamma_\alpha$ has a correpsonding vector space $V_c$ associated with it. The process of replacing the $0$-smoothing in $\Gamma_\alpha$ with the $1$-smoothing in $\Gamma_{\alpha'}$ either fuses two circles together, splits one circle into two, or introduces a double point to a circle.  The corresponding linear maps between the vector spaces are determined by these three processes:
\begin{enumerate}
\item If the process fuses two circles $c_1,c_2$ in $\Gamma_\alpha$ into one circle $c'_1$ in $\Gamma_{\alpha'}$, define a map $m_{12}:V_{c_1}\otimes V_{c_2} \rightarrow V_{c'_1}$ for this situation by multiplication in the algebra $V=\BZ_2[x]/(x^2)$.  That is, $$m(1\ot 1) =1, \ \ m(1\ot x) = m(x\ot 1)=x, \ \ \mbox{and} \ \  m(x\ot x) = 0.$$

\item If the process splits one circle $c_1$ in $\Gamma_\alpha$ into two circles $c'_1,c'_2$ in $\Gamma_{\alpha'}$,  define a map $\Delta_1: V_{c_1} \rightarrow V_{c'_2}\ot V_{c'_3}$ by comultiplication: $$\Delta(1)= x\ot 1 + 1\ot x, \ \
\mbox{ and } \ \ \Delta(x) = x\ot x.$$
\item If the proces introduces a double point in a circle $c_1$ in $\Gamma_\alpha$ to get a circle $c'_1$ in $\Gamma_{\alpha'}$, define a map $A:V_{c_1}\rightarrow V_{c'_1}$ to be the zero map: $A(1) =0$ and $A(x)=0$.
\end{enumerate}
The map $\del_\eta:V_\alpha\rightarrow V_{\alpha'}$ is defined as the tensor product of maps given by the identity on the vector spaces associated with circles that do not change from $\Gamma_\alpha$ to $\Gamma_{\alpha'}$, and either $m, \Delta$ or $A$ on the vector space(s) associated to circles that are modified by the change from a $0$-smoothing in $\Gamma_\alpha$ to a $1$-smoothing $\Gamma_{\alpha'}$.

The differential, $\del^i:C^{i,*}(\Gamma) \rightarrow C^{i+1,*}(\Gamma)$, is defined as the sum of appropriate $\del_\eta$'s.  For $v\in V_\alpha \subset C^{i,*}(\Gamma)$, 
$$\del^i(v) = \sum_{\substack{\eta \ \mbox{\tiny such that }\\ \mbox{\tiny Tail}(\eta) = \alpha}} \del_\eta(v).$$
A sign for each $\del_\eta$ in the sum is not needed due to $\BZ_2$ coefficients.  It is possible to build an integral cohomology for perfect matching graphs, but it is unknown whether it is invariant under the flip moves.

Readers familiar with Khovanov homology \cite{K1} should recognize the maps $m$ and $\Delta$.  The extra map, $A:V\ra V$, is due to the fact that we are working with immersed circles rather than only embedded circles as in Khovanov homology.  The $A$ map is zero in order to preserve the $q$-grading of the differential. In this way, the $A$ map is similar to what happens in virtual link theory (cf. \cite{ArbitraryCoeffs}).  

We now have enough to prove that the main theorem of this section:

\begin{theorem}  $(C^{i,*}(\Gamma), \del^i)$ is a  cochain complex with differential that increases the homological degree by one and preserves the quantum grading, i.e., it has bigrading $(1,0)$. \label{thm:A-Chain-Complex}
\end{theorem}

\begin{proof}
To show that the square of the differential is zero, each diagram of maps corresponding to each possible face in the hypercube of states must be shown to commute.  Recall the standard TQFT/Frobenius algebra argument in Khovanov homology: since saddles appear at different soothing sites and commute in the cobordism category, the induced maps must also commute. In this situation, that argument can not be appealed to directly because of the $A$ map.  Instead, each possible pair of compositions of the maps $m, \Delta$ and $A$ must be analyzed in a case-by-case basis and be shown to lead to commuting diagrams, or are ruled out as possible diagrams by the Jordan curve theorem. 

Fortunately, many of the cases are handled by what is already known about Frobenius algebras and the construction of Khovanov homology, i.e., all diagrams involving only $m$ and $\Delta$ commute as desired \cite{K1}.  The remaining diagram cases are interactions of $m$ and $\Delta$ with the $A$ map.  Since the $A$ map is 0, diagrams such as $A\circ m = A \circ m$ automatically commute (cf. Figure~\ref{fig:3-prism-hypercube-of-states} to see such an example).   An exhaustive analysis of possible commuting diagrams involving $A$ with $m$ and $\Delta$ shows that they either (1) do not come from a hypercube face because they violate the Jordan curve theorem, (2)  commute due to  the fact that $A$ is the zero map, or (3) are the diagram $A\circ A = m\circ \Delta$.  The last case does show up in  hypercubes.  For example, one of the faces of the hypercube of $P_3$ with the candlestick perfect matching (see Figure~\ref{fig:P_3} in Section~\ref{section:Examples}) is:

\begin{figure}[H]
\psfragscanon
\psfrag{A}{$A$}
\psfrag{m}{$m$}
\psfrag{D}{$\Delta$}
\includegraphics[scale=.5]{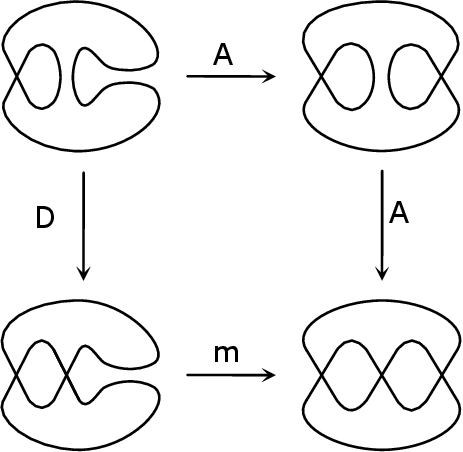}
\label{fig:P_3_A-squared-ex}
\end{figure}
The final case, $A\circ A = m\circ \Delta$, can be checked by hand: Applying the maps to $x$ gives $m\circ\Delta(x) = m(x\ot x) =0$ and $A\circ A(x)=0$. Hence, the diagram commutes in this situation.  Also, since we are using $\BZ_2$ coefficients, the diagram commutes for $1$:  $m\circ\Delta(1) = m(1\ot x + x \ot 1) = 2x=0$.  Since all possible valid diagrams commute,  $\del^{i+1}\circ\del^i=0$.

A calculation shows that the bigrading of $\del^i$ is $(1,0)$. 
\end{proof}

\subsection{Cohomology of graphs with perfect matchings} 
\label{section:cohomology-of-graphs-with-perfect-matchings}
We are now ready to define the cohomology of the pair $(G,M)$.  

\begin{definition} Let $(G,M)$ be a planar trivalent graph $G$ with perfect matching $M$.  Let $\Gamma$ be a perfect matching graph of $(G,M)$.  The {\em cohomology} of the pair $(G,M)$ is
$$H^{i,j}(G,M) = \frac{\ker \del:C^{i,j}(\Gamma) \ra C^{i+1,j}(\Gamma)}{\Ima \del:C^{i-1,j}(\Gamma) \ra C^{i,j}(\Gamma)}.$$
\end{definition}

The cochain complex is enough to prove the second part of Theorem~\ref{thm:main-theorem}, that is, the $2$-factor polynomial is equal to the $q$-graded Euler characteristic of this cohomology.  Since the $q$-graded Euler characteristic of the cohomology is the same as that of the cochain groups, the graded Euler characteristic of $H^{i,j}(G,M)$ can be expressed in terms of the vector spaces associated to each vertex of the hypercube of the perfect matching graph.  Using $\qdim V^{\otimes k} = (q^{-1}+q)^k$ and $A=1$, $B=-q$, and $C=(q^{-1}+q)$ in  Equation~\ref{eqn:general_state-sum}, one gets $$\chi_q(H^{*,*}(G,M)) = \tfp{G}{M}(q).$$

\section{The cochain complex is invariant under flip moves}
\label{section:invariant-under-0-and-1-flip-moves}

One of the main theorems of this paper, Theorem~\ref{thm:main-theorem}, is an immediate consequence of the following stronger theorem:

\begin{theorem} Let $(G,M)$ be a planar trivalent graph $G$ with perfect matching $M$. Let $\Gamma$ and $\tilde{\Gamma}$ be perfect matching graphs of $(G,M)$ related by local isotopies and a sequence of flip moves described in (1), (2), and (3c) of Analysis~\ref{analysis:flip_moves}. Then $(C^{i,j}(\Gamma),\del) \cong (C^{i,j}(\tilde{\Gamma}),\tilde{\del})$ as cochain complexes, that is, there exists a cochain isomorphism $S: C^{i,j}(\Gamma) \ra C^{i,j}(\tilde{\Gamma})$ for each $i$ and $j$ such that $\del \circ S = S\circ \tilde{\del}$.  \label{theorem:main-theorem-of-cohomology}
\end{theorem}

 \begin{proof}
 This theorem  follows from Theorem~\ref{thm:perfect-matching-iso} once we show that the cochain complex is invariant under the flip moves.  Thus, we finish the proof of Theorem~\ref{theorem:main-theorem-of-cohomology} by proving invariance of the cohomology under the $0$-flip move (Proposition~\ref{Prop:Cohomology-invariant-under-0-flip-move}), the $1$-flip move (Proposition~\ref{Prop:Cohomology-invariant-under-1-flip-move}), and the $2$-flip move (Proposition~\ref{Prop:Cohomology-invariant-under-2-flip-move})  in the remaining subsections.
\end{proof}

The invariance under $2$-flip moves is the difficult part of the proof.  This is because the circles of the states before and after a $2$-flip move  can interact with each other differently at the smoothing sites. 

\begin{remark} Theorem~\ref{theorem:main-theorem-of-cohomology} also proves Theorem~\ref{theorem:general_poly_invariants} for the $2$-factor polynomial.  But there are many choices for the loop value $C$ in Theorem~\ref{theorem:general_poly_invariants}  for which this Theorem~\ref{theorem:main-theorem-of-cohomology} does not apply.  For example, the four-color polynomial with $C=q^{-1}+1+q$ is such an example.
\end{remark}

\subsection{The cochain complex is invariant under $0$-flip and $1$-flip moves and some consequences} \label{subsec:cochain-inv}

The proof of invariance under $0$-flip moves follows from the definition of the cochain complex:  After reflecting a component of the graph, the hypercube of states is essentially the same as the original hypercube of states, except all of the circles derived from that component are reflected.  These reflected circles do not change any of the states of the perfect matching edges.  In particular, for all $\alpha \in \{0,1\}^n$,  $k_\alpha$ is the same for the reflected component, the indexes and associated vector spaces remain the same after the flip, the smoothing sites are the same, and therefore the associated maps between vector spaces are as well. Hence,

\begin{proposition}
Let $\Gamma$ and $\tilde{\Gamma}$ be two perfect matching graphs of $(G,M)$ related by a $0$-flip move.  Then $(C^{i,j}(\Gamma),\del) \cong (C^{i,j}(\tilde{\Gamma}),\tilde{\del})$ via a map $S: C^{i,j}(\Gamma) \ra C^{i,j}(\tilde{\Gamma})$  that is the canonical isomorphism. \label{Prop:Cohomology-invariant-under-0-flip-move}
\end{proposition}

An immediate consequence of the definition cochain complex is that the cohomology of the union of two disjoint graphs is the tensor product of their cohomologies:

\begin{proposition}
Let $(G_1,M_1)$ and $(G_2,M_2)$ be two connected, planar trivalent graphs with perfect matchings.  Then 
$$H^{i,j}(G_1 \sqcup G_2, M_1 \sqcup M_2) = \sum_{\substack{i=k+m\\j=l+n}} H^{k,l}(G_1,M_1)\ot H^{m,n}(G_2,M_2).$$ \label{prop:disjoint-union-of-graphs}
\end{proposition}

A $1$-flip move is performed by choosing a disk whose boundary intersects the interior of a single perfect matching edge of the graph in a perfect matching graph (cf. Lemma~\ref{thm:bridge-is-in-perfectmatching}). Let $\Gamma$ be the initial perfect matching graph and $\tilde{\Gamma}$ be the perfect matching graph after a $1$-flip along edge $e_i\in M$.  As in the $0$-flip case, the circles of a state $\Gamma_\alpha$ that do not pass through the $\alpha_i$-smoothing either remain the same (if they are outside the flipping disk) or are reflected (if they are in the flipping disk) in $\tilde{\Gamma}_\alpha$.  

The only circle left to consider is the circle that passes through the $\alpha_i$-smoothing: The proof of Proposition~\ref{lemma:k_alpha_is_k_alpha_prime} shows that, in $\Gamma_\alpha$, any circle that enters through one of the two arcs of $\alpha_i$ (for $\alpha_i=0$ or $\alpha_i=1$) must then exit through the other arc.  The same must occur in $\tilde{\Gamma}_\alpha$, i.e., the circle in $\Gamma_\alpha$ that passes through $\alpha_i$ is the same circle in $\tilde{\Gamma}_\alpha$ except it travels in the opposite direction while in the flipping disk. This means that if two circles in $\Gamma_\alpha$ merge to get one circle in $\Gamma_{\alpha'}$, then the same will occur for the corresponding circles in $\tilde{\Gamma}_\alpha$ and $\tilde{\Gamma}_{\alpha'}$.  A similar correspondence occurs if one circle splits into two (the $\Delta$ map) or if a double point is introduced to a circle (the $A$ map).

Hence,  the indexes and associated vector spaces remain the same after the flip, as well as the associated maps between those vector spaces.  Thus,

\begin{proposition}
Let $\Gamma$ and $\tilde{\Gamma}$ be two perfect matching graphs of $(G,M)$ related by a $1$-flip move.  Then $(C^{i,j}(\Gamma),\del) \cong (C^{i,j}(\tilde{\Gamma}),\tilde{\del})$ via a map $S: C^{i,j}(\Gamma) \ra C^{i,j}(\tilde{\Gamma})$  that is the canonical isomorphism. \label{Prop:Cohomology-invariant-under-1-flip-move}
\end{proposition}

In the remainder of this subsection, some calculations  needed for computing the cohomology of examples in Section~\ref{section:Examples} are discussed. The calculations are described here because they are variations on the previous proposition: the first describes ``adding a loop'' to a graph (the simplest version of where a $1$-flip move can be performed) and the other adds a perfect matching edge to a given non-perfect matching edge of a graph (which behaves like a Reidemeister one move in knot theory).

When a graph has a loop, the cohomology calculation of the graph and perfect matching simplifies.  Let $(G,M)$ be a planar trivalent graph $G$ with perfect matching $M$.  For an edge $e\in E\setminus M$ in a perfect matching graph, pick a point $v$ in the interior of $e$ and add a perfect matching edge with a loop to $G$ at $v$ as in Figure~\ref{fig:Add-Lollipop}:  

\begin{figure}[h]
\psfrag{G}{$G$}
\psfrag{g}{$G'$}
\psfrag{a}{$e$}
\psfrag{b}{$v_1$}
\psfrag{c}{$v_2$}
\psfrag{d}{$e'$}
\psfrag{e}{$e''$}
\psfrag{v}{$v$}
\psfrag{l}{$m$}
\psfrag{m}{$\ell$}
\psfrag{w}{$v'$}
\includegraphics[scale=.750]{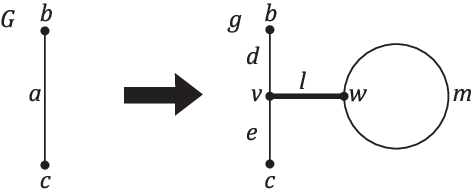}
\caption{Adding a loop to $(G,M)$ to get $(G',M')$.}
\label{fig:Add-Lollipop}
\end{figure}

\noindent In terms of graphs, this amounts to the following operation:
\begin{equation}
G'=\left(V \cup \{v,v'\}, \left(E\setminus \{ e \}\right) \cup \{ e', e'', \ell, m \}\right), \mbox{\ \  and  \ \ }  M'=M\cup \{ m\}.
\label{eq:adding-lollipop}
\end{equation}

Let $\Gamma$ and $\Gamma'$ be perfect matching graphs for $(G,M)$ and $(G',M')$ respectively.  The number of vertices in the hypercube of states for $\Gamma'$ is $2^{n+1}$, which  can be thought of as two copies of the hypercube of states of $\Gamma$ (with $2^n$ vertices each) connected by edges defined as follows:  In the hypercube of states of $\Gamma$, the edge $e$ is part of only one circle (since it is not a perfect matching edge) in each state $\Gamma_\alpha$.  In the first copy of the hypercube of states, remove $e$ from each of these circles and replace it with $\HorseshoeDiag$, and in the second hypercube of states, remove $e$ and replace it with $\RibbonDiag$.  Connect each state with a $\HorseshoeDiag$ to the corresponding state with a $\RibbonDiag$ with an edge. 

One can then calculate the cohomology of $(G',M')$ using this description of the hypercube of states and the cohomology of $(G,M)$:

\begin{proposition}\label{theorem:cohomology-of-a-loop}
Let $(G,M)$ be a planar trivalent graph $G$ with a perfect matching $M$.  Let $(G',M')$ be the graph and perfect matching constructed by adding a perfect matching edge with a loop to the graph $G$ at edge $e$ according to Equation~\ref{eq:adding-lollipop} (see Figure~\ref{fig:Add-Lollipop}). Then 
\begin{eqnarray*}
H^{i,j}(G',M') & \cong  & H^{i,j}(G,M) \oplus H^{i-1,j-1}(G,M).
\end{eqnarray*}
\end{proposition}

This theorem brings up the question of computing the cohomology of  graphs  joined together by a bridge:

\begin{question} Let $(G,M)$ and $(G',M')$ be two planar trivalent graphs and perfect matchings.  Let $e$ be an edge of $G$ such that $e\not\in M$ and $e'$ be an edge of $G'$ such that $e' \not\in M'$.  Construct $(G'',M'')$ by joining $e$ and $e'$ via a bridge edge $m$ ($M''=M \cup M' \cup \{ m \}$) using a similar procedure as Equation~\ref{eq:adding-lollipop} (see Figure~\ref{fig:Add-Lollipop}).  What is the cohomology of $(G'',M'')$ in terms of the cohomology $(G,M)$ and $(G',M')$?
\end{question}

Another operation for which the cohomology can be calculated is adding a perfect matching edge to a given non-perfect matching edge of the graph.  Let $\Gamma$ be a perfect matching graph for $(G,M)$ and let $e$ be an edge not in $M$, then construct $\Gamma'$ for a new pair $(G',M')$ as in the following picture:

\begin{figure}[H]
\psfrag{G}{$G$}
\psfrag{G'}{$G'$}
\psfrag{m}{$m_e$}
\psfrag{e}{$e$}
\includegraphics[scale=.60]{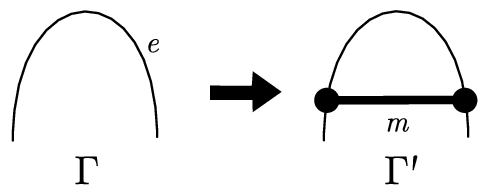}
\caption{Adding a perfect matching edge $m$ to edge $e$  to get $\Gamma'$.}
\label{fig:Add-perfect-matching-edge}
\end{figure}

The hypercube of states for $\Gamma'$ bifurcates into two hypercubes that correspond to the $0$-smoothing of the edge $m$  and the $1$-smoothing, with edges between them.  These hypercubes are depicted in Figure~\ref{fig:Smoothings-for-PM-edge-addition}.

\begin{figure}[H]
\psfrag{G}{$G$}
\psfrag{G'}{$G'$}
\psfrag{m}{$m_e$}
\psfrag{e}{$e$}
\includegraphics[scale=.60]{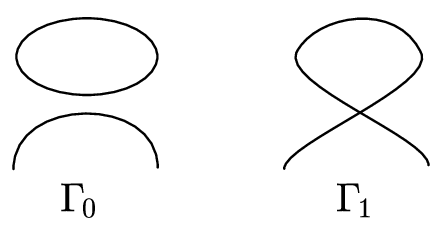}
\caption{The hypercube of states for the $0$-smoothing, $\Gamma_0$, and $1$-smoothing, $\Gamma_1$, of $\Gamma'$.}
\label{fig:Smoothings-for-PM-edge-addition}
\end{figure}

Notice that the states in these hypercubes are nearly the same as the diagrams used to prove that Khovanov homology is invariant under a Reidemeister one move.  In fact, they generate  the same groups and maps on the chain complexes.  Hence, Khovanov's original proof (see Section 5.1 of \cite{K1}) can be used to show that the complex $C(\Gamma')$ is quasi-isomorphic to a gradings-shifted version of the complex $C(\Gamma_1)$.  An analysis of the shifting shows that the $q$-grading of the cohomology of $\Gamma'$ is shifted down by one when  compared to the cohomology of $\Gamma$.  Thus,

\begin{proposition}\label{theorem:add-perfect-matching-edge}
Let $(G,M)$ be a planar trivalent graph $G$ with a perfect matching $M$.  Let $(G',M')$ be the graph constructed from $(G,M)$ by introducing a  perfect matching edge to some edge of $G$ not in $M$ (cf. Figure~\ref{fig:Add-perfect-matching-edge}).  Then
$H^{i,j}(G',M')  \cong  H^{i,j+1}(G,M)$.
\end{proposition}

Another way to view the operation of obtaining $(G',M')$ from $(G,M)$ is to connect sum  $(G,M)$ with the theta graph and its perfect matching $(\theta,M_\theta)$, i.e., $(G',M') = (G \# \theta, M\cup M_\theta)$.  This leads to the question:

\begin{question} Let $(G_1,M_1)$ and $(G_2,M_2)$ be two planar trivalent graphs with perfect matchings.  Let $(G_1\# G_2, M_1\cup M_2)$ be the connect sum of the two graphs along edges in $G_1$ and $G_2$ that are not perfect matching edges.   What is the cohomology of $(G_1\# G_2, M_1\cup M_2)$ in terms of each of the cohomologies?
\end{question}

\subsection{The cochain complex is invariant under $2$-flip moves}
\label{section:cohomology-2-flip-move-invariance}

In this subsection and the next two subsections we complete the proof of Theorem~\ref{theorem:main-theorem-of-cohomology}, and therefore Theorem~\ref{thm:main-theorem}, by proving:  

\begin{proposition}
Let $\Gamma$ and $\tilde{\Gamma}$ be two perfect matching graphs of $(G,M)$ related by a $2$-flip move described in (3c) of Analysis~\ref{analysis:flip_moves}.  Then $(C^{i,j}(\Gamma),\del) \cong (C^{i,j}(\tilde{\Gamma}),\tilde{\del})$ via an isomorphism $S: C^{i,j}(\Gamma) \ra C^{i,j}(\tilde{\Gamma})$ that commutes with the differentials. \label{Prop:Cohomology-invariant-under-2-flip-move}
\end{proposition}

Let $\Gamma$ be a perfect matching graph for $(G,M)$.  Assume that $G$ is a connected graph.  Let $D\subset S^2$ be a flipping disk whose boundary intersects $\Gamma\subset S^2$ along the interiors of two edges $e_1$ and $e_2$  as in 3(c) of Analysis~\ref{analysis:flip_moves}.  Let $\tilde{\Gamma}$ be the perfect matching graph of $(G,M)$ after performing the $2$-flip on $\Gamma$ using the flipping disk $D$.  According to 3(c), either both edges are perfect matching edges or both are not.

If $e_1,e_2\not\in M$, then $e_1$ and $e_2$ are both part of at least one cycle of the graph $G$ (neither are bridges by Lemma~\ref{thm:bridge-is-in-perfectmatching}).  This means that for every state $\Gamma_\alpha$ of the perfect matching graph $\Gamma$, the circle of $\Gamma_\alpha$ that contains (the interior of) edge $e_1$ must also contain (the interior of) edge $e_2$: let $\beta_\alpha$ be the arc of that circle that is contained within the flipping disk.  Thus, the hypercube of states of $\Gamma$ and $\tilde{\Gamma}$ are the same, except that all of the circles and arc $\beta_\alpha$ contained within the flipping disk of each $\Gamma_\alpha$ are reflected.  In particular, for $\Gamma_\alpha$ and $\tilde{\Gamma}_\alpha$, the number of circles, smoothing sites between the circles, and the configurations of how the circles intersect are exactly the same.  This induces a chain map on their associated vector spaces that is a canonical isomorphism.  Therefore, the cochain complex defined by $\Gamma$ is isomorphic to the cochain complex defined by $\tilde{\Gamma}$ via this chain map.

The hardest case of Proposition~\ref{Prop:Cohomology-invariant-under-2-flip-move} is to prove that cochain complex of $\Gamma$ is isomorphic to the cochain complex of $\tilde{\Gamma}$ when the two are related by a $2$-flip move when $e_1,e_2 \in M$ as in 3(c) of Analysis~\ref{analysis:flip_moves}.  The next two subsections address this case. The first subsection reformulates the chain groups in terms of exterior algebras with $\BZ_2$ coefficients and defines the map $$S:C^{i,j}(\Gamma) \ra C^{i,j}(\tilde{\Gamma})$$ in terms of them.  The second subsection proves that $S$ is a chain map that induces an isomorphism between the two cochain complexes.

\subsection{Defining the map $S$ when $e_1,e_2 \in M$ of 3(c)}  Let $k$ be the function that counts the number of circles in a state, i.e., $k(\Gamma_\alpha)=k_\alpha$ (cf. Section~\ref{section:polynomial_invariants}).  By Proposition~\ref{lemma:k_alpha_is_k_alpha_prime-cut-edge-set-case}, $k(\Gamma_\alpha)=k(\tilde{\Gamma}_\alpha)$ for all $\alpha \in \{0,1\}^n$.  Hence, the vector space $V_\alpha$ associated to the state $\Gamma_\alpha$ of $\Gamma$ is the same as the vector space $\tilde{V}_\alpha$ associated to the state $\tilde{\Gamma}_\alpha$ of $\tilde{\Gamma}$.  Likewise, the graded vector spaces  $C^{i,j}(\Gamma)$ are the same as $C^{i,j}(\tilde{\Gamma})$ for all $i$ and $j$. 

To define the map $S:C^{i,j}(\Gamma) \ra C^{i,j}(\tilde{\Gamma})$, we examine the different cases of how circles in $\Gamma_\alpha$ can enter or exit the flipping disk. This analysis is similar to the proof of Proposition~\ref{lemma:k_alpha_is_k_alpha_prime-cut-edge-set-case}, but now takes into account how the arcs are interacting at the smoothing sites with immersed circles inside the flipping disk. 

\begin{figure}[h]
\psfragscanon
\psfrag{a}{$a$}\psfrag{A}{$\tilde{a}$}
\psfrag{b}{$b$}\psfrag{B}{$\tilde{b}$}
\psfrag{c}{$c$}\psfrag{C}{$\tilde{c}$}
\psfrag{d}{$d$}\psfrag{D}{$\tilde{d}$}
\psfrag{g}{$\Gamma_\alpha$}
\psfrag{f}{$\tilde{\Gamma}_{\alpha}$}
\includegraphics[scale = 0.25]{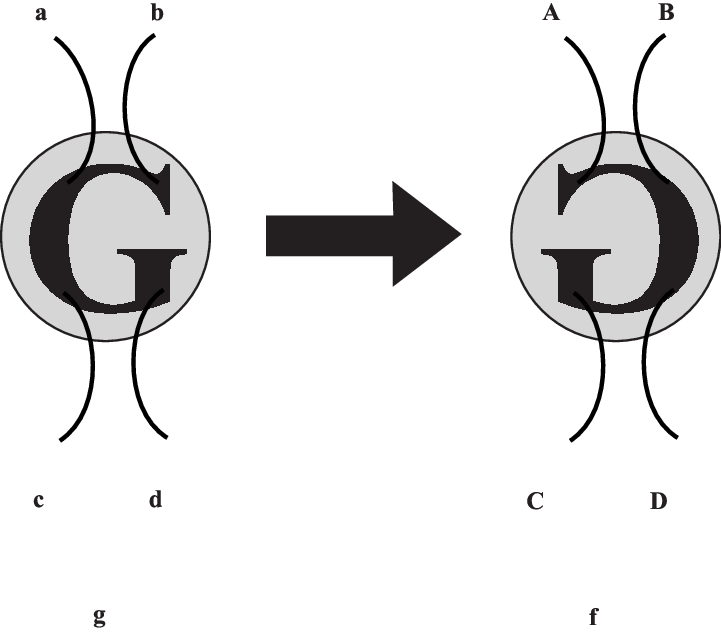}
\caption{The flipping disk of $\Gamma_\alpha$ and $\tilde{\Gamma}_\alpha$ in the case of a $0$-smoothing of edge $e_1$ (top) and a $0$-smoothing of edge $e_2$ (bottom).}
\label{fig:Homology-Invariant-3a-1}
\end{figure}

\begin{analysis} Figure~\ref{fig:Homology-Invariant-3a-1} shows a local picture of the flipping disk of the state $\Gamma_\alpha$ of $\Gamma$ in the situation where, without loss of generality, both edges $e_1$ and $e_2$ are  $0$-smoothings, and the resulting state $\tilde{\Gamma}_\alpha$ in the $2$-flipped  $\tilde{\Gamma}$. There are four possible cases for how circles enter and exit the flipping disk in both $\Gamma_\alpha$ and $\tilde{\Gamma}_\alpha$:
\label{analysis:circles-in-flipping-disk}
\end{analysis}

\begin{enumerate}
\item {\bf One circle, an arc for each smoothing.}  In state $\Gamma_\alpha$, a single circle enters and exits the flipping disk twice: once through an arc $\beta_\alpha$ in the disk that enters at arc $a$ and exits at arc $b$, and once through an arc $\gamma_\alpha$ in the disk that enters at arc $c$ and exits at arc $d$.  After the $2$-flip move,  the arc that enters at $\tilde{a}$ transverses the reflection of $\beta_\alpha$, and then exits at $\tilde{b}$.  A similar statement applies to the $\gamma_\alpha$ arc. Thus, each circle in $\Gamma_\alpha$ gets mapped to a circle in $\tilde{\Gamma}_\alpha$ with all the same smoothing sites (i.e., $0$- and $1$-smoothings) as in $\Gamma_\alpha$.

\item {\bf One circle, arcs through both smoothings.} In state $\Gamma_\alpha$, a single circle enters and exits the flipping disk twice: once through an arc $\beta_\alpha$ in the disk that enters at arc $a$ and exits at arc $c$ (or $d$), and once through an arc $\gamma_\alpha$ in the disk that enters at arc $b$ and exits at arc $d$ (or $c$), i.e., both arcs run from the top smoothing to the bottom smoothing  of the picture through the flipping disk.   After the 2-flip move, the arc that enters at $\tilde{a}$ now transverses the reflection of $\gamma_\alpha$, and then exits at arc $\tilde{c}$ (or $\tilde{d}$).  The other arc enters at arc $\tilde{b}$, transverses the reflection of ${\beta}_\alpha$, and exits at arc $\tilde{d}$ (or $\tilde{c}$).  Thus, the resulting circle  after the $2$-flip ``exchanges'' the arcs in $\tilde{\Gamma}_\alpha$ of the original circle in $\Gamma_\alpha$.  Since this is an exchange of arcs within the same circle, each circle in $\Gamma_\alpha$ gets mapped to a circle in $\tilde{\Gamma}_\alpha$ that continues to have all the same smoothing sites as in $\Gamma_\alpha$.

\item  {\bf Two circles, an arc for each smoothing.}  In state $\Gamma_\alpha$, two circles enter and exit the flipping disk: the first circle enters and exists through an arc $\beta_\alpha$ in the disk that enters at  arc $a$ and exits at arc $b$.   The second circle enters and exits through an arc $\gamma_\alpha$ in the disk that enters at  arc $c$ and exits at arc $d$.  After the $2$-flip move,  the first circle enters at $\tilde{a}$, transverses the reflection of $\beta_\alpha$,  and exits at arc $\tilde{b}$.  A similar comment applies to the second circle. Thus, each circle in $\Gamma_\alpha$ gets mapped to a circle in $\tilde{\Gamma}_\alpha$ that continue to have all the same smoothing sites as in $\Gamma_\alpha$.

\item {\bf Two circles, arcs through both smoothings.} In state $\Gamma_\alpha$, two circles enter and exit the flipping disk: the first circle through an arc $\beta_\alpha$ in the disk that enters at arc $a$ and exits at arc $c$~(or $d$).  The second circle enters and exits through an arc $\gamma_\alpha$ in the disk that enters at arc $b$ and exits at arc $d$~(or $c$).  The arcs of both circles run from the top smoothing to the bottom one.   After the 2-flip move, the first circle  enters at arc $\tilde{a}$, transverses the reflection of $\gamma_\alpha$, and exits at arc $\tilde{c}$ (or $\tilde{d}$).  The second circle enters at arc $\tilde{b}$, transverses the reflection of $\beta_\alpha$, and exits through arc $\tilde{d}$ (or $\tilde{c}$).  The $2$-flip  exchanges the two arcs of the two circles.  In this case, the circles exchange the smoothing sites contained within the flipping disk.  For example, if there is  a third circle in the flipping disk that has a smoothing site with the first circle in $\Gamma_\alpha$, the reflection of that circle will now have a smoothing site with the second circle in $\tilde{\Gamma}_\alpha$. 

\end{enumerate}

To set up how the mapping between algebras $V_\alpha$ and $\tilde{V}_\alpha$, a correspondence between circles in $\Gamma_\alpha$ correspond to circles in $\tilde{\Gamma}_\alpha$ based upon  Analysis~\ref{analysis:circles-in-flipping-disk} is required. For circles that do not enter or exit the flipping disk, the correspondence is well-defined:  $c\subset \Gamma_\alpha$ corresponds to the same circle $\tilde{c}$ in $\tilde{\Gamma}_\alpha$ if it is outside the flipping disk, or it maps to its reflection if it is contained in the flipping disk.  For circle(s) that enter and exit the flipping disk, a choice has to be made.  For example, if $c\subset \Gamma_\alpha$ runs through the flipping disk, then $\tilde{c}$ could be {\em either} the circle of $\tilde{\Gamma}_\alpha$ such that $c|_{S^2\setminus D} = \tilde{c}|_{S^2\setminus D}$ or the circle such that $c|_D = \tilde{c}|_D$. As discussed above, for Cases (1)--(3) of Analysis~\ref{analysis:circles-in-flipping-disk}, either choice gives rise to the same circle.  But in Case (4), a choice has to be made.  Choose the circle that matches the original circle outside of the flipping disk:

\begin{definition}[\bf Correspondence between circles in $\Gamma_\alpha$ and $\tilde{\Gamma}_\alpha$] A circle in $\Gamma_\alpha$ completely outside (or inside) the flipping disk corresponds to same circle (or reflection of that circle) in $\tilde{\Gamma}_\alpha$. A circle $c \subset \Gamma_\alpha$ that enters or exits the flipping disk corresponds to the circle $\tilde{c} \subset \tilde{\Gamma}_\alpha$ such that $$c|_{S^2\setminus D} = \tilde{c}|_{S^2\setminus D}.$$ \label{definition:circle-correspondence}
\end{definition}

It will also be helpful to put a specific ordering on the circles:  trifurcate the set of circles in $\Gamma_\alpha$ (and $\tilde{\Gamma}_\alpha$) into three sets depending on whether they (1) run through the flipping disk, (2) are contained in the flipping disk, or (3) lie completely outside the flipping disk:

\begin{definition}[\bf The trifurcation of circles in $\Gamma_\alpha$ and $\tilde{\Gamma}_\alpha$] For a state $\Gamma_\alpha$ with $k$ circles, trifurcate the set of circles as follows: In Cases (1)--(2) of Analysis~\ref{analysis:circles-in-flipping-disk}, let $\{c_1\}$ of $\Gamma_\alpha$ be the circle that passes through the flipping disk, $\{c_2,\ldots,c_d\}$ be the circles that are contained in the flipping disk, and $\{c_{d+1},\ldots,c_k\}$ lie completely outside the disk.  In Cases (3)--(4), let $\{c_1, c_k\}$ be the circles that pass through the flipping disk, $\{c_2,\ldots, c_d\}$ be the circles that are contained in the flipping disk, and $\{c_{d+1}, \ldots, c_{k-1}\}$ be the circles that lie completely outside.  Trifurcate $\tilde{\Gamma}_\alpha$ similarly using the correspondence described above.
\label{def:correspondence_between_circles}
\end{definition}

Note:  The circle labels $\{c_1, c_k\}$ were chosen instead of $\{c_1,c_2\}$ to make it easier to define $\Delta$ and $m$ in the proofs that follow ($c_k$ will be the circle that is created or merged).

The algebra $V_\alpha$ (and $\tilde{V}_\alpha$) can be redefined in terms of an exterior algebra $\Lambda(W)$ with $\BZ_2$ coefficients for a vector space $W$.   To set up this isomorphism, assume that the state $\Gamma_\alpha$ has $k$ circles (take $k=k_\alpha$ to simplify notation) that have been ordered $c_1,\ldots, c_k$.  Use this order to define $V_\alpha = V^{\ot k}\{|\alpha |\}$, where $V^{\ot k} = V_{c_1}\ot \cdots \ot V_{c_k}$ such that each $V_{c_i}=\BZ_2[x]/ (x^2)$.  The space $V^{\ot k}$ is equivalent to 
\begin{equation}
\BZ_2[x_1,\cdots, x_k]/(x^2_1, x^2_2, \dots, x^2_k)\\
\label{eq:Votk}
\end{equation}
via a map that takes, for example, $1\ot 1\ot x \ot 1\mapsto x_3$ when $k=4$. Here and throughout, the variable $x_i$ corresponds to the circle $c_i$.     
This space, in turn, can be thought of as an exterior algebra with $\BZ_2$ coefficients: 

\begin{lemma}
In the space described in Equation~\ref{eq:Votk}, let $W$ be the subspace given by $W=\langle x_1,x_2,\dots,x_k\rangle$, where the polynomial $x_i$ is associated to the $i$th circle in $\Gamma_\alpha$.  Then $V^{\ot k}$ can be identified with $\Lambda(W)$, i.e.,
$$V^{\ot k} = \Lambda^0(W)\oplus\Lambda^1(W)\oplus \cdots \oplus\Lambda^k(W),$$
where the grading of a generic monomial $x_{i_1}x_{i_2}\cdots x_{i_p}$ is  degree $p$. \label{lemma:space_W} 
\end{lemma}

Clearly, $\dim_{\BZ_2} \Lambda^p(W_\alpha) = C(k, p)$.  The grading on $V^{\ot k}$ and $\Lambda(W)$ are different but compatible:  an element $\omega\in\Lambda^p(W)$ has $q$-grading $k-2p$ in $V^{\ot k}$.  Note that, because of the $\BZ_2$ coefficients, multiplication in $\Lambda(W)$ is commutative.  

The beauty of this formulation is that the maps $S$, $\Delta$, and $m$ can be formalized in terms of polynomial multiplication and division. First, we record  useful formulas for $\Delta$ and $m$, and then build up to the map $S$:

\begin{lemma}
Let $\Gamma$ be a perfect matching graph for a planar trivalent graph $G$ with a perfect matching $M$ such that $n=|M|$.  For $\alpha \in \{0,1\}^n$, let $\Gamma_\alpha$ be a state in the hypercube of states of $\Gamma$ that contains $k$ circles.  The vector space corresponding to $\Gamma_\alpha$ is $V_\alpha = V^{\ot k}\{|\alpha |\}$,  where $V^{\ot k} = \Lambda(W_\alpha)$.
Suppose $\omega \in \Lambda^p(W_\alpha)$ such that $\omega=x_{i_1}x_{i_2}\cdots x_{i_p}$ is a basis element.

\begin{enumerate}
\item Let $\Gamma_{\alpha'}$ be the state resulting from merging circles $c_1$ and $c_k$ in $\Gamma_\alpha$ to get a circle $c_1$ in $\Gamma_{\alpha'}$.  On the level of vector spaces, this corresponds to $m_{1k}:\Lambda(W_\alpha) \ra \Lambda(W_{\alpha'})$.  If $x_k \divides \omega$ (i.e., $x_k$ is a factor of $\omega$), then $$m_{1k}(\omega) = x_1\cdot \left(\frac{\omega}{x_k}\right),$$ and $m_{1k}(\omega)=\omega$ otherwise.  Note: $m_{1k}(\omega)=0$ if both $x_1$ and $x_k$ divide $\omega$.

\item Let $\Gamma_{\alpha'}$ be the state resulting from splitting circle $c_1$ in $\Gamma_\alpha$ into two circles $c_1$ and $c_{k+1}$ in $\Gamma_{\alpha'}$. Then  $\Delta_1:\Lambda(W_\alpha) \ra \Lambda(W_{\alpha'})$ and
$$\Delta_1(\omega) = \omega(x_1+x_{k+1}).$$

\item Let $\Gamma_{\alpha'}$ be the state resulting from adding a double point to a circle $c_1$ in $\Gamma_\alpha$ to get a circle $c_1$ in $\Gamma_{\alpha'}$.  Then $A_1:\Lambda(W_\alpha) \ra \Lambda(W_{\alpha'})$ and $A_1(\omega) = 0$.

\item The $\alpha$-degree of $\omega$ does not change under $\Delta$ or $m$:  $\deg_\alpha(\omega) = k-2p+|\alpha |$, $\deg_{\alpha'}(\Delta(\omega)) = \deg_\alpha(\omega)$, and $\deg_{\alpha'}(m(\omega))=\deg_\alpha(\omega)$.
\end{enumerate} 
The maps $\Delta, m$ and $A$ can be defined on basis elements as above and extended linearly to all of $\Lambda(W_\alpha)$.
\end{lemma}

\begin{proof} Each statement follows easily from the original definitions of the maps.
\end{proof}

The map $S:V_\alpha \ra \tilde{V}_\alpha$ can now be defined for each case in Analysis~\ref{analysis:circles-in-flipping-disk}.  For Cases (1)--(3), the last sentence of the first three cases of Analysis~\ref{analysis:circles-in-flipping-disk}  implies that  $S$ can be taken to be the canonical isomorphism in each of those cases: The circles in $\Gamma_\alpha$ that are completely outside the flipping disk are mapped directly to the same circles in $\tilde{\Gamma}_\alpha$.  The circles in $\Gamma_\alpha$ that are contained in the flipping disk are mapped to their reflections in $\tilde{\Gamma}_\alpha$.  Finally, the circle(s) in $\Gamma_\alpha$ that enter and exit the flipping disk also go to circle(s) in $\tilde{\Gamma}_\alpha$ with the same smoothing sites.  Thus, after choosing a correspondence of circles between $\Gamma_\alpha$ and $\tilde{\Gamma}_\alpha$ (cf. Definition~\ref{definition:circle-correspondence}) and using that identification to define the vector spaces $V_\alpha$ and $\tilde{V}_\alpha$, the map $S:V_\alpha \ra \tilde{V}_\alpha$ can be taken to be the identity map for Cases (1)--(3).  

The definition of $S$ for a Case (4) state $\Gamma_\alpha$ is at the heart of the proof for the $2$-flip move for the cochain complexes (and therefore the heart of the proof of Theorem~\ref{theorem:main-theorem-of-cohomology}).  The reason is that the two circles in $\Gamma_\alpha$ that enter and exit the flipping disk exchange their arcs in the disk when the disk is flipped, which means the corresponding circles in $\tilde{\Gamma}_\alpha$ will have exchanged smoothing sites associated to those arcs.  Meanwhile, smoothing sites of the original circles outside the flipping disk remains the same.  Therefore the maps $S:V_\alpha \ra \tilde{V}_\alpha$  and $S': V_{\alpha'} \ra \tilde{V}_{\alpha'}$ must commute with maps $\del_\eta:V_\alpha \ra V_{\alpha'}$ and $\tilde{\del}_\eta:\tilde{V}_\alpha \ra \tilde{V}_{\alpha'}$ in a way that takes into account that circles are {\em sometimes but not always} splitting off, merging, or adding double points to different circles in $\tilde{\Gamma}_\alpha$ than the original circles in $\Gamma_\alpha$.

The map that works for Case (4) is one that adds ``partial derivatives'' with respect to the circles that lie completely inside the flipping disk.  Define:

\begin{definition}
Let $\omega \in \Lambda^p(W_\alpha)$ be the basis element $\omega=x_{i_1}x_{i_2}\cdots x_{i_p}$.  The {\em partial derivative of $\omega$ with respect to $x_i$} is
$$ \Der{\omega}{x_i} := \left\{\begin{array}{cc}  \frac{\omega}{x_i} & \mbox{ if } x_i \divides \omega\\ 0 & \mbox{ if } x_i \notdivides \omega\end{array}\right.$$
Extend this definition linearly to all elements of $\Lambda^p(W_\alpha)$.
\end{definition}

\begin{remark} One must be careful with this definition and not let the notation of partial derivatives deceive---it is defined only formally.  For example, it is true that this partial derivative satisfies a product rule for basis elements $x_i$: if $\omega \in \Lambda^p(W_\alpha)$ and $\eta\in \Lambda^q(W_\alpha)$, then
\begin{eqnarray}
\Der{(\omega \eta)}{x_i} = \Der{\omega}{x_i}\eta+\omega\Der{\eta}{x_i}.\label{eqn:product-rule}
\end{eqnarray}
However, the reason is partially due to the $\BZ_2$ coefficients: $\Der{(x_1\cdot x_1)}{x_1} = \Der{x_1}{x_1}x_1 +x_1\Der{x_1}{x_1} = x_1+x_1 =0$.
\end{remark}

Note: Equation~\ref{eqn:product-rule} will become important in the proofs below.

\begin{definition}[\bf The definition of $S:V_\alpha \ra \tilde{V}_\alpha$]
Let $\Lambda(W_\alpha)$ and $\Lambda(\tilde{W}_\alpha)$ be defined as in Lemma~\ref{lemma:space_W} using the correspondence set up in Definition~\ref{definition:circle-correspondence}.  For Cases (1)--(3) of Analysis~\ref{analysis:circles-in-flipping-disk}, define $S=Id$.  In Case (4), the spaces can be further written  $$W_\alpha=\langle x_1,x_k\rangle \oplus \langle x_2,\ldots, x_d\rangle \oplus \langle x_{d+1},\ldots, x_{k-1}\rangle \mbox{ \ and \ } \tilde{W}_\alpha=\langle x_1,x_k\rangle \oplus \langle x_2,\ldots, x_d\rangle \oplus \langle x_{d+1},\ldots, x_{k-1}\rangle,$$ where each $x_i$ is associated with the appropriate circle in $\Gamma_\alpha$ or $\tilde{\Gamma}_\alpha$ according to the trifurcation set up in Definition~\ref{def:correspondence_between_circles}. With respect to these bases, define for $\omega \in \Lambda(W_\alpha)$,
$$S(\omega) = \omega + x_1\sum_{a=2}^d \Der{\omega}{x_a} +x_k\sum_{a=2}^d \Der{\omega}{x_a}.$$
\label{lemma:calculations_with_S}
\end{definition}

\begin{lemma}
The map $S:\Lambda(W_\alpha) \ra \Lambda(\tilde{W}_\alpha)$ is an isomorphism.
\end{lemma}

\begin{proof}
Define $\tilde{S}:\Lambda(\tilde{W}_\alpha) \ra \Lambda(W_\alpha)$ by the same map as $S:\Lambda(W_\alpha) \ra \Lambda(\tilde{W}_\alpha)$ depending on the case in  Definition~\ref{lemma:calculations_with_S}.  Then $S\circ \tilde{S} = Id$ and $\tilde{S}\circ S = Id$.
\end{proof}

To get the map  $S:C^{i,j}(\Gamma) \ra C^{i,j}(\tilde{\Gamma})$, write $C^{i,j}(\Gamma) =\oplus \Lambda(W_\alpha)$ for the appropriate states $\Gamma_\alpha$ (and do the same for $C^{i,j}(\tilde{\Gamma}_\alpha)$), and use the $S$ in Definition~\ref{lemma:calculations_with_S} for each $\Lambda(W_\alpha)$ depending upon Cases (1)--(4).  Call this collection of maps $S$ when the context is clear.
\bigskip

\subsection{Proof that $S$ is a chain map} 
Since $S$ is an isomorphism of exterior algebras, to prove Theorem~\ref{Prop:Cohomology-invariant-under-2-flip-move} we need only show that $S$ a cochain map.

\begin{theorem}
The map $S:C^{i,j}(\Gamma) \ra C^{i,j}(\tilde{\Gamma})$ commutes with the differential $\del$.
\label{thm:S-commutes-with-del}
\end{theorem}

\begin{proof}
We show that $S$ commutes with the differential $\del$ by checking that the map commutes for each possible diagram corresponding to maps $\del_\eta:V_\alpha \ra V_{\alpha'}$  and $\tilde{\del}_\eta:\tilde{V}_\alpha \ra \tilde{V}_{\alpha'}$.  That is, given an edge $\eta$ in the hypercube from $\Gamma_\alpha$ to $\Gamma_{\alpha'}$, it is required to show $\tilde{\del}_\eta \circ S = S' \circ \del_\eta$ for $S:V_\alpha \ra\tilde{V}_\alpha$ and $S':V_{\alpha'}\ra\tilde{V}_{\alpha'}$.  Certain maps and cases can be dispensed with immediately: All diagrams involving the map $A$ immediately commute since $A$ is the zero map.  Also, for any two of the first three cases of Analysis~\ref{analysis:circles-in-flipping-disk},  we have that $\Delta \circ S = S' \circ \Delta$ and $m \circ S =S' \circ m $ because $S$ and $S'$ are both the identity map in those cases.  Thus, only diagrams involving  Case 4 together with Cases 1-4 of Analysis~\ref{analysis:circles-in-flipping-disk} need to be checked when $\del_\eta$ (and $\tilde{\del}_\eta$) is $\Delta$ or $m$.  

A careful check of the immersed curves in Analysis~\ref{analysis:circles-in-flipping-disk} shows that there cannot be an edge $\eta$ in the hypercube of $\Gamma$ that goes from a Case (4) state to a Case (3) state, or vice versa.  If an $\eta$ did exist, the smoothing change from a $0$-smoothing to a $1$-smoothing must occur inside the flipping disk and must involve both circles of $\Gamma_\alpha$ and both circles of $\Gamma_{\alpha'}$ that enter and exit the flipping disk.  The $1$-smoothing in this situation would add a double point between the two circles of $\Gamma_\alpha$ to give an odd number of double points between the two circles of $\Gamma_{\alpha'}$, which is impossible because the number of double points between two immersed circles in the plane is always an even number by the Jordan curve theorem.

Therefore, using the cases in Analysis~\ref{analysis:circles-in-flipping-disk}, the following {\em edge types} involving Case (4) for an edge of the hypercube $\Gamma$ are the only ones that need to be considered:

\begin{enumerate}
\item  a Case (1) or Case (2) state that splits off a circle to become a Case (4) state, 

\item  a Case (4) state that splits off a circle resulting in a Case (4) state, 

\item a Case (4) state that merges two circles to become a Case (1) or Case (2) state, or 

\item a Case (4) state that merges two circles to become another Case (4) state. 
\end{enumerate}

The proofs that $S$ commutes with the differential are different for each edge type above  and  are given in Lemmas~\ref{lemma:edge-1-commutes}, ~\ref{lemma:edge-2-commutes}, ~\ref{lemma:edge-3-commutes}, ~\ref{lemma:edge-4-commutes}  below.  The proofs of these lemmas completes the proof of Theorem~\ref{thm:S-commutes-with-del}.
\end{proof}

The  notation and definitions set up in this subsection was specifically created to help prove the next set of lemmas.  To summarize, let $e_1,e_2\in M$ of Analysis~\ref{analysis:flip_moves}.  Let $c_1$ (and in the fourth case of Analysis~\ref{analysis:circles-in-flipping-disk}, $c_k$) be the circle(s) that pass through the $e_1$ and $e_2$ smoothings of $\Gamma_\alpha$.  Throughout, assume that $\Gamma_\alpha$ has $k$ circles.  Then $\tilde{\Gamma}_\alpha$ has $k$ circles, and $\Gamma_{\alpha'}$ and $\tilde{\Gamma}_{\alpha'}$ will both have $k+1$ or will both have $k-1$ circles.  Also, trifurcate the set of circles of $\Gamma_\alpha$ and $\tilde{\Gamma}_\alpha$, and the corresponding circles in $\Gamma_{\alpha'}$ and $\tilde{\Gamma}_{\alpha'}$ respectively, into three sets according to Definition~\ref{def:correspondence_between_circles}. 

Let $\Lambda(W_\alpha)$ be the vector space associated to $\Gamma_\alpha$ as in Lemma~\ref{lemma:space_W}, where each $x_i$ in $W_\alpha=\langle x_1, x_2,\ldots, x_{k}\rangle$ corresponds to the circle $c_i$ in $\Gamma_\alpha$.  Continue to write $x_i$ for the corresponding variable in $\Lambda(W_{\alpha'}), \Lambda(\tilde{W}_\alpha),$ and $\Lambda(\tilde{W}_{\alpha'})$ according to the correspondence given by $\eta$, Definition~\ref{definition:circle-correspondence}, and Definition~\ref{def:correspondence_between_circles}. With these conventions understood, showing that $\tilde{\del}_\eta \circ S = S' \circ \del_\eta$ for the different types of edges $\eta$ listed in Edge Types (1)--(4) above amounts to showing that the following diagram commutes,
\begin{equation}\label{eq:commute-diagram}
\begin{gathered}
\xymatrix{
\Lambda(W_\alpha) \ar[rr]^{\Delta \mbox{ (or $m$)}} \ar[d]_{S} \ar @{} [drr] |{\# ?} & &\Lambda(W_{\alpha'}) \ar[d]^{S'}\\
\Lambda(\tilde{W}_\alpha) \ar[rr]_{\tilde{\Delta} \mbox{ (or $\tilde{m}$)}} && \Lambda(\tilde{W}_{\alpha'})
}
\end{gathered}
\end{equation}
where $\Delta$ is the map for Edge Types (1) and (2) and $m$ is the map for Edge Types (3) and (4).

The lemmas below step through each of the diagrams given by Equation~\ref{eq:commute-diagram} for Edge Types (1)--(4) and show that they commute. Throughout the proofs of the following lemmas,  $\omega\in \Lambda(W_\alpha)$ is always the basis element $\omega=x_{i_1}x_{i_2}\cdots x_{i_p}$.  Also, throughout all  calculations,  $S':\Lambda(W_{\alpha'}) \ra \Lambda(\tilde{W}_{\alpha'})$ is calculated in terms of $\omega$, which started  in $\Lambda(W_\alpha)$.

\begin{lemma}[Edge Type (1) Commutes] \label{lemma:edge-1-commutes}
The diagram in Equation~\ref{eq:commute-diagram} commutes when $\Gamma_\alpha$ is the first or second case and $\Gamma_{\alpha'}$ is the fourth case, i.e., a circle $c_{k+1}$ is split off of circle $c_1$  where both $c_1$ and $c_{k+1}$ both enter and exit the flipping disk.  The differentials are $\Delta_1$ and $\tilde{\Delta}_1$.
\end{lemma}

\begin{proof}
In this case, $\Delta_1(\omega)=\omega(x_1+x_{k+1})$.    Apply $S'$ to get:

\begin{eqnarray*}
S'(\omega(x_1+x_{k+1})) = \omega(x_1+x_{k+1}) + x_1\sum_{a=2}^d\Der{(x_1 \omega)}{x_a} +  x_1\sum_{a=2}^d\Der{(x_{k+1} \omega)}{x_a} + x_{k+1}\sum_{a=2}^d\Der{(x_1 \omega)}{x_a} +  x_{k+1}\sum_{a=2}^d\Der{(x_{k+1} \omega)}{x_a}
\end{eqnarray*}
The first and fourth sums on the right hand side of the equation above are zero by applying the product rule (Equation~\ref{eqn:product-rule}).  For example, for $2\leq a \leq d$, $x_1\Der{(x_1\omega)}{x_a} = x_1 (0+ x_1 \Der{\omega}{x_a})=0$ using the fact that $x_1\cdot x_1=0$. The second and third sums are the same since both are equal to 
$$x_1x_{k+1}\sum_{a=2}^d \Der{\omega}{x_a}$$
using the product rule.  The sum of these  sums is zero modulo two.  Therefore, $S'(\Delta_1(\omega))= \omega(x_1+x_{k+1})$.  Since $S$ is the identity, we have $S'(\Delta_1(\omega))=\omega(x_1+x_{k+1})=\tilde{\Delta}_1(S(\omega))$, and the diagram commutes.
\end{proof}

\begin{lemma}[Edge Type (2) Commutes] \label{lemma:edge-2-commutes}
The diagram in Equation~\ref{eq:commute-diagram} commutes when both $\Gamma_\alpha$ and $\Gamma_{\alpha'}$ are the fourth case, and a circle $c_{k+1}$ is split off of:
\begin{enumerate}

\item circle $c_1$ (or equivalently, $c_k$) where the circle $c_{k+1}$ is completely inside the flipping disk. After performing the $2$-flip, circle $c_{k+1}$ is split off of circle $c_k$ instead. The differentials are $\Delta_1$ and $\tilde{\Delta}_k$.

\item the circle $c_1$ (or equivalently, $c_k$) where the circle $c_{k+1}$ is completely outside the flipping disk.  The differentials are $\Delta_1$ and $\tilde{\Delta}_1$.

\item the circle $c_\ell$ for $2\leq \ell \leq d$ (the circle $c_\ell$ is completely inside the flipping disk).  Here the circle $c_{k+1}$ must also be inside the flipping disk.  Hence, the differentials are $\Delta_\ell$ and $\tilde{\Delta}_\ell$.

\item the circle $c_\ell$ for $d+1\leq \ell \leq k-1$ (the circle $c_\ell$ is completely outside of the flipping disk).  Here the circle $c_{k+1}$ must also be outside the flipping disk. The differentials are $\Delta_\ell$ and $\tilde{\Delta}_\ell$.
\end{enumerate}
\end{lemma}

\begin{proof} We prove the third statement of the lemma and leave the other  statements to the reader.  Without loss of generality, assume $\ell=2$.  Then circle $c_{k+1}$ is split off of circle $c_2$ where $c_2$ is inside the flipping disk. Going across and then down  in Equation~\ref{eq:commute-diagram} gives:
\begin{eqnarray}
  S'(\Delta_2(\omega)) & =& S'(x_2\omega + x_{k+1}\omega) \nonumber \\
&=& x_2\omega  +x_1 \sum_{a=2}^d \Der{(x_2\omega)}{x_a}+ x_1\Der{(x_2\omega)}{x_{k+1}} + x_k \sum_{a=2}^d \Der{(x_2\omega)}{x_a}+ x_k\Der{(x_2\omega)}{x_{k+1}} \label{eq:2ndline}\\ 
&& \hspace{1cm} +x_{k+1}\omega +x_1\sum_{a=2}^d \Der{(x_{k+1}\omega)}{x_a}+ x_1\Der{(x_{k+1}\omega)}{x_{k+1}}+ x_k \sum_{a=2}^d \Der{(x_{k+1}\omega)}{x_a}+x_k\Der{(x_{k+1}\omega)}{x_{k+1}} \label{eq:3rdline}
\end{eqnarray}
The derivatives with respect to $x_{k+1}$ in the third and fifth terms of both Equation~\ref{eq:2ndline} and Equation~\ref{eq:3rdline} are because circle $c_{k+1}$ is in the flipping disk and therefore these terms have to be included by the definition of $S'$. The derivatives with respect to $x_{k+1}$ in Equation~\ref{eq:2ndline} are zero  because $x_2\omega$ cannot have a $x_{k+1}$ factor.  However, the sum of the third and fifth terms in Equation~\ref{eq:3rdline} is equal to $x_1\omega +x_{k+1}\omega$ by the product rule (Equation~\ref{eqn:product-rule}).  Also by the product rule, for all terms  in the sums of Equation~\ref{eq:3rdline}, $\Der{(x_{k+1}\omega)}{x_a}=x_{k+1}\Der{\omega}{x_a}$ since $a\not=k+1$.  The same is true for terms in the sums of Equation~\ref{eq:2ndline} except in the case of $a=2$.  For those two terms, $x_1\Der{(x_2\omega)}{x_2}=x_1\omega+x_1x_2\Der{\omega}{x_2}$ and $x_k\Der{(x_2\omega)}{x_2}=x_k\omega+x_k x_2\Der{\omega}{x_2}$ by the product rule. The $x_1\omega$ and $x_k\omega$ terms in the equations of the previous sentence add together with the sum $x_1\omega+x_k\omega$ of the third and fifth terms in Equation~\ref{eq:3rdline} to get zero modulo two. What is left are terms that include either $x_2$ or $x_{k+1}$ as factors.  Factoring the binomial $(x_2+x_{k+1})$ out of this expression gives

$$ S'(\Delta_2(\omega))= \left(\omega + x_1 \sum_{a=2}^d \Der{\omega}{x_a}+ x_k \sum_{a=2}^d \Der{\omega}{x_a}\right)(x_2+x_{k+1}),$$
but the right hand side of this equation is  $\tilde{\Delta}_2(S(\omega))$. Hence, $S'\circ \Delta_2 = \tilde{\Delta}_2\circ S$, and the diagram commutes.
\end{proof}

\begin{lemma}[Edge Type (3) Commutes] \label{lemma:edge-3-commutes}
The diagram in Equation~\ref{eq:commute-diagram} commutes when $\Gamma_\alpha$ is the fourth case and $\Gamma_{\alpha'}$ is the first or second case, i.e., circles $c_1$ and $c_k$ in $\Gamma_\alpha$ are merged to become one circle $c_1$ in $\Gamma_{\alpha'}$.  The differentials are $m_{1k}$ and $\tilde{m}_{1k}$.
\end{lemma}

\begin{proof}
The differential $m_{1k}$ is defined as follows: If $x_k\notdivides \omega$, then $m_{1k}(\omega)=\omega$.  If $x_k\divides \omega$, then $m_{1k}(\omega) = x_1(\frac{\omega}{x_k})$. A similar statement holds for $\tilde{m}_{1k}$.

First, suppose $x_k \divides \omega$ and write $\omega = x_k\eta$. Then 
\begin{eqnarray}
\tilde{m}_{1k}(S(\omega)) & = & \tilde{m}_{1k}\left(x_k \eta + x_1\sum_{a=2}^d\Der{(x_k\eta)}{x_a} + x_k\sum_{a=2}^d\Der{(x_k\eta)}{x_a}\right) \label{eq:proof-case-3-2} \\
&=& \tilde{m}_{1k}\left(x_k \eta + x_1 x_k\sum_{a=2}^d\Der{\eta}{x_a} + x_k x_k\sum_{a=2}^d\Der{\eta}{x_a}\right) \nonumber\\
&=& x_1\eta. \nonumber
\end{eqnarray}

The second line of Equation~\ref{eq:proof-case-3-2} follows from the first by the product rule (Equation~\ref{eqn:product-rule}). The third term on the second line is zero because $x_k x_k=0$.  Similarly, after applying $\tilde{m}_{1k}$, the second term on the second line has a $x_1 x_1$ factor and is also zero. 

On the other hand, $m_{1k}(\omega) = x_1\eta$, and applying $S'=Id$ to this  gives the last line of Equation~\ref{eq:proof-case-3-2}.  Hence, $S'\circ m_{1k} = \tilde{m}_{1k} \circ S$.

If $x_k \notdivides \omega$, then $x_k\notdivides \Der{\omega}{x_a}$ for all $2 \leq a \leq d<k$ when $\Der{\omega}{x_a}$ is nonzero, and both $m_{1k}$ and $\tilde{m}_{1k}$ map $\omega$ to $\omega$ and $\Der{\omega}{x_a}$ to $\Der{\omega}{x_a}$.  Thus, $$\tilde{m}_{1k}(S(\omega))= \tilde{m}_{1k}\left(\omega + x_1\sum_{a=2}^d\Der{\omega}{x_a} +x_k\sum_{a=2}^d\Der{\omega}{x_a}\right) =\omega + x_1\sum_{a=2}^d\Der{\omega}{x_a} +x_1\sum_{a=2}^d\Der{\omega}{x_a} = \omega = S'(m_{1k}(\omega)).$$

For both cases,  $S'\circ m_{12} = \tilde{m}_{k2} \circ S$, and the diagram commutes.
\end{proof}

\begin{lemma}[Edge Type (4) Commutes] \label{lemma:edge-4-commutes}
The diagram in Equation~\ref{eq:commute-diagram} commutes when both $\Gamma_\alpha$ and $\Gamma_{\alpha'}$ are the fourth case, and merge:
\begin{enumerate}
\item circle $c_1$ (or equivalently, $c_k$) with circle $c_\ell$, $2\leq \ell\leq d$, to get circle $c_1$ (circle $c_\ell$ is completely inside the flipping disk).   After performing the $2$-flip, circle $c_\ell$ will be merged with $c_k$ instead.  Hence, the differentials are $m_{1\ell}$ and $\tilde{m}_{k\ell}$.

\item circle $c_1$ (or equivalently, $c_k$) with circle $c_\ell$, $d+1\leq \ell\leq k-1$, to get circle $c_1$ (circle $c_\ell$ is completely outside the flipping disk).  The differentials are $m_{1\ell}$ and $\tilde{m}_{1\ell}$.

\item circle $c_\ell$, $2\leq\ell \leq d$, with circle $c_{\ell'}$, $2\leq \ell' \leq d$, to get circle $c_\ell$ (both circles $c_\ell$ and $c_{\ell'}$ are completely inside the flipping disk).  The differentials are $m_{\ell\ell'}$ and $\tilde{m}_{\ell\ell'}$.

\item circle  $c_\ell$, $d+1\leq\ell \leq k-1$, with circle $c_{\ell'}$, $d+1\leq \ell' \leq k-1$, to get circle $c_\ell$ (both circles $c_\ell$ and $c_{\ell'}$ are completely outside the flipping disk).  The differentials are $m_{\ell\ell'}$ and $\tilde{m}_{\ell\ell'}$.
\end{enumerate}
\end{lemma}

\begin{proof}
We prove the first statement and leave the other statements to the reader.  Without loss of generality, assume $\ell=2$. For the differential $m_{12}$, circles $c_1$ and $c_2$ are merged to become one circle $c_1$.  After the $2$-flip move, since $c_2$ is inside the flipping disk, the circle $c_2$ will be merged with circle $c_k$ instead to get circle $c_k$.  The differential  $m_{12}$ is defined as follows: If $x_2\notdivides \omega$, then $m_{12}(\omega)=\omega$.  If $x_2\divides \omega$, then $m_{12}(\omega) = x_1(\frac{\omega}{x_2})$. A similar statement holds for $\tilde{m}_{k2}$.

First, suppose $x_2 \divides \omega$ and write $\omega = x_2\eta$. Then

\begin{eqnarray}
\tilde{m}_{k2}(S(\omega)) & = & \tilde{m}_{k2}\left(x_2 \eta + x_1\sum_{a=2}^d\Der{(x_2\eta)}{x_a} + x_k\sum_{a=2}^d\Der{(x_2\eta)}{x_a}\right) \label{eq:proof-case-4-2} \\
&=& \tilde{m}_{k2}\left(x_2 \eta + x_1 \left(\eta + x_2\sum_{a=2}^d\Der{\eta}{x_a}\right) + x_k \left(\eta+x_2\sum_{a=2}^d\Der{\eta}{x_a}\right)\right) \nonumber\\
&=& x_k\eta +x_1\eta +x_1 x_k \sum_{a=2}^d\Der{\eta}{x_a} + x_k\eta + x_k x_k\sum_{a=2}^d\Der{\eta}{x_a} \nonumber\\
&=& x_1\eta+x_k \sum_{a=2}^d\Der{(x_1\eta)}{x_a}. \nonumber
\end{eqnarray}

The second line of Equation~\ref{eq:proof-case-4-2} follows from the first by the product rule (Equation~\ref{eqn:product-rule}). Note: the terms $x_1\eta$ and $x_k\eta$  in the second line are due to taking the derivative with respect to $x_2$, i.e., the $a=2$ term.  After applying the map $\tilde{m}_{k2}$, the last term of third line  is zero because $x_k\cdot x_k=0$, and the two $x_k\eta$ terms sum to zero modulo two. Using the product rule again on the remaining nonzero sum of the third line results in the last line.

On the other hand, $m_{12}(\omega) = x_1\eta$, and applying $S'$ to this together with the product rule gives the last line of Equation~\ref{eq:proof-case-4-2}.  Hence, $S'\circ m_{1\ell} = \tilde{m}_{1\ell} \circ S$.

If $x_2 \notdivides \omega$, then $x_2\notdivides \Der{\omega}{x_a}$ for all $2 \leq a \leq d$ when $\Der{\omega}{x_a}$ is nonzero, and both $m_{12}$ and $\tilde{m}_{k2}$ map $\omega$ to $\omega$ and $\Der{\omega}{x_a}$ to $\Der{\omega}{x_a}$.  Thus, $$S'(m_{12}(\omega))= \omega + x_1\sum_{a=2}^d\Der{\omega}{x_a} +x_k\sum_{a=2}^d\Der{\omega}{x_a} = \tilde{m}_{k2}(S(\omega)).$$

For both cases,  $S'\circ m_{12} = \tilde{m}_{k2} \circ S$, and the diagram commutes.
\end{proof}

\subsection{Is there a category theoretic approach to the proof of Theorem~\ref{theorem:main-theorem-of-cohomology}?}
In examining the proofs of the previous four lemmas together, note that each diagram commutes based upon a different algebra calculation.  Thus, there does not seem to be a ``universal'' algebra calculation that works for all cases at once---each case depends upon whether the circle(s) being split off or merged are completely inside, outside, or running through the flipping disk.  Still there are important commonalities between the proofs: the product rule is consistently applied to the partial derivatives to introduce or cancel (modulo two) exactly the terms needed  to make the diagrams commute each time.  Therefore, one can hope for a more encompassing argument in the future.  

Is there a category theoretic approach to this proof? There are examples of similar ideas in the literature based upon webs and foams.  In \cite{KM3}, Kronheimer and Mrowka  show how to define an instanton homology for webs, i.e., trivalent graphs embedded in $\BR^3$ (see also \cite{KM2,KM3, KM1}).  Based upon these results, and closer to the present paper in terms of working with combinatorial structures, is the work of Khovanov and Robert \cite{KR} (see also \cite{RW}). In both papers, the authors define or suggest the category of {\fontfamily{phv}\selectfont Foams}  with webs as objects and isomorphism classes of foams with boundary as morphisms.  These papers certainly lend support to a conjectural relationship between the cohomology of this paper and webs and foams, for example.  

It is interesting to note that the instanton homologies and foam evaluations described in these paper depends on the topology of how the webs are embedded into $\BR^3$, at least for the definition of the invariants. The invariants of this paper, on the other hand, depend only on the graph and a choice of perfect matching, and later (see Section~\ref{section:extending-to-graphs}), only the graph itself. There are also questions about what role the perfect matching edges would play in terms of foams and the meaning of the $A$ map in the context of webs and foams. Both perfect matchings and the $A$ map seem to fit in more naturally with TQFT-like theories of virtual knots than the TQFT approaches above (cf. \cite{BKR}). But this only makes the search for a relationship more enticing, not less. The map $S$ and its use in the proof of the lemmas above hints at something new and interesting in category theory yet undiscovered. 

\section{Examples of cohomology calculations}
\label{section:Examples}

In this section we calculate the cohomology of a few well known families of planar trivalent graphs. These examples where chosen to highlight different properties and behavior of the cohomology.  For instance, examples are presented where the cohomology of the pair $(G,M)$ has strictly more information than the $2$-factor polynomial, showing that the cohomology is stronger than the $2$-factor polynomial as an invariant.  Many of the results established in this section rely upon theorems from Subsection~\ref{subsec:cochain-inv}.

The first example, the $3$-prism, shows that the cohomology is a finer invariant than the $2$-factor polynomial.  Let $P_3$ be the $3$-prism together with each of its perfect matchings (cf. Figure~\ref{fig:P_3}).    

\begin{figure}[h]
\psfrag{A}{\!\!\!\!$(P_3, L)$}
\psfrag{B}{\!\!\!\!$(P_3, C)$}
\includegraphics[scale=.7]{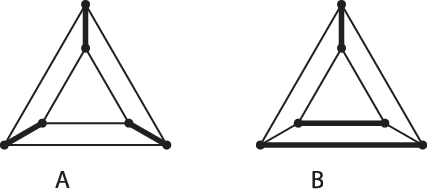}
\caption{The $P_3$ graph with its circular ladder matching $L$  and candlestick matching $C$.}
\label{fig:P_3}
\end{figure}

As an illustrative example, the hypercube of states for the pair $(P_3,L)$ is presented in Figure~\ref{fig:3-prism-hypercube-of-states}.  Note that the number of circles drops by one for the states as the cohomological degree is increased by one---up until the last degree. 

\begin{figure}[H]
\psfragscanon
\psfrag{-1-1-1-1}{\small $(\!-\!1\!,\!-\!1\!,\!-\!1\!,\!-1\!)$}\psfrag{a}{$m$}
\psfrag{000}{\small $(0,0,0)$}
\psfrag{100}{\small $(1,0,0)$}
\psfrag{010}{\small $(0,1,0)$}
\psfrag{001}{\small $(0,0,1)$}
\psfrag{110}{\small $(1,1,0)$}
\psfrag{101}{\small $(1,0,1)$}
\psfrag{011}{\small $(0,1,1)$}
\psfrag{111}{\small $(1,1,1)$}
\psfrag{b}{$A$}\psfrag{B}{$b'$}
\includegraphics[scale=.1]{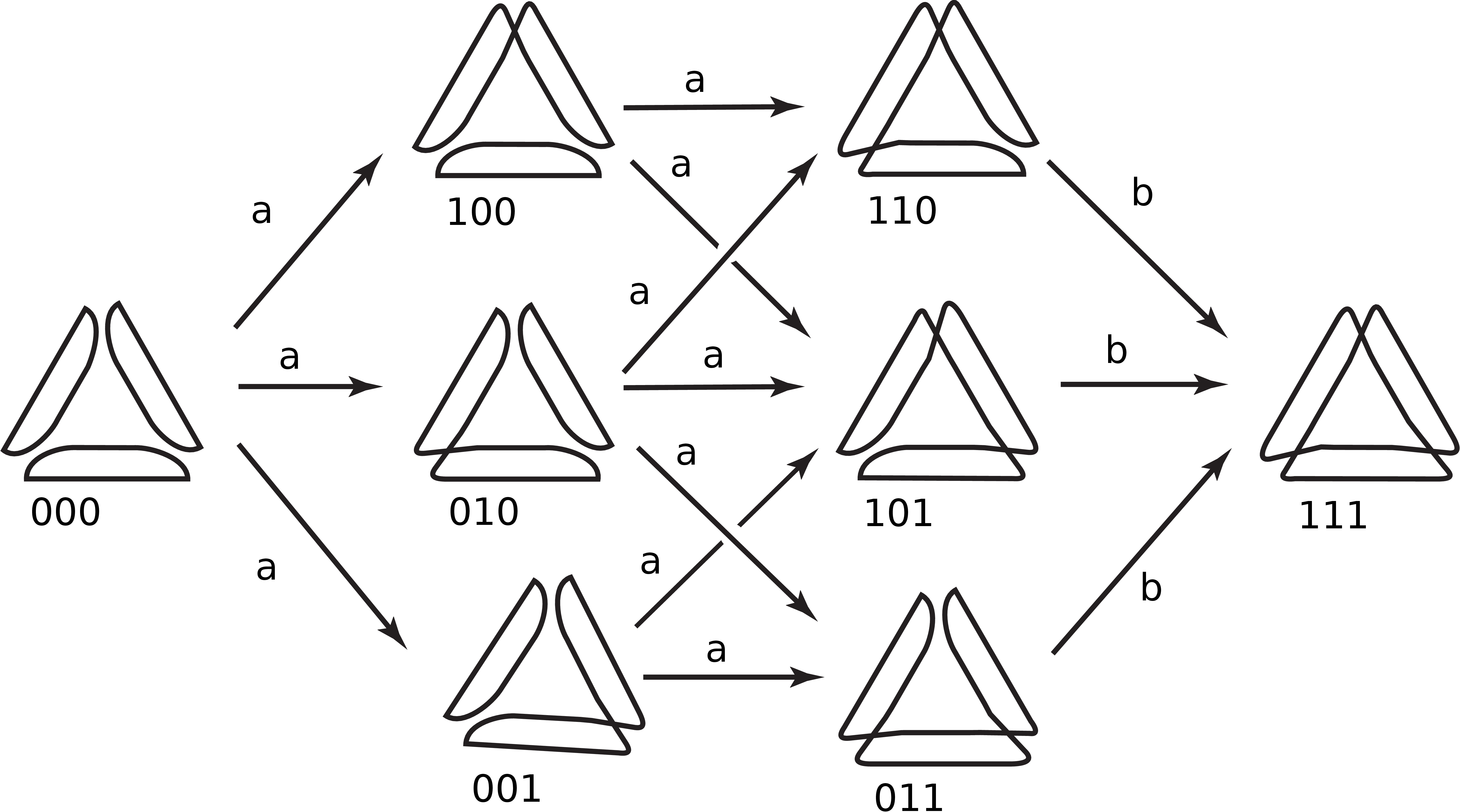}
\caption{Hypercube of states for the graph $P_3$ with the circular ladder $L$ perfect matching.}
\label{fig:3-prism-hypercube-of-states}
\end{figure}

Using the chain complex for $(P_3,L)$ generated from the hypercube in Figure~\ref{fig:3-prism-hypercube-of-states}, the cohomology of $H^{i,j}(P_3,L)$   is presented in Table~\ref{table:cohomology-of-P3-L}.  The cohomology of $(P_3,C)$ is presented in Table~\ref{table:cohomology-of-P3-C}.

\begin{minipage}{3in}
\begin{table}[H]
\renewcommand{\arraystretch}{1}
\caption{$H^{i,j}(P_3,L)$.}
\centering
\begin{tabular}{| c || c | c | c | c | c |}
\hline
$4$ & & & & $\BZ_2$ \\ \hline
$3$ & & & $\BZ_2$ & \mbox{}  \\ \hline
$2$ & & &  & $\BZ_2$ \\ \hline
$1$ & & $\BZ_2$ & $\BZ_2$ & \\ \hline
$0$ &  &   &      & \\ \hline
$-1$ & $\BZ_2$ & $\BZ_2$ &  & \\ \hline
$-2$ &&&&\\ \hline
$-3$ & $\BZ_2$ &&&\\ \hhline{|=||=|=|=|=|=|}
\diagbox[dir=SW]{$j$}{$i$} & $0$& $1$ & $2$ & $3$\\ \hline
\end{tabular} \label{table:cohomology-of-P3-L}
\end{table}
\vspace{.2cm}
\end{minipage}\begin{minipage}{3in}
\begin{table}[H]
\renewcommand{\arraystretch}{1}
\caption{$H^{i,j}(P_3,C)$.}
\centering
\begin{tabular}{| c || c | c | c | c | c |}
\hline
$3$ & & & $\BZ_2$  \mbox{}  \\ \hline
$2$ & & &   \\ \hline
$1$ & & $\BZ_2$ & $\BZ_2$ \\ \hline
$0$ & $\BZ_2$ &   &       \\ \hline
$-1$ & & $\BZ_2$   &   \\ \hline
$-2$ & $\BZ_2$ &&\\ \hhline{|=||=|=|=|=|}
\diagbox[dir=SW]{$j$}{$i$} & $0$& $1$ & $2$ \\ \hline
\end{tabular} \label{table:cohomology-of-P3-C}
\end{table}
\end{minipage}


The $2$-factor polynomials are computed by taking the graded Euler characteristic of the cohomology above:
\begin{eqnarray*}
\tfp{P_3}{L}(q) & = & q^{-3} - q^2 +q^3 -q^4, \mbox{ or}\\
\tfp{P_3}{C}(q) &=& q^{-2}-q^{-1}+1+q^3.
\end{eqnarray*}

By comparing the cohomology with the $2$-factors for each perfect matching, one can see that the cohomology groups are stronger invariants than the polynomial invariants.  In this example, both types of invariants distinguish the two perfect matchings on $P_3$, but the stronger cohomology invariant is likely  able to distinguish two perfect matchings on the same graph with the same $2$-factor polynomials (cf. Conjecture~\ref{conj:same-2-factor-different-cohomology}.)

These perfect matchings also highlight results from other theorems in this paper.  For example, the ladder perfect matching $L$ is an odd perfect matching (subtracting the ladder gives two odd 3-cycles).  Consequently, evaluating the $2$-factor polynomial for $L$ at $1$ should be $0$, as it is.  The candlestick perfect matching $C$ is an even perfect matching. Evaluating the $2$-factor polynomial at $1$ is $2$, which captures the fact that there are two $2$-factors that contain $C$ as in Theorem~\ref{conj:f_size_conjecture}.

While the cochain complex is actually an invariant (cf. Theorem~\ref{theorem:main-theorem-of-cohomology}), the cohomology, like all homology theories, can capture salient information about the graph not directly observed in the cochain complex.  To show that this is possible, we present both extremes, i.e., examples of graphs and perfect matchings where the cochain complex is equal to the  cohomology (i.e., $\del = 0$ for all $i$ and $j$) and examples with large nonzero cochain complexes that all collapse to a cohomology isomorphic to $\BZ_2\oplus \BZ_2$ up to a quantum grading shift.

Define the {\em $m$-th dumbbell graph} to be  $D_m$ together with perfect matchings $M_m$ as shown in  the figure:
\begin{figure}[H]
\psfrag{A}{$\ldots m$}\psfrag{B}{$\ldots m$}
\psfrag{C}{$(\theta_m, M_{\theta})$}\psfrag{D}{$(D_m, M_m)$}
\includegraphics[scale=0.2]{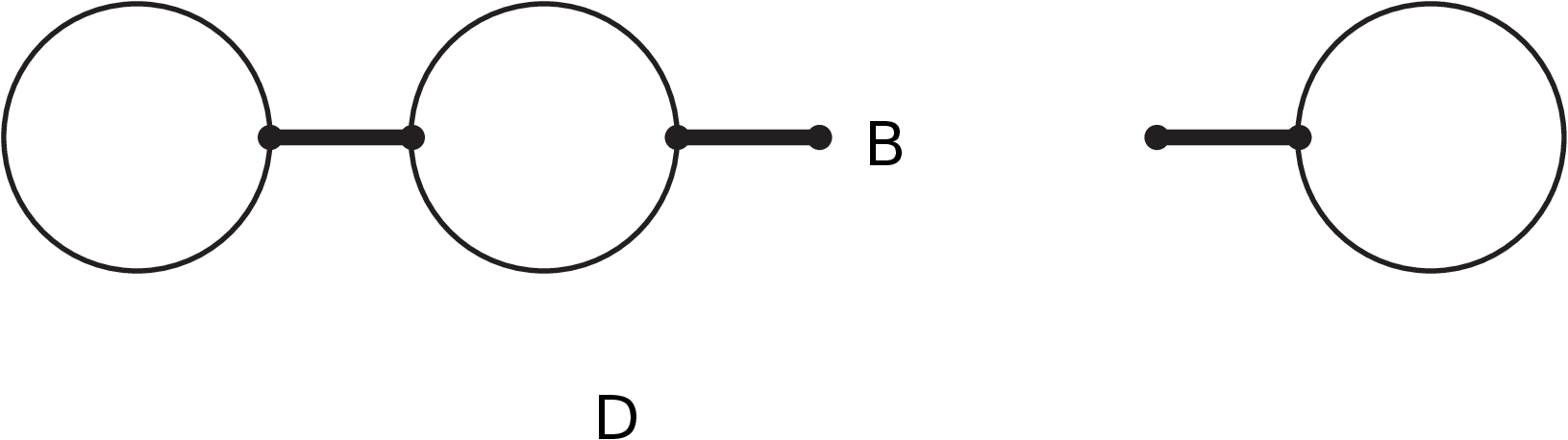}
\label{fig:theta_and_dumbbell}
\end{figure}
\noindent  Every cochain group ($2^{m+1}$ in all) of $m$-th dumbbell graph $D_m$ is also a nontrivial cohomology group.  This is because every map corresponding to an edge in the hypercube of states of $D_m$ is the $A$ map.  (In fact, the maps will be $A$ maps for any ``tree of loops'' created by chaining and branching loops together using the construction described above Proposition~\ref{theorem:cohomology-of-a-loop}.) Therefore,

\begin{theorem}
For the $m$-dumbbell graph $D_m$ with perfect matching $M_m$, 
$$H^{i,j}(D_m,M_m) =\bigoplus_{k=1}^{C(m,i)} \BZ_2$$
for $0\leq i \leq m$, and  $j=i-1$ or $j=i+1$, and $0$ otherwise.
\end{theorem}

The proof of this theorem follows from a repeated application of Proposition~\ref{theorem:cohomology-of-a-loop}.  

Next, we compare the cohomology of $D_m$ to the $m$-theta graph $\theta_m$  with its perfect matching $M_\theta$ (cf. Subsection~\ref{subsection:distinguishing-pm-of-a-graph}).  The $m$-theta graph is the other extreme. The hypercube for $(\theta_m, M_\theta)$ consists of $C(m,i)$ number of states for each homology grading $i$, and each state in that grading has $m+1-i$ circles.  Thus, the $q$-dimension of the cochain groups corresponding to each state in grading $i$ is $(q+q^{-1})^{(m+1-i)}$.  Therefore the number and dimension of non-zero cochain groups quickly grows as $m$ grows.  However, almost all of the non-zero cochain groups give rise to trivial cohomology groups:

\begin{theorem}
For the $m$-theta graph $\theta_m$ with perfect matching $M_\theta$, 
$$H^{0,-1-m}(\theta_m,M_\theta) = \BZ_2 \mbox{ \ \ and \ \ } H^{0,1-m}(\theta_m,M_\theta) = \BZ_2,$$
and $0$ otherwise.
\end{theorem}

The proof of this theorem follows from a repeated application of Proposition~\ref{theorem:add-perfect-matching-edge}.  

Even though the cochain complex for $(\theta_m,M_\theta)$ is large, the cohomology says that in some sense all $m$-theta graphs are ``equivalent'' to the ungraph ($m=0$) up to a normalization of the quantum grading.  This is one of the realizations that lead to the paper of Kauffman, Rushworth, and the author \cite{BKR} and its equivalence relationships on trivalent ribbon graphs.

Finally, notice that computing the cohomology is entirely straightforward and can be done easily for any reasonably-sized graph or reasonable family of graphs (like the families presented in this section). This calculation can be automated with a computer program.  In fact, it is possible to convert a computer program already written by the author, Heather Dye, Aaron Kaestner, Lou Kauffman, and Ben McCarty for virtual links into a program that can compute the cohomology of this paper (cf. \cite{BKM}). This program for planar trivalent graphs will be made available in the near future.

\section{Polynomial and cohomology invariants of planar trivalent graphs}\label{section:extending-to-graphs}

At this point, the reader may think that having to choose a perfect matching for the graph $G$ is too restrictive.  That is, while the polynomial and cohomology invariants introduced in this paper are powerful tools for describing properties  of and distinguishing perfect matchings of a graph, often mathematicians are interested in invariants of just  the graph itself.  In this section, we address this issue by showing that the following two  questions have positive answers:

\begin{enumerate}
\item Can the cohomology theory in this paper be extended in a way to define an invariant of the graph?  
\item The loop value of the four color polynomial is $3$ as in the Penrose Formula (cf. Definition~\ref{defn:four-color-poly}) while the loop value for the $2$-factor polynomial is $2$ (substitute $q=1$ into $q+q^{-1}$).  The Penrose Formula counts all $3$-edge colors directly while the $2$-factor counts $3$-edge colorings that have a fixed color on the perfect matching edges.  Is there a way to build a polynomial/cohomology theory out of the chain complex defined in this paper that  also counts {\em all} $3$-edge colorings like the Penrose Formula?
\end{enumerate}

These two questions are related and show the versatility of working with graph and perfect matching pairs, at least initially.  The answers to these questions also shows that one can choose a  loop value of $2$ or $3$. Both lead to theories that count the number of $3$-edge colorings of a planar trivalent graph.

The first question is addressed by showing that all trivalent plane graphs have an associated perfect matching graph called  the ``blow-up.''  The blow-up has a canonically defined perfect matching.

\begin{definition}
Let $G$ be a planar trivalent graph and $\Gamma$ be a plane graph of $G$.  Define the {\em blow-up of $\Gamma$}, denoted $\Gamma^\flat$, to be the perfect matching graph given by replacing every vertex of $\Gamma$ with a circle as in
\begin{center}
\includegraphics[scale=0.09]{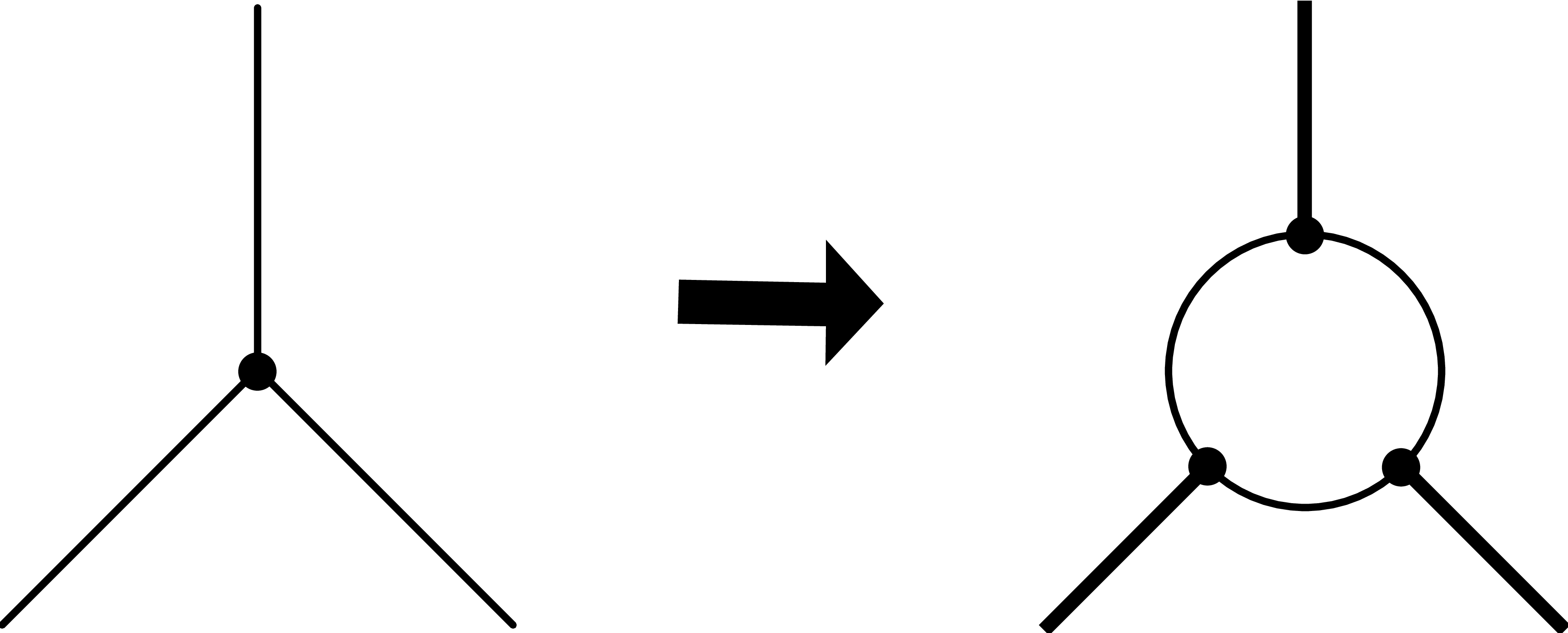}
\end{center}
with perfect matching given by the original edges of $\Gamma$. If $E$ is the set of edges of $G$, then the resulting perfect matching graph can be identified as $(\Gamma^\flat, E)$.  \label{def:blow-up-of-a-graph}
\end{definition}

The history of this idea can be traced back to ``patches'' described by Kempe \cite{Kempe} in his attempted-proof of the four color theorem and, slightly later, to Tait \cite{Tait} who used the blow-up of a general plane graph to show that only trivalent plane graphs need be considered in proving the four color theorem. 

\begin{remark}
Note that the all $0$-smoothing state of $\Gamma^\flat$ represents the set of faces of $\Gamma$. 
\end{remark}

A flip move on $\Gamma$ induces a corresponding flip move on $\Gamma^\flat$ as follows: A flip move on $\Gamma$ is defined for any flipping disk whose boundary intersects the interior of zero, one, or two edges of $\Gamma$ as in Theorem~\ref{thm:graph-iso}.  The boundary of the corresponding disk on $\Gamma^\flat$ will intersect in the same number of points, but now the boundary will always intersect perfect matching edges.  Let $(G^\flat,E)$ denote the equivalence class of  perfect matching graphs under such moves (cf. Theorem~\ref{thm:graph-iso} and Theorem~\ref{thm:perfect-matching-iso}). The polynomial and cohomology invariants of this paper are invariant under this set of restricted flip moves on $\Gamma^\flat$, giving the following invariants of $G$:  

\begin{theorem}
Let $G$ be a planar trivalent graph, $E$ be the set of edges of $G$, and $\Gamma$ be a plane graph of $G$.  Define the {\em $2$-factor polynomial of $G$} to be $\tfb{G}:=\tfp{\Gamma^\flat}{E}$ and the {\em cohomology of $G$} to be $H^{*,*}(G):=H^{*,*}(\Gamma^\flat,E)$. Then the polynomial and cohomology are invariants of the graph $G$.  Furthermore, the $2$-factor polynomial of $G$ is the graded Euler characteristic of its cohomology. \label{thm:cohomology-of-a-graph}
\end{theorem}

Using the blow-up turns any polynomial described in Theorem~\ref{theorem:general_poly_invariants} into an invariant of planar graphs.  For example, the four-color polynomial of a planar graph $G$, $P_4(G):=P_4(G^\flat,E)$, is an invariant of $G$. It continues to count the number of $3$-edge colorings of $G$ when evaluated at one, i.e., $P_4(G)(1)=[G]$. (See the definition of the Penrose Formula $[G]$ above Definition~\ref{defn:four-color-poly}). This is because the $3$-edge colorings at a vertex are in one-to-one correspondence with the $3$-edge colorings of the blow-up of that vertex.

Unlike the four-color polynomial of $G$, the $2$-factor polynomial of $G$ is  zero when evaluated at one: the blow-up creates an odd cycle for each vertex, and by Theorem~\ref{thm:existence_of_even_perfect_matching},  $\tfb{G}(1)=0$.  This zero is not a problem, however.  It turns out that the cochain complex for $(G^\flat,E)$ is balanced between cochain ``subcomplexes'' whose Euler characteristics each report the number of $3$-edge colorings, but with opposite signs. Summing up over these Euler characteristics is zero, which is what $\tfb{G}(1)=0$ is reporting.  Isolating one of these cochain subcomplexes answers the second question above. To motivate how to do this, we need to briefly describe Roger Penrose's seminal work on abstract tensor systems.

In \cite{Penrose}, Roger Penrose derived the Penrose formula $[G]$ from a regular Cartesian abstract tensor system with dimension $\nu=3$, i.e., with loop value 3 (cf. pages 233-234 of \cite{Penrose} where $\nu$ is defined).  He then showed how to translate this count into a negative dimensional abstract tensor system with dimension $\nu=-2$.  This ``binor system'' can be extended to a new polynomial bracket in the language of this paper, which we will call the {\em binor polynomial}, $\langle G\rangle_{binor}$, defined as follows: on the blow-up of the trivalent plane graph, use the bracket defined by $\langle \PMEdgeDiag \rangle = \langle \IIDiag  \rangle  - q \langle \XDiag \rangle$ and $\langle \bigcirc \rangle = -q^{-1}-q$ in Theorem~\ref{theorem:general_poly_invariants}.  One can then recover Penrose's formula from the binor polynomial by evaluating it at one: 

\begin{proposition}[See page 238 of \cite{Penrose}] \label{prop:binor-polynomial}
The binor polynomial of a trivalent graph $G$ satisfies:
$$[G]= \left(-\frac14\right)^{\frac12|V|}\langle G \rangle_{binor}(1),$$
where $|V|$ is the number of vertices of $G$.
\end{proposition} 

Finally, he reinterpreted the binor system on the edges of the blow-up of the graph into a special vertex formula on the vertices of the original graph. This vertex formula leads to an abstract tensor system with dimension $\nu=2$ (loop value 2) on the blow-up of the graph.  The formula he defines continues to count the $3$-edge colorings of the original planar graph (cf. pages 239--240 of \cite{Penrose}).

The vertex formula can also be extended to a new polynomial bracket  on the vertices using the language of this paper.  Let $\Gamma$ be a  planar graph of a trivalent graph $G(V,E)$ and let $(\Gamma^\flat, E)$ be the blow-up of $\Gamma$ with its canonical perfect matching $E$.  Let $\Gamma_\triangle$ be the all $0$-smoothing of $(\Gamma^\flat, E)$ where a triangle has been placed in the region whenever three circles are adjacent (see the left-hand side picture of Equation~\ref{eq:vertex-bracket}).  These adjacencies occur at each of the vertices in the original $\Gamma$.  The {\em vertex bracket $\langle \Gamma_{\triangle} \rangle_{v}$} on $\Gamma_\triangle$ is characterized by:

\begin{eqnarray}
\bigg\langle \vertexbracketvertex \bigg\rangle_{\! v} &=&  \bigg\langle \vertexbracketzero \bigg\rangle_{\! v} \ - \ q^3 \bigg\langle\vertexbracketone \bigg\rangle_{\! v} \label{eq:vertex-bracket}\\ 
\bigg\langle \bigcirc  \bigg\rangle_{\! v} & = & q+q^{-1}  \label{eq:vertex-loop-value}\\
\bigg\langle \Gamma_1 \sqcup \Gamma_2 \bigg\rangle_{\! v}&=& \bigg\langle \Gamma_1 \bigg\rangle_{\! v} \cdot \bigg\langle \Gamma_2 \bigg\rangle_{\! v}\label{eq:vertex-disjoint-union}
\end{eqnarray}

It is instructive to calculate the vertex bracket (and hypercube generated by it) for the theta graph $\theta$. First, the blow-up of $\theta$ and $\Gamma_\triangle$ are:

\begin{center}
\psfragscanon
\psfrag{b}{$\flat$}\psfrag{=}{$=$}
\psfrag{G}{$\Gamma_\triangle  \ = $}
\includegraphics[scale=.5]{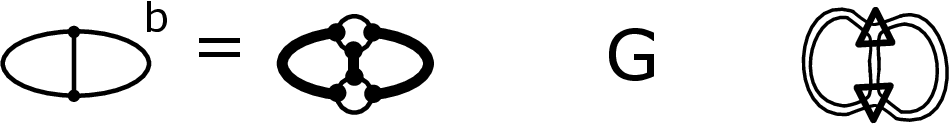}
\end{center}

\noindent Using Equation~\ref{eq:vertex-bracket} on $\Gamma_\triangle$ gives four states, which can be arranged into a hypercube:

\begin{figure}[H]
\psfragscanon
\psfrag{b}{$\flat$}\psfrag{=}{$=$}
\psfrag{G}{$\Gamma_\triangle  \ = $}
\includegraphics[scale=.17]{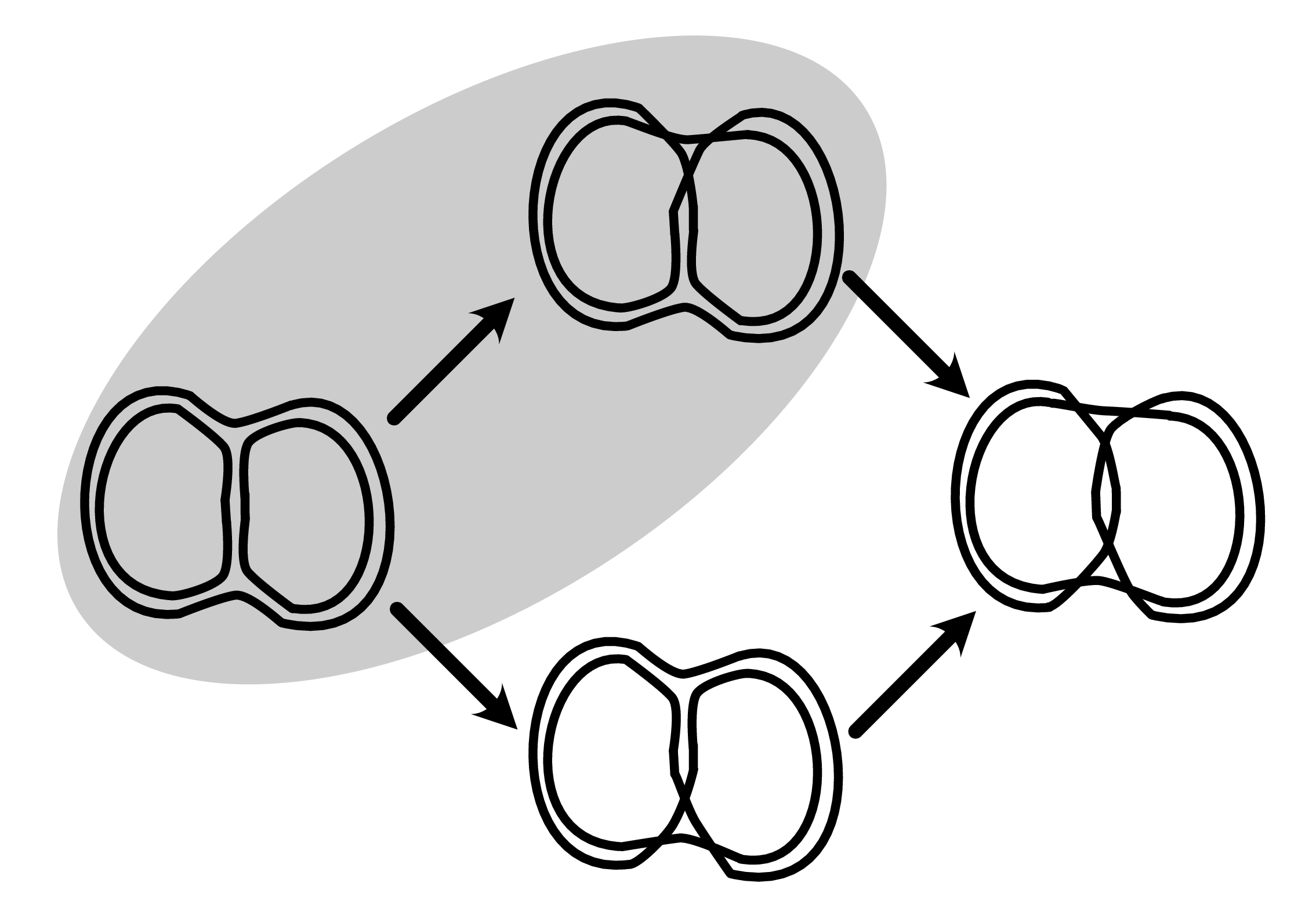}
\caption{The hypercube of states for the vertex bracket of the theta graph.}\label{fig:vertex-state-of-theta}
\end{figure}

\noindent The polynomial can be calculated from the hypercube to get, \begin{equation} \langle \Gamma_\triangle \rangle_v (q)=(q^{-1}+q)^3 - q^3(q^{-1}+q) - q^3(q^{-1}+q)+q^6(q^{-1}+q)^3,\label{eq:vertex-poly-of-theta}
\end{equation}
or $\langle \Gamma_\triangle \rangle_v(q) = q^{-3}+3q^{-1} + 3q -2q^2+2q^3 -2q^4 +3q^5+3q^7+ q^9$.

In addition to having loop value 2 again (when substituting $q=1$ in Equation~\ref{eq:vertex-loop-value}), the vertex bracket is equivalent to taking $0$-smoothings on each of the edges of the blow-up associated to the vertex to get the first term on the right-hand side of Equation~\ref{eq:vertex-bracket} and taking $1$-smoothings on each of the edges of the blow-up of the vertex to get the second term.  When a state corresponds to an edge of the original plane graph with two $1$-smoothings on it (one for each vertex), the two $1$-smoothings are equivalent to a $0$-smoothing.  For example, the all $1$-smoothing state of $\Gamma_\triangle$ in the third column of Figure~\ref{fig:vertex-state-of-theta} is equivalent to the all $0$-smoothing state in the first column. Similarly, the two states in the middle column are equivalent to each other.  In terms of the hypercube of states of $(\Gamma^\flat, E)$ of the theta graph, the states in the shaded region of Figure~\ref{fig:vertex-state-of-theta} are equivalent to the all $0$-smoothing state and all $1$-smoothing state of $(\Gamma^\flat, E)$.  In a general planar graph, the set of states of the vertex bracket is naturally a subset of the hypercube of smoothings of the blow-up of the graph.  Furthermore, the mapping from states generated by $\Gamma_\triangle$ to a subset of states of $(\Gamma^\flat,E)$ is two-to-one (after resolving all edges with two $1$-smoothings  into $0$-smoothings as described above).  

This two-to-one mapping allows us to use the cochain complex of $(\Gamma^\flat,E)$ to prove that the vertex polynomial, defined below, is an invariant of the graph:

\begin{theorem}\label{thm:vertex-poly}
Let $G(V,E)$ be a planar trivalent graph and $\Gamma$ be a plane graph of $G$.  Let $\Gamma_\triangle$ be the vertex graph associated to the all $0$-smoothing state of $(\Gamma^\flat,E)$. Define the {\em vertex polynomial of $G$} to be $\langle G \rangle_v :=\langle \Gamma_\triangle\rangle_v$. Then the vertex polynomial is invariant under the plane graph used to define it, i.e.,  the polynomial is an invariant of $G$. Furthermore, 
$$[G]= \left(\frac12\right)^{\frac12|V|}\langle G \rangle_v(1),$$
where $|V|$ is the number of vertices of $G$.
\end{theorem}

\begin{proof} Invariance is a corollary of Theorem~\ref{theorem:general_poly_invariants}, the discussion above Theorem~\ref{thm:cohomology-of-a-graph}, and the fact that states generated from $\Gamma_\triangle$ via Equation~\ref{eq:vertex-bracket} are equivalent to a subset of the states of the blow-up of the graph (whose number of circles for each state are invariant under  flip moves).  The formula follows from observing that the state sum obtained from evaluating the vertex polynomial at $q=1$ matches the equation in the middle of  page 240 of \cite{Penrose} and the discussion that leads to that equation.
\end{proof}

For an example of how the vertex polynomial counts the number of $3$-edge colorings for the theta graph $\theta$, a calculation shows that $\frac12 \left(\langle \theta \rangle_v(1)\right) = \frac12(8-2-2+8) = 6$ by substituting $q=1$ into the vertex polynomial Equation~\ref{eq:vertex-poly-of-theta}, which matches the well known fact that $[\theta]=6$.

A close inspection of the hypercube of states of the vertex bracket reveals that the vertex polynomial can be categorified into a flip-move invariant cohomology theory.  We briefly describe the setup for such a theory using the theta graph $\theta$ as a guide. For a graph $G(V,E)$ and a perfect matching graph $(\Gamma^\flat,E)$ of a plane graph of it, we get  the associated hypercube of states for the vertex bracket of $\Gamma_\Delta$. This hypercube is a set of states (vertices of the hypercube) with directed edges between states produced in the same way as other hypercubes have been defined in this paper (compare Figure~\ref{fig:vertex-state-of-theta} to Figure~\ref{fig:Theta_3_cube_of_states}). For each state of the hypercube, associate a vector space $V_\alpha$ to the state $\alpha$ by taking the tensor product $V^{\ot k}$ where $V=\BZ_2[x]/(x^2)$ and $k$ is the number of circles in the state. The gradings on $V_\alpha$ are the same as in Equation~\ref{eq:degrees-of-V} and are shifted up by the number of $\vertexbracketone$'s in the state.

Next, associate a map to each edge of the hypercube in the following manner.  Two states in the vertex bracket hypercube that are connected by an edge correspond to states in the larger hypercube of $(\Gamma^\flat,E)$ that are always three edges away from each other.  For example, the two states in the shaded region of Figure~\ref{fig:vertex-state-of-theta} correspond to column zero and column three states of Figure~\ref{fig:3-prism-hypercube-of-states}. (The perfect matching graph of $(P_3, L)$ in Figure~\ref{fig:P_3} is equivalent to the blow-up of the theta graph via a 2-flip move.) Note that the diagram to each face of the hypercube of $(\Gamma^\flat,E)$ commutes. Thus, one may choose any $3$-edge path from the first state in $(\Gamma^\flat,E)$ to the final state---any other path would give an equivalent map.  With these maps, any face of the vertex bracket hypercube would also commute (for the same reason any $3$-edge path can be chosen). Thus, one gets a cohomology theory $H_\Delta^{i,j}(\Gamma_\Delta;\BZ_2)$.  By Theorem~\ref{theorem:main-theorem-of-cohomology}, since  the cochain complex of $(\Gamma^\flat, E)$ is invariant under flip moves, this cohomology theory is also invariant under flip moves. Define $H_\Delta^{i,j}(G;\BZ_2) :=H_\Delta^{i,j}(\Gamma_\Delta;\BZ_2)$ for some plane graph $\Gamma$ of $G$.  

\begin{example} \label{ex:vertex-cohomology}The theta graph has the following set of cochain groups and maps:

\begin{center}
\parbox{1in}{
\xymatrix{
& V\{3\}\ar[dr]^{\Delta\circ\Delta\circ A} \\
V^{\ot 3}\{0\}  \ar[ur]^{A\circ m \circ m} \ar[dr]_{A\circ m \circ m} & & V^{\ot 3}\{6\}  \\
& V\{3\}\ar[ur]_{\Delta\circ\Delta\circ A} 
}}
\end{center}

 \noindent Since each of the maps in the diagram above contains an $A$ map, each is the zero map.  Thus, the cohomology $H_\Delta(\theta; \BZ_2)$ is given by the cochain groups and the graded Euler characteristic of this cohomology is $\langle \Gamma_\triangle \rangle_v (q)$ (cf. Equation~\ref{eq:vertex-poly-of-theta}).
 \end{example}

One might ask if any composition of three maps for an edge of the vertex bracket hypercube always contains an $A$ map.  If true, this would imply that the vertex bracket polynomial has the same information as the cohomology.  Interestingly, there are configurations for nonplanar trivalent graphs that give rise to maps such as $\Delta \circ m \circ \Delta$.  Unfortunately, these maps still evaluate to zero in a $\BZ_2$-coefficient theory: $\Delta \circ m \circ \Delta (1) = 2 x \ot x$. However, in a $\BZ$-coefficient theory, this map would  be nonzero and the cohomology would have more information than the vertex polynomial.  This would be guaranteed if Theorem~\ref{thm:main-theorem} could be upgraded to a $\BZ$-coefficient theory.  This is likely (cf. \cite{BKR}), but the proof of invariance of the flip moves is well beyond the scope of this paper.  We offer it as a conjecture for future research:

\begin{conjecture}
Let $(G,M)$ be a trivalent planar graph $G$ with perfect matching $M$, and let $\Gamma$ be a perfect matching graph of the pair.  Then there exists a cochain complex $(C^{i,j}(\Gamma),\partial)$ with $\BZ$-coefficients that is invariant under $2$-flip moves (compare to Theorem~\ref{theorem:main-theorem-of-cohomology}).
\end{conjecture}

\section{Conclusion}
\label{section:conclusion}

This paper is a new approach to investigating ribbon graphs in terms of topological quantum field theories.  It shows that the $2$-factor polynomial of the pair $(G,M)$ can be categorfied into a cohomology theory.  Of the recent articles on webs and foams, this theory is the closest to how Khovanov originally categorified the Jones polynomial since it directly uses a Kauffman-like bracket to define it.  This ``closeness'' is useful because one can transfer many ideas from knot theory directly over to graph theory.

The cohomology and 2-factor polynomial introduced in this paper are fundamentally different from Khovanov homology and the Jones polynomial in the following sense: unlike TQFTs in knot theory, there are cohomology theories of graphs (some of which are described or conjectured in this paper) that are not locked into using a vector space $V$ of dimension 2. In fact, one can define cohomology theories of trivalent graphs using $V=\BZ_2[x]/(x^n)$ for $n>1$ (this is future research with Ben McCarty). In these examples, one can think of $n=2$ as the case where these theories overlap with the Jones-Khovanov theory. 

Thus, graphs have a wider variety of cohomology theories associated to them than what one finds for knots and links in knot theory.  Since these cohomology theories are directly related to abstract tensor systems, Feynman diagrams, colorings of graphs, representation theory, and possibly even (delta) matroids, there is a vast, fertile ground of useful new TQFT-like theories to explore based upon the ideas presented here.


\begin{thebibliography}{99}

\bibitem{AH} K. Appel and W. Haken, {\em  Every planar map is four colorable}, Contemporary
Mathematics, {\bf 98}, With the collaboration of J. Koch, Providence, RI:  American Mathematical Society, Providence, 1989.






\bibitem{BM} S. Baldridge, A. Lowrance, and B. McCarty, {\em The 2-factor polynomial detects
even perfect matching}, The Electronic Journal of Combinatorics {\bf 27} (2020),
no. 2, P2.27, 16 pp. doi: 10.37236/9214, arXiv:1812.10346.


\bibitem{BKM} S. Baldridge, L. Kauffman, and B. McCarty, {\em  Unoriented Khovanov Homology}, New York Journal of Mathematics {\bf 28} (2022), 367-401, arXiv: 2001.04512.

\bibitem{BKR} S. Baldridge, L. Kauffman, and W. Rushworth, {\em On ribbon graphs and virtual links}, European Journal of Combinatorics {\bf 103}, June 2022, doi: 10.1016/j.ejc.2022.103520, arXiv: 2010.04238.


%
%


\bibitem{BN} D. Bar Natan, {\em On Khovanov's categorification of the Jones polynomial}, Algebraic and Geometric Topology {\bf 2} (2002), no. 16, 337-370.

\bibitem{BN2} D. Bar Natan, {\em  Lie algebras and the four color theorem}, Combinatorica {\bf 17} (1997), no. 1,  43-52.





\bibitem{EKKKN} L. Esperet, F. Kardo\v{s}, A. D. King, D.  Kr\'{a}l', and S. Norine, {\em Exponentially many perfect matchings in cubic graphs}, Advances in Mathematics {\bf 227} (2011), no. 4, 1646-1664.

\bibitem{Err} A. Errera, {\em Du colorage des cartes}, Mathesis {\bf 36} (1922), 56-60.


\bibitem{G} J. E. Greene, \emph{Lattices, graphs, and Conway mutation}, Invent. Math. {\bf 192} (2013), no. 3, 717-750.


\bibitem{H} P. J. Heawood, {\em Map-colour theorem}, Quarterly Journal of Mathematics {\bf 24} (1890), 332-338.


\bibitem{Kauffman} L. H. Kauffman, {\em A state calculus for graph coloring}, Illinois Journal of Mathematics {\bf 60} (2015), no. 1, 251-271.

\bibitem{Kauffman2} L. H. Kauffman, {\em State models and the Jones polynomial}, Topology {\bf 26} (1987), no. 3, 395-407.

\bibitem{Kauffman3} L. H. Kauffman, {\em Map coloring and the vector cross product}, J. Combin. Theory Ser. B {\bf 48} (1990), 145-154.

\bibitem{Kempe} A. B. Kempe, {\em On the Geographical Problem of the Four Colours}, American Journal of Mathematics, The Johns Hopkins University Press {\bf 2} (1879), no. 3, 193-220.

\bibitem{K1} M. Khovanov, {\em A categorification of the Jones polynomial}, Duke Math
J. {\bf 101} (1999), no. 3, 359-426.


\bibitem{KR} M. Khovanov and L-H. Robert, {\em Foam evaluation and Kronheimer--Mrowka theories}, Advances in Mathematics {\bf 376} (2021), 07433.


\bibitem{ArbitraryCoeffs}
V.~Manturov.
\newblock \uppercase{K}hovanov homology for virtual knots with arbitrary
  coefficients.
\newblock {\em J. Knot Theory Ramifications}, {\bf 16} (2007), no. 3, 345--377.





\bibitem{KM1} P.B. Kronheimer and T.S. Mrowka, {\em A deformation of instanton homology for webs}, Geometry \& Topology {\bf 23} (2019), 1491-1547.

\bibitem{KM2} P.B. Kronheimer and T.S. Mrowka, {\em Exact triangles for SO(3) instanton homology of webs}, Journal of Topology {\bf 9} (2016), no. 3, 774-796.

\bibitem{KM4} P.B. Kronheimer and T.S. Mrowka, {\em Khovanov homology is an unknot-dectector}, Publ.math.IHES {\bf 113} (2011), no. 1, 97-208. 

\bibitem{KM3} P.B. Kronheimer and T.S. Mrowka, {\em Tait colorings, and an instanton homology for webs and foams}, Journal of the European Mathematical Society {\bf 21} (2019), no. 1, 55-119.

\bibitem{LP} L. Lov\'{a}sz and M.D. Plummer, {\em Matching Theory}, Elsevier Science, Amsterdam, 1986.

\bibitem{MT} B. Mohar and C. Thomassen, \emph{Graphs on surfaces}, Johns Hopkins Studies
in the Mathematical Sciences, Johns Hopkins University Press, Baltimore, 2001.

\bibitem{Penrose} R. Penrose, ``Applications of negative dimensional tensors,'' in {\em Combinatorial Mathematics and Its Applications}, Academic Press (1971).

\bibitem{P} J. Petersen, {\em Die Theorie der regul\"{a}ren graphs}, Acta Mathematica, {\bf 15} (1891), 193-220.


\bibitem{R} J. Rasmussen, {\em Khovanov homology and the slice genus}, Inventiones Mathematicae, {\bf 182} (2010), no. 2, 419-447.

\bibitem{RW} L-H. Robert and E. Wagner, {\em A closed formula for the evaluation of $\mathfrak{sl}_N$-foams}, Quantum Topology, {\bf 11} (2020), no. 3, 411-487.

\bibitem{RSST} N. Robertson, D. P. Sanders,P. Seymour, R. Thomas, {\em The four-colour theorem}, J. Combin. Theory
Ser. B {\bf 70} (1997), 2-44.


\bibitem{Tait} P. G. Tait, {\em Note on a Theorem in Geometry of Position}, Trans. Roy. Soc. Edinburgh {\bf 29} (1880), 657-660.

\bibitem{T1} W. T. Tutte, {\em The Factorization of Linear Graphs}, J. London Math. Soc. {\bf 22} (1947), 107-111.

\bibitem{T2} W. T. Tutte, {\em On the $2$-factors of bicubic graphs}, Discrete Math. {\bf 1} (1971), no. 2, 203-208.




\bibitem{W} H. Whitney, {\em $2$-Isomorphic graphs}, Amer. J. Math. {\bf 55} (1933), no. 1-4, 245-254.





\end{thebibliography}
\end{document}